\documentclass[a4paper]{amsart}
\pdfoutput=1
\usepackage[utf8]{inputenc}
\usepackage[T1]{fontenc}
\usepackage{lmodern}
\usepackage{amssymb}
\usepackage{mathtools,enumitem,dsfont}
\usepackage{stackengine}
\usepackage{microtype}

\usepackage{tikz,tikz-cd}
\usetikzlibrary{arrows}
\tikzset{cd/.style=matrix of math nodes} %,row sep=2em,column sep=2em, text height=1.5ex, text depth=0.5ex}
\tikzset{cdar/.style=->,auto}

\tikzset{dar/.style={double,double equal sign distance,-implies}}%style for 2-arrows in triangles

\usepackage[colorlinks=true,linkcolor=blue, citecolor=blue,
urlcolor=blue, raiselinks=true,
pdftitle={Groupoid models for diagrams of groupoid correspondences},
pdfauthor={Ralf Meyer},
pdfsubject={Mathematics}
]{hyperref}

\setlist[enumerate,1]{label=\textup{(\arabic*)}}
\setlist[enumerate,2]{label=\textup{(\alph*)}}
\usepackage{xcolor}

\usepackage[lite]{amsrefs}

\usepackage{amsthm}
\newcommand{\longref}[2]{\hyperref[#2]{#1~\textup{\ref*{#2}}}}

\numberwithin{equation}{section}
\theoremstyle{plain}
\newtheorem{theorem}[equation]{Theorem}
\newtheorem{lemma}[equation]{Lemma}
\newtheorem{proposition}[equation]{Proposition}
\newtheorem{deflem}[equation]{Definition and Lemma}
\newtheorem{corollary}[equation]{Corollary}

\theoremstyle{definition}
\newtheorem{definition}[equation]{Definition}

\theoremstyle{remark}
\newtheorem{remark}[equation]{Remark}
\newtheorem{example}[equation]{Example}

\newcommand*{\alb}{\hspace{0pt}} % allow break in following word
\newcommand*{\nb}{\nobreakdash}

\newcommand*{\U}{\mathcal U}% unitaries, or a slice
\newcommand{\V}{\mathcal{V}}% another slice

\newcommand*{\lc}{\mathrm{lc}}% locally compact

\newcommand*{\Qu}{\mathsf{p}}% Orbit space projection, quotient map
\newcommand*{\s}{s} % source map
\newcommand*{\rg}{r}% range map

\newcommand*{\Cat}[1][C]{\mathcal #1}% category, usually C
\newcommand*{\GroupoidCat}[1][C]{\textup{Gr}(\mathcal #1)}% groupoid completion of C
\newcommand*{\Bisp}[1][X]{\mathcal #1}% Bibundle actor
\newcommand{\Gr}[1][G]{\mathcal #1}

\newcommand{\into}{\rightarrowtail}% kernel inclusion in an extension
\newcommand{\prto}{\twoheadrightarrow}% projection in an extension

\newcommand*{\op}{\mathrm{op}}%opposite
\newcommand*{\prop}{\mathrm{prop}}%proper
\newcommand*{\mult}{\mathrm{mult}}%multiplication map
\newcommand*{\tight}{\mathrm{tight}}% tight
\newcommand*{\Grcat}{\mathfrak{Gr}}%2-category of groupoids and bispaces
\newcommand*{\ContS}{\mathfrak S}% space of quasi-continuous functions
\newcommand*{\Cont}{\mathrm C}% space of continuous functions
\newcommand*{\Cst}{\texorpdfstring{\textup C^*}{C*}}%C*-algebra

\newcommand*{\defeq}{\mathrel{\vcentcolon=}}
\newcommand*{\congto}{\xrightarrow\sim}

\newcommand*{\id}{\mathrm{id}}% identity map
\newcommand*{\pr}{\mathrm{pr}}% coordinate projection

\newcommand{\C}{\mathbb{C}}% complex numbers
\newcommand{\N}{\mathbb{N}}

\newcommand{\Bis}{\mathcal{S}}% slices of someting
\newcommand{\IS}{\mathcal{I}}% inverse semigroup of a diagram

\newcommand{\Grcomp}{\circ}%correspondences composition
\newcommand{\pt}{\mathrm{pt}}%Point

\newcommand{\blank}{{\llcorner\!\lrcorner}}

\DeclareMathOperator{\Ad}{Ad}
\DeclareMathOperator{\Dom}{Dom}
\newcommand{\const}{\mathrm{const}}% constant diagram

\DeclarePairedDelimiter{\abs}{\lvert}{\rvert}% absolute value
\DeclarePairedDelimiterX{\braket}[2]{\langle}{\rangle}{#1\,\delimsize\vert\,\mathopen{}#2}% inner product
\DeclarePairedDelimiterX{\setgiven}[2]{\{}{\}}{#1\,{:}\,\mathopen{}#2}% set given by

\begin{document}
\title{Groupoid models for diagrams of groupoid correspondences}
\author{Ralf Meyer}
\email{rmeyer2@uni-goettingen.de}
\address{Mathematisches Institut\\
  Universität Göttingen\\
  Bunsenstraße 3--5\\
  37073 Göttingen\\
  Germany}

\keywords{étale groupoid; groupoid correspondence; bicategory;
  bilimit; groupoid action; complex of groups; self-similar group;
  higher-rank graph; self-similar graph}

\begin{abstract}
  A diagram of groupoid correspondences is a homomorphism to the
  bicategory of étale groupoid correspondences.  We study examples
  of such diagrams, including complexes of groups and self-similar
  higher-rank graphs.  We encode the diagram in a single groupoid,
  which we call its groupoid model.  The groupoid model is defined
  so that there is a natural bijection between its actions on a
  space and suitably defined actions of the diagram.  We describe
  the groupoid model in several cases, including a complex of groups
  or a self-similar group.  We show that the groupoid model is a
  bilimit in the bicategory of groupoid correspondences.
\end{abstract}

\maketitle

\setcounter{tocdepth}{1}
\tableofcontents

\section{Introduction}
\label{sec:intro}

This article is about a certain construction of étale groupoids, and
there will be no \(\Cst\)\nb-algebras in the main text.  The
motivation for this article, however, is that these groupoids have
interesting groupoid \(\Cst\)\nb-algebras.  This is why
\(\Cst\)\nb-algebras do occur in this introduction.

Many interesting \(\Cst\)\nb-algebras may be realised as
\(\Cst\)\nb-algebras of étale, locally compact groupoids.  Examples
are the \(\Cst\)\nb-algebras associated to group actions on spaces,
(higher-rank) graphs, self-similar groups, and many
\(\Cst\)\nb-algebras associated to semigroups.  A (higher-rank)
graph is interpreted in~\cite{Albandik-Meyer:Product} as a
generalised dynamical system.  Similarly, a self-similarity of a
group may be interpreted in this way, namely, as a generalised
endomorphism of a group.  The right setting for such
``endomorphisms'' is the bicategory~\(\Grcat\) of étale groupoid
correspondences introduced
in~\cite{Antunes-Ko-Meyer:Groupoid_correspondences}.  A
correspondence between two groupoids is a space with commuting
actions of the two groupoids, subject to some conditions.  It is
important not to identify isomorphic correspondences.  Instead, we
add \(2\)\nb-arrows between groupoid correspondences.

A groupoid correspondence~\(\Bisp\) from an étale groupoid~\(\Gr\)
to itself gives rise to a \(\Cst\)\nb-correspondence~\(\Cst(\Bisp)\)
from \(\Cst(\Gr)\) to itself, which then has a Cuntz--Pimsner
algebra~\(\mathcal{O}_{\Cst(\Bisp)}\).  It is shown
in~\cite{Antunes-Ko-Meyer:Groupoid_correspondences} that this
construction gives many important classes of \(\Cst\)\nb-algebras.
Katsura's topological graph \(\Cst\)\nb-algebras arise when~\(\Gr\)
is a locally compact Hausdorff space, viewed as an étale groupoid.
Nekrashevych's \(\Cst\)\nb-algebras of self-similar groups arise
when~\(\Gr\) is a discrete group, viewed as an étale groupoid.  The
\(\Cst\)\nb-algebras of self-similar graphs by Exel and Pardo arise
when~\(\Gr\) is a transformation group \(V\rtimes \Gamma\) for a
discrete group~\(\Gamma\) and a discrete \(\Gamma\)\nb-set~\(V\).
These examples suggest to interpret a groupoid correspondence on an
étale groupoid~\(\Gr\) as a self-similar topological graph.

More generally, a ``homomorphism'' from a monoid~\(M\) to the
correspondence bicategory produces a homomorphism to the
\(\Cst\)\nb-correspondence bicategory, which is equivalent to a
product system over~\(M\).  Again, this gives rise to a
Cuntz--Pimsner algebra.  If \(M=(\N^k,+)\), these diagrams are
higher-rank versions of graphs, topological graphs, self-similar
groups or self-similar graphs, and we may interpret them as
higher-rank self-similar topological graphs.  Here ``homomorphisms''
are bicategorical generalisation of functors between categories,
and their data consists of both arrows and \(2\)\nb-arrows.

In all the cases above, it is known that the resulting
Cuntz--Pimsner algebra is a groupoid \(\Cst\)\nb-algebra of an étale
groupoid.  For ``regular'' self-similar topological graphs, this
groupoid is built by Albandik in his thesis~\cite{Albandik:Thesis}.
More generally, Albandik describes this groupoid for homomorphisms
\(M\to \Grcat\) assuming that~\(M\) satisfies ``right Ore
conditions'' (see \longref{Lemma}{lem:Ore_conditions}) and that the
groupoid correspondences are ``regular'' (see
\longref{Definition}{def:proper_regular_tight}).  We are going to
define a groupoid model for any diagram of groupoid correspondences.
In the cases treated by Albandik, our groupoid model agrees with
his.  We will show in~\cite{Ko-Meyer:Groupoid_models} that such a
groupoid model always exists and that it is locally compact for
diagrams of regular, locally compact groupoid correspondences.
Another separate article will study the \(\Cst\)\nb-algebra of the
groupoid model and relate it to the Cuntz--Pimsner algebra of the
product system attached to the diagram.  This article is limited to
the definition of the groupoid model, examples, and the proof that
it gives a bicategorical bilimit.  For a better perspective,
however, we briefly sketch some results about its
\(\Cst\)\nb-algebra, to be proven later.

The regularity condition greatly simplifies the underlying groupoid
of the Cuntz--Pimsner algebra, already in the case of graphs.
Albandik's construction of the groupoid model treats only the
simpler ``regular'' case, and our groupoid model generalises that.
Therefore, the groupoid model that we are going to define only has a
chance to correspond to the Cuntz--Pimsner algebra of the product
system for regular diagrams.  I expect that irregular groupoid
correspondences may be treated as well by changing the
definition of the groupoid model.  Celso Antunes is currently working
out the case of a single groupoid correspondence in his thesis.

Without the right Ore conditions, the result fails in an interesting
way, as follows.  In \longref{Section}{sec:universal_nm}, we study
the groupoid model for a class of simple diagrams, which are related
to the \((m,n)\)-dynamical systems by Ara, Exel and
Katsura~\cite{Ara-Exel-Katsura:Dynamical_systems}.  The groupoid
model in this case is the one identified
in~\cite{Ara-Exel-Katsura:Dynamical_systems}.  The relevant
Cuntz--Pimsner algebra is already computed
in~\cite{Albandik-Meyer:Colimits}, and it is the ``Leavitt
\(\Cst\)\nb-algebra'' of~\cite{Ara-Exel-Katsura:Dynamical_systems}.
The Leavitt \(\Cst\)\nb-algebra itself is not a groupoid
\(\Cst\)\nb-algebra in any natural way, but the \(\Cst\)\nb-algebra
of our groupoid model is a quotient of it.

\smallskip

Now we leave \(\Cst\)\nb-algebras and come to the actual contents of
this article.  Our purpose is to build an étale groupoid.  Its
applications to \(\Cst\)\nb-algebras will be discussed elsewhere.
The input data for our construction is a homomorphism from a
category~\(\Cat\) to the bicategory~\(\Grcat\) of groupoid
correspondences.  We recall this bicategory in
\longref{Section}{sec:groupoid_corr}.  Actually, the bicategory
introduced in~\cite{Antunes-Ko-Meyer:Groupoid_correspondences} is
the bicategory~\(\Grcat_\lc\) of \emph{locally compact} groupoid
correspondences.  We allow a larger bicategory where some
assumptions on objects, arrows and \(2\)\nb-arrows are dropped.
This generalisation is not terribly important, but it is useful for
the existence results in~\cite{Ko-Meyer:Groupoid_models}.  Namely, in the
irregular case, the groupoid model --~as it is defined here~-- may
fail to be locally compact.

In \longref{Section}{sec:diagrams}, we study homomorphisms from a
category~\(\Cat\) to the bicategory~\(\Grcat\) of groupoid
correspondences.  We also call such homomorphisms diagrams in
analogy to their use in category theory.  We first specialise the
general bicategorical concepts to the case at hand and then look at
a number of examples of such diagrams.

In these examples, we usually show that diagrams of a certain form
are ``equivalent'' to something simpler.  Here ``equivalent'' means
that the bicategory of such diagrams is equivalent as a bicategory
to a simpler bicategory.  Therefore, to talk precisely about such
equivalences, we need a bicategory that has homomorphisms as
objects.  Its arrows are called ``strong transformations'' and its
\(2\)\nb-arrows are called ``modifications''.  The first example we
treat are tight diagrams of spaces, that is, the groupoids in the
diagram are spaces viewed as groupoids and the arrows are ``tight''
groupoid correspondences between them.  A tight groupoid
correspondence between two spaces is just a local homeomorphisms,
and so we get that a tight diagram of groupoid correspondences
between spaces is equivalent to an action of a category on spaces by
local homeomorphisms.  The latter even form a category, that is,
there are only identity \(2\)\nb-arrows.  Our second example is
more interesting: tight diagrams of groups are closely related to
complexes of groups, and strong transformations and modifications
specialise to morphisms between complexes of groups and homotopies
between these morphisms.  Thus complexes of groups become examples
of diagrams of groupoid correspondences.

We also look at diagrams where the category~\(\Cat\) is special.
When~\(\Cat\) is a group, then we get a groupoid graded by a group.
When~\(\Cat\) is a free monoid, then we simply get a groupoid with a
number of groupoid correspondences on it, without any relations.
Something similar happens when~\(\Cat\) is the path category of a
directed graph.  We also study the case when~\(\Cat\) is the free
Abelian monoid \((\N^k,+)\).  Then diagrams are related to
higher-rank graphs.  We describe such a diagram by groupoid
correspondences \(\Bisp_j\), \(j=1,\dotsc,k\), together with
isomorphisms
\(\Bisp_i \Grcomp \Bisp_j \cong \Bisp_j \Grcomp \Bisp_i\) for
\(i<j\), such that certain hexagon diagrams for \(i<j<k\) commute.

In \longref{Section}{sec:actions}, we define the groupoid model of a
diagram of groupoid correspondences.  The idea here is to specify an
étale groupoid by specifying its actions on spaces.  This category,
together with the forgetful functor to spaces, determines the
groupoid uniquely up to isomorphism.  We first define what it means
for a diagram of groupoid correspondences to act on a space.  Then
we require a natural bijection between these diagram actions and
actions of the groupoid model.  Besides the definition, we look at
two easy and one nontrivial example.  The two easy examples are
diagrams where~\(\Cat\) has only identity arrows or where~\(\Cat\)
is a group.  In the first case, the groupoid model is a disjoint
union of groupoids.  In the second case, it is the graded groupoid
that describes the homomorphism to~\(\Grcat\).  The third example is
the equaliser diagram for two groupoid correspondences on the
one-point one-arrow groupoid.  In this case, the actions of the
diagram are identified with the \((m,n)\)-dynamical systems of
Ara--Exel--Katsura~\cite{Ara-Exel-Katsura:Dynamical_systems}.  We
notice that their study of these dynamical systems also produces a
groupoid model for this diagram.

One important ingredient in the groupoid model is the universal
action of the diagram.  This produces the object space of the
groupoid model.  A second ingredient is a certain inverse semigroup
\(\IS(F)\) built from the slices of the diagram~\(F\).  This inverse
semigroup is defined in \longref{Section}{sec:encoding}.  The main
result is that the groupoid model is the transformation groupoid for
the canonical action of \(\IS(F)\) on the underlying space of the
universal action.

The universal action is very easy to describe for tight diagrams.
This is the point of \longref{Section}{sec:limit_tight}, and it is
the reason why we are particularly interested in examples of tight
diagrams in \longref{Section}{sec:diagrams}.  In
\longref{Section}{sec:group_complex} we examine the groupoid model
for a tight diagrams of groupoid correspondences between groups.  We
have already related such diagrams to complexes of groups, and we
find that the fundamental group of such a complex of groups is
related to our groupoid model, but in a nontrivial way.  Our
groupoid model is usually a groupoid with many objects.  What we
find is that the fundamental group of a complex of groups is Morita
equivalent as a groupoid to the groupoid model of a certain tight
diagram of groupoid correspondences between groups.  This diagram,
however, is not the original diagram, we first have to augment the
complex of groups in a certain natural way by the trivial group.

Beyond the tight diagrams, we also describe the groupoid model
rather explicitly for those diagrams \(\Cat \to \Grcat\) where the
category~\(\Cat\) satisfies the ``right Ore conditions''.  We study
such diagrams in \longref{Section}{sec:Ore_shape}.  We describe the
universal action as a projective limit, and then we also describe
the groupoid model in a similar way.  The resulting alternative
construction is due to Albandik~\cite{Albandik:Thesis}.  He uses
this construction to show that the \(\Cst\)\nb-algebra of the
groupoid model is isomorphic to a Cuntz--Pimsner algebra, provided
the diagram is proper.  (Beware, however, that the Cuntz--Pimsner
algebra that comes up in~\cite{Albandik:Thesis} is the absolute one,
which requires the Cuntz--Pimsner covariance condition on the whole
\(\Cst\)\nb-algebra and not just on Katsura's ideal.)

In \longref{Section}{sec:examples_Ore}, we examine the two different
constructions of the groupoid model for diagrams of Ore shape in a
few special cases.  For diagrams of spaces, we get the same results
as in~\cite{Albandik-Meyer:Product}.  For self-similar groups and
self-similar graphs, we get the groupoids already associated to
these situations by
Nekrashevych~\cite{Nekrashevych:Cstar_selfsimilar} and
Exel--Pardo~\cite{Exel-Pardo:Self-similar}.  We also compare the
inverse semigroup used in our construction to the inverse semigroups
used by Nekrashevych and by Exel and Pardo.

Finally, in \longref{Section}{sec:models_vs_limits}, we show that a
groupoid model of a diagram of groupoid correspondences is a bilimit
for the diagram in~\(\Grcat\).  This concept generalises the limit
of a diagram in a category.  The bilimit is only unique up to
equivalence, which in our case means Morita equivalence of
groupoids.  In contrast, the groupoid model is unique up to
isomorphism.  In the thesis of Albandik~\cite{Albandik:Thesis}, the
concept of groupoid model is still missing.  He constructs a
groupoid under some extra assumptions and shows that it is a bilimit
of the given diagram and that its groupoid \(\Cst\)\nb-algebra is
isomorphic to a Cuntz--Pismner algebra attached to the diagram.
That the groupoid model is a bilimit is important because it allows
to apply general results from bicategory theory.  In particular, it
clarifies in what sense the construction of groupoid models is
functorial.

\section{The bicategory of groupoid correspondences}
\label{sec:groupoid_corr}

In this section, we define the class of groupoids and groupoid
correspondences that we work with, and we explain how they form a
bicategory.  Our definitions are close to those
in~\cite{Antunes-Ko-Meyer:Groupoid_correspondences}.  We allow more
objects, arrows and \(2\)\nb-arrows, however, because we drop some
assumptions that are needed
in~\cite{Antunes-Ko-Meyer:Groupoid_correspondences} in order to pass
from groupoids to \(\Cst\)\nb-algebras.  These assumptions are
mostly harmless for the purposes of this article.  It may, however,
happen that the groupoid model of a diagram of groupoid
correspondences fails to be locally compact.  While I do not think
that the groupoid model as defined here is very useful in such
cases, it is desirable to at least talk about this phenomenon.  This
seems justification enough to leave out all extra assumptions that
are not needed for the bicategory structure on groupoid
correspondences.

A topological groupoid~\(\Gr\) consists of topological spaces
\(\Gr\) and \(\Gr^0\subseteq \Gr\) of arrows and objects, continuous range and
source maps \(\rg,\s\colon \Gr \rightrightarrows \Gr^0\), and a
continuous multiplication map
\(\Gr\times_{\s,\Gr^0,\rg} \Gr \to \Gr\), \((g,h)\mapsto g\cdot h\).
We require that each object has a unit arrow and each arrow has an
inverse arrow, such that the usual algebraic identities are
satisfied and the unit map and the inversion are continuous.  We
tacitly assume all groupoids to be \emph{étale}, that is, \(\s\)
and~\(\rg\) are local homeomorphisms.  In particular, \(\s\)
and~\(\rg\) are open.

In an étale groupoid, each arrow \(g\in\Gr\) has an open
neighbourhood \(U\subseteq \Gr\) such that \(\s|_U\) and~\(\rg|_U\)
are homeomorphisms onto open subsets of~\(\Gr^0\).  Such open
subsets are called \emph{slices} (see
\cites{Exel:Inverse_combinatorial,
  Antunes-Ko-Meyer:Groupoid_correspondences}).  Any open subset of a
slice is again a slice.  The units in~\(\Gr\) form a slice, which we
denote by~\(\Gr^0\) and call the \emph{unit slice}.  If~\(U\) is a
slice, so is
\[
  U^* \defeq \setgiven{g^{-1}}{g\in U}.
\]
If \(U\)
and~\(V\) are slices, so is
\[
  U\cdot V\defeq \setgiven{g\cdot h}{g\in U,\ h\in V}.
\]
With this structure, the slices of~\(\Gr\) form an inverse semigroup
(see~\cite{Exel:Inverse_combinatorial}).

\begin{definition}
  An (étale) groupoid~\(\Gr\) is called \emph{locally compact} if
  its object space~\(\Gr^0\) is Hausdorff and locally compact.
\end{definition}

Unlike in~\cite{Antunes-Ko-Meyer:Groupoid_correspondences}, we
usually allow groupoids that are not locally compact.  Let~\(\Gr\)
be a locally compact groupoid.  Each slice \(U\subseteq \Gr\) is
Hausdorff and locally compact because it is homeomorphic to an open
subset in~\(\Gr^0\).  Thus the arrow space~\(\Gr\) is locally
compact and locally Hausdorff.  It may, however, fail to be
Hausdorff.

\begin{definition}[\cite{Antunes-Ko-Meyer:Groupoid_correspondences}*{Definition~2.3}]
  Let~\(\Gr\) be a groupoid.  A (right) \emph{\(\Gr\)\nb-space} is a
  topological space~\(\Bisp\) with a continuous map
  \(\s\colon \Bisp\to \Gr^0\), the \emph{anchor map}, and a
  continuous map
  \[
    \mult\colon \Bisp\times_{\s,\Gr^0,\rg} \Gr\to\Bisp,\qquad
    \Bisp\times_{\s,\Gr^0,\rg}\Gr\defeq
    \setgiven{(x,g)\in \Bisp\times \Gr}{\s(x)=\rg(g)},
  \]
  which we denote multiplicatively as \((x,y)\mapsto x\cdot y\),
  such that
  \begin{enumerate}
  \item \(\s(x\cdot g)=\s(g)\) for \(x\in\Bisp\), \(g\in\Gr\) with
    \(\s(x)=\rg(g)\);
  \item \((x\cdot g_1)\cdot g_2=x\cdot (g_1\cdot g_2)\) for
    \(x\in \Bisp\), \(g_1, g_2\in \Gr\) with \(\s(x)=\rg(g_1)\),
    \(\s(g_1)=\rg(g_2)\);
  \item \(x\cdot \s(x)=x\) for all \(x\in \Bisp\).
  \end{enumerate}
\end{definition}

Left \(\Gr\)\nb-spaces are defined similarly.  We do \emph{not}
require the anchor map to be open.  We always write
\(\s\colon \Bisp\to \Gr^0\) for the anchor map in a right action and
\(\rg\colon \Bisp\to \Gr^0\) for the anchor map in a left action.
We sometimes write \(\s_{\Bisp}\) or~\(\rg_{\Bisp}\) if~\(\Bisp\) is
not clear enough from the context.

\begin{definition}[\cite{Antunes-Ko-Meyer:Groupoid_correspondences}*{Definition~2.4}]
  Let~\(\Bisp\) be a \(\Gr\)\nb-space.  The \emph{orbit
    space}~\(\Bisp/\Gr\) is the quotient~\(\Bisp/{\sim_{\Gr}}\) with
  the quotient topology, where \(x\sim_{\Gr} y\) if there is an
  element \(g\in \Gr\) with \(\s(x)=\rg(g)\) and \(x\cdot g=y\).
\end{definition}

\begin{definition}[\cite{Antunes-Ko-Meyer:Groupoid_correspondences}*{Definition~2.7}]
  \label{def:basic_action}
  A right \(\Gr\)\nb-space~\(\Bisp\) is \emph{basic} if the
  following map is a homeomorphism onto its image with the subspace
  topology from \(\Bisp\times\Bisp\):
  \begin{equation}
    \label{eq:action_map}
    \Bisp\times_{\s,\Gr^0,\rg} \Gr \to \Bisp\times\Bisp,\qquad
    (x,g)\mapsto (x\cdot g,x).
  \end{equation}
\end{definition}

\begin{lemma}[\cite{Antunes-Ko-Meyer:Groupoid_correspondences}*{Lemma~2.12}]
  \label{lem:basic_orbit_lh}
  The orbit space projection \(\Qu\colon \Bisp\to\Bisp/\Gr\)
  for a basic \(\Gr\)\nb-action is a local homeomorphism.
\end{lemma}

\begin{definition}[compare
  \cite{Antunes-Ko-Meyer:Groupoid_correspondences}*{Definition~3.1}]
  \label{def:Bibundles}
  Let \(\Gr[H]\) and~\(\Gr\) be (étale) groupoids.  An
  \textup{(}étale\textup{)} \emph{groupoid correspondence}
  from~\(\Gr\) to~\(\Gr[H]\), denoted
  \(\Bisp\colon \Gr[H]\leftarrow \Gr\), is a space~\(\Bisp\) with
  commuting actions of~\(\Gr[H]\) on the left and~\(\Gr\) on the
  right, such that the right anchor map \(\s\colon \Bisp\to \Gr^0\)
  is a local homeomorphism and the right \(\Gr\)\nb-action is basic.

  Having commuting actions of \(\Gr[H]\) and~\(\Gr\) means that
  \(\s(h\cdot x)=\s(x)\), \(\rg(x\cdot g) = \rg(x)\), and
  \((h\cdot x) \cdot g = h\cdot (x\cdot g)\) for all \(g\in \Gr\),
  \(x\in\Bisp\), \(h\in \Gr[H]\) with \(\s(h)=\rg(x)\) and
  \(\s(x)=\rg(g)\), where \(\s\colon \Bisp\to\Gr^0\) and
  \(\rg\colon \Bisp\to\Gr[H]^0\) are the anchor maps.
\end{definition}

\begin{definition}[\cite{Antunes-Ko-Meyer:Groupoid_correspondences}*{Definition~3.3}]
  \label{def:proper_regular_tight}
  A correspondence \(\Bisp\colon \Gr[H]\leftarrow \Gr\) is
  \emph{proper} if the map \(\rg_*\colon \Bisp/\Gr\to \Gr[H]^0\)
  induced by~\(\rg\) is proper.  It is \emph{regular} if
  \(\rg_*\colon \Bisp/\Gr\to \Gr[H]^0\) is proper and surjective.
  It is \emph{tight} if~\(\rg_*\) is a homeomorphism.
\end{definition}

\begin{definition}
  \label{def:lc_correspondence}
  Let \(\Gr[H]\) and~\(\Gr\) be locally compact groupoids.  A
  \emph{locally compact groupoid correspondence}
  \(\Bisp\colon \Gr[H]\leftarrow \Gr\) is a groupoid
  correspondence~\(\Bisp\) with the extra property
  that~\(\Bisp/\Gr\) is Hausdorff.
\end{definition}

The ``groupoids'' and ``groupoid correspondences'' as defined
in~\cite{Antunes-Ko-Meyer:Groupoid_correspondences} are the
``locally compact groupoids'' and the ``locally compact groupoid
correspondences'' in the notation in this article.  As shown in
\cite{Antunes-Ko-Meyer:Groupoid_correspondences}*{Proposition~2.19},
an action of a locally compact groupoid is basic with Hausdorff
object space as in \longref{Definition}{def:lc_correspondence} if
and only if it is ``free'' and ``proper''.

The following example shows that the homomorphism from the groupoid
correspondence bicategory to the \(\Cst\)\nb-correspondence
bicategory breaks down if we allow \(\Bisp/\Gr\) to be
non-Hausdorff.

\begin{example}
  \label{exa:groupoid_corr_no_Cstar_corr}
  Let \(\Gr[H]\) be the one-point space, viewed as a groupoid, and
  let~\(\Gr\) be the space~\([0,1]\).  Let~\(\Bisp\) be the locally
  Hausdorff space obtained by identifying two copies of~\([0,1]\)
  along~\([0,1)\), so that only the point~\(1\) becomes doubled.
  This is a groupoid correspondence \(\Gr[H]\leftarrow\Gr\) because
  any action of the groupoid~\(\Gr\) is basic.  This correspondence
  fails to be locally compact, however.  The space
  \(\Bisp=\Bisp/\Gr\) is not Hausdorff and so the \(\Gr\)\nb-action
  on~\(\Bisp\) is not proper.  By the way, this correspondence is
  proper, that is, the map \(\rg\colon \Bisp \to \Gr^0=\pt\) is
  proper.  We now use the notation
  of~\cite{Antunes-Ko-Meyer:Groupoid_correspondences} to explain
  why~\(\Bisp\) does not induce a \(\Cst\)\nb-correspondence between
  \(\Gr[H]\) and~\(\Gr\).
  
  The characteristic functions of the two copies of \([0,1]\)
  in~\(\Bisp\) belong to~\(\ContS(\Bisp)\).  Their inner product is
  the characteristic function of~\([0,1)\), which is not in
  \(\ContS([0,1]) = \Cont([0,1])\).  Hence the formula for the inner
  product on \(\ContS([0,1])\) in
  \cite{Antunes-Ko-Meyer:Groupoid_correspondences}*{(7.6)} breaks
  down.  In fact, there is no \(\Cst\)\nb-correspondence that could
  adequately describe~\(\Bisp\).  The left action of
  \(\Cst(\Gr) \cong \C\) in a correspondence is unique because~\(1\)
  must act identically.  Hence such a \(\Cst\)\nb-correspondence
  would be a Hilbert module over~\(\Cont([0,1])\).  This is
  equivalent to a
  continuous field of Hilbert spaces over~\([0,1]\).  Its fibres
  should have dimension~\(1\) over \(t\in[0,1)\) and~\(2\) at
  \(t=1\) because the fibres of~\(\Bisp\) have that many points.
  For a continuous field of Hilbert spaces, however, the dimension
  function is always semicontinuous in the other direction, that is,
  the set of points with fibre dimension~\(2\) is open.
\end{example}

\begin{remark}
  \label{rem:tight_lc}
  Any tight correspondence \(\Bisp\colon \Gr[H]^0 \leftarrow \Gr\)
  between locally compact groupoids \(\Gr[H]\) and~\(\Gr\) is
  locally compact because then \(\Bisp/\Gr \cong \Gr[H]^0\) is
  Hausdorff.
\end{remark}

\begin{definition}[\cite{Antunes-Ko-Meyer:Groupoid_correspondences}*{Definition~7.2}]
  \label{def:correspondence_slices}
  Let \(\Bisp\colon \Gr[H]\leftarrow \Gr\) be a groupoid
  correspondence.  A \emph{slice} of~\(\Bisp\) is an open subset
  \(\U\subseteq \Bisp\) such that both \(\s\colon \Bisp \to \Gr^0\)
  and the orbit space projection \(\Qu\colon \Bisp\prto \Bisp/\Gr\)
  are injective on~\(\U\).  Let \(\Bis(\Bisp)\) be the set of all
  slices of~\(\Bisp\).
\end{definition}

\begin{remark}
  \label{rem:slice_tight}
  Let~\(\Bisp\) be a tight groupoid correspondence.  Then an open
  subset \(U\subseteq \Bisp\) is a slice if and only if \(\s|_U\)
  and~\(\rg|_U\) are injective.  This is because
  \(\rg = \rg_*\circ \Qu\) for the homeomorphism
  \(\rg_*\colon \Bisp/\Gr \congto \Gr[H]^0\).
\end{remark}

Let \(\Bisp\colon \Gr[H]\leftarrow \Gr\) be a groupoid
correspondence.  Then~\(\s\) is a local homeomorphism by assumption
and~\(\Qu\) is one by \longref{Lemma}{lem:basic_orbit_lh}.  This
implies that any element of~\(\Bisp\) has a slice as an open
neighbourhood.  Even more, since any open subset of a slice is again
a slice, the slices form a basis for the topology of~\(\Bisp\).

\begin{lemma}
  \label{Hausdorff_orbit}
  If~\(\Bisp\) is a locally compact groupoid correspondence, then
  the orbit space~\(\Bisp/\Gr\) is Hausdorff and locally compact.
\end{lemma}

\begin{proof}
  By assumption, \(\Bisp/\Gr\) is Hausdorff and~\(\Gr^0\) is locally
  compact.  If \(U\subseteq \Bisp\) is a slice, then \(\s|_U\)
  and~\(\Qu|_U\) are homeomorphisms \(U\congto \s(U)\) and
  \(U\congto \Qu(U)\).  Then \(\s(U) \cong \Qu(U)\), and this
  implies that~\(\Bisp/\Gr\) is locally compact because~\(\Gr^0\)
  is.
\end{proof}

The following example shows that we should not require~\(\Bisp\)
itself to be Hausdorff.

\begin{example}
  \label{exa:identity_corr}
  The \emph{identity correspondence}~\(1_{\Gr}\) for a
  groupoid~\(\Gr\) is the arrow space of~\(\Gr\) with the obvious
  left and right actions of~\(\Gr\) by multiplication.  This need
  not be Hausdorff if~\(\Gr\) is a locally compact groupoid.
\end{example}

\begin{deflem}[\cite{Antunes-Ko-Meyer:Groupoid_correspondences}*{Definition
    and Lemma~3.4}]
  \label{def:innerprodofcorr}
  Let~\(\Bisp\) be a space with a basic right \(\Gr\)\nb-action.
  Let \(\Qu\colon \Bisp\to \Bisp/\Gr\) be the orbit space
  projection.  The image of the map~\eqref{eq:action_map} is the
  subset
  \(\Bisp \times_{\Bisp/\Gr} \Bisp = \Bisp
  \times_{\Qu,\Bisp/\Gr,\Qu} \Bisp\) of all
  \((x_1,x_2)\in\Bisp\times\Bisp\) with \(\Qu(x_1)=\Qu(x_2)\).  The
  inverse to the map in~\eqref{eq:action_map} induces a continuous
  map
  \begin{equation}
    \label{eq:innprod_correspondence}
    \Bisp \times_{\Bisp/\Gr} \Bisp \congto
    \Bisp \times_{\s,\Gr^0,\rg} \Gr \xrightarrow{\pr_2} \Gr,\qquad
    (x_1,x_2)\mapsto \braket{x_2}{x_1}.
  \end{equation}
  That is, \(\braket{x_1}{x_2}\) is defined for \(x_1,x_2\in\Bisp\)
  with \(\Qu(x_1)=\Qu(x_2)\) in~\(\Bisp/\Gr\), and it is the unique
  \(g\in\Gr\) with \(\s(x_1) = \rg(g)\) and \(x_2 = x_1 g\).
  Conversely, if \(g\in\Gr\) with \(x_1= x_2g\) for
  \(x_1,x_2\in\Bisp\) with \(\Qu(x_1)=\Qu(x_2)\) is unique and
  depends continuously on
  \((x_1,x_2) \in\Bisp\times_{\Bisp/\Gr}\Bisp\), then the right
  \(\Gr\)\nb-action on~\(\Bisp\) is basic.
\end{deflem}

\begin{proposition}[\cite{Antunes-Ko-Meyer:Groupoid_correspondences}*{Proposition~3.6}]
  \label{pro:innprod}
  Let \(\Bisp\colon \Gr[H]\leftarrow\Gr\) be a groupoid
  correspondence.  The map in~\eqref{eq:innprod_correspondence} is a
  local homeomorphism.  It has the following properties:
  \begin{enumerate}
  \item \label{en:innprod_1}%
    \(\rg(\braket{x_1}{x_2}) = \s(x_1)\),
    \(\s(\braket{x_1}{x_2}) = \s(x_2)\), and
    \(x_2 = x_1\cdot \braket{x_1}{x_2}\) for all \(x_1,x_2\in\Bisp\)
    with \(\Qu(x_1)= \Qu(x_2)\);
  \item \label{en:innprod_2}%
    \(\braket{x}{x} = \s(x)\) for all \(x\in\Bisp\);
  \item \label{en:innprod_3}%
    \(\braket{x_1}{x_2} = \braket{x_2}{x_1}^{-1}\) for all
    \(x_1,x_2\in\Bisp\) with \(\Qu(x_1)= \Qu(x_2)\);
  \item \label{en:innprod_4}%
    \(\braket{h x_1 g_1}{h x_2 g_2} = g_1^{-1}
    \braket{x_1}{x_2} g_2\) for all \(h\in\Gr[H]\),
    \(x_1,x_2\in\Bisp\), \(g_1,g_2\in\Gr\) with
    \(\s(h) = \rg(x_1) = \rg(x_2)\),
    \(\s(x_1)=\rg(g_1)\), \(\s(x_2)=\rg(g_2)\),
    \(\Qu(x_1)=\Qu(x_2)\).
  \end{enumerate}
\end{proposition}

We turn to the composition of groupoid correspondences.  Let
\(\Gr[H]\), \(\Gr\) and~\(\Gr[K]\) be étale groupoids.  Let
\(\Bisp\colon \Gr[H]\leftarrow \Gr\) and
\(\Bisp[Y]\colon \Gr\leftarrow \Gr[K]\) be groupoid correspondences.
Let
\[
  \Bisp\times_{\Gr^0}\Bisp[Y] =
  \Bisp\times_{\s,\Gr^0,\rg}\Bisp[Y]
  \defeq \setgiven{(x,y)\in \Bisp\times \Bisp[Y]}{\s(x)=\rg(y)}.
\]
We let~\(\Gr\) act on \(\Bisp\times_{\Gr^0}\Bisp[Y]\) by the
\emph{diagonal action}
\[
  g\cdot (x,y)\defeq (x\cdot g^{-1},g\cdot y)
\]
for \(g\in \Gr\), \(x\in\Bisp\), and \(y\in\Bisp[Y]\) with
\(\s(g) = \s(x) = \rg(y)\).  Let~\(\Bisp\Grcomp_{\Gr}\Bisp[Y]\) be
the orbit space of this action, and let \([x,y]\) denote the image
of \((x,y)\in \Bisp\times_{\Gr^0}\Bisp[Y]\) in
\(\Bisp\Grcomp_{\Gr}\Bisp[Y]\).  The maps \(\rg(x,y)\defeq \rg(x)\)
and \(\s(x,y)\defeq \s(y)\) on \(\Bisp\times_{\Gr^0}\Bisp[Y]\)
induce maps \(\rg\colon \Bisp\Grcomp_{\Gr}\Bisp[Y]\to \Gr[H]^0\) and
\(\s\colon \Bisp\Grcomp_{\Gr}\Bisp[Y]\to \Gr[K]^0\).  These are the
anchor maps for commuting actions of~\(\Gr[H]\) on the left
and~\(\Gr[K]\) on the right, which we define as:
\[
  h\cdot[x,y]\defeq [h \cdot x,y],\qquad
  [x,y]\cdot k\defeq [x,y\cdot k]
\]
for all \(h\in \Gr[H]\), \(x\in\Bisp\), \(y\in\Bisp[Y]\),
\(k\in\Gr[K]\) with \(\s(h) = \rg(x)\), \(\s(x)=\rg(y)\), and
\(\s(y) = \rg(k)\).

\begin{proposition}[compare
\cite{Antunes-Ko-Meyer:Groupoid_correspondences}*{Proposition~5.7}]
  \label{pro:groupoid_correspondences_composition}
  The actions of \(\Gr[H]\) and~\(\Gr[K]\) on
  \(\Bisp\Grcomp_{\Gr}\Bisp[Y]\) are well defined and turn this into
  a groupoid correspondence \(\Gr[H]\leftarrow \Gr[K]\).  If both
  correspondences \(\Bisp\) and~\(\Bisp[Y]\) are proper, tight or
  locally compact, then so is \(\Bisp\Grcomp_{\Gr}\Bisp[Y]\).
\end{proposition}

\begin{proof}
  When all groupoids and groupoid correspondences are locally
  compact, the assertions are those in
  \cite{Antunes-Ko-Meyer:Groupoid_correspondences}*{Proposition~5.7}.
  Its proof also shows the assertions if we do not assume local
  compactness.
\end{proof}

If \(\Gr,\Gr[H]\) are groupoids and
\(\Bisp,\Bisp[Y]\colon \Gr[H] \leftleftarrows \Gr\) are groupoid
correspondences, then we take continuous
\(\Gr[H]\)-\(\Gr\)-equivariant maps \(\Bisp\to\Bisp[Y]\) as
\(2\)\nb-arrows \(\Bisp\Rightarrow\Bisp[Y]\).  These are
automatically local homeomorphisms by
\cite{Antunes-Ko-Meyer:Groupoid_correspondences}*{Lemma~6.1}.
In~\cite{Antunes-Ko-Meyer:Groupoid_correspondences}, the
\(2\)\nb-arrows are assumed to be injective as maps.  This is
important for the passage to groupoid \(\Cst\)\nb-algebras.
Therefore, we assume this in the locally compact case.  But there is
no need to assume this when working with general topological
groupoids.

Let \(\Bisp\colon \Gr[H]\leftarrow \Gr\) be a groupoid correspondence.
The canonical \(\Gr[H],\Gr\)\nb-equivariant homeomorphisms
\begin{alignat}{2}
  \label{eq:unitor_left}
  1_{\Gr[H]}\Grcomp_{\Gr[H]}\Bisp&\to \Bisp,&\qquad
  [h,x]&\mapsto h\cdot x,\\
  \label{eq:unitor_right}
  \Bisp\Grcomp_{\Gr} 1_{\Gr}&\to \Bisp,&\qquad
  [x,g]&\mapsto x\cdot g,
\end{alignat}
in \cite{Antunes-Ko-Meyer:Groupoid_correspondences}*{Lemma~6.3}
witness that~\(1_{\Gr}\) behaves like a unit arrow on~\(\Gr\).  With
the canonical homeomorphisms in
\cite{Antunes-Ko-Meyer:Groupoid_correspondences}*{Lemma~6.4} as
associators, we get the bicategory~\(\Grcat\) of groupoid
correspondences.  The tight and the proper groupoid correspondences
form subbicategories
\(\Grcat_\tight\subseteq\Grcat_\prop\subseteq\Grcat\), respectively.
We define \(\Grcat_\lc\subseteq \Grcat\) to be the subcategory with
locally compact groupoids as objects, locally compact groupoid
correspondences between them as arrows, and injective equivariant
maps as \(2\)\nb-arrows.  In addition, there are subbicategories
\(\Grcat_{\lc,\tight} \subseteq \Grcat_{\lc,\prop} \subseteq
\Grcat_\lc\) of tight or proper locally compact groupoid
correspondences.

The following theorem is not yet proven
in~\cite{Antunes-Ko-Meyer:Groupoid_correspondences}.

\begin{theorem}
  \label{the:groupoid_equivalence}
  The equivalences in~\(\Grcat\) are exactly the Morita equivalences
  of groupoids.  That is, a groupoid correspondence \(\Bisp\colon
  \Gr[H]\leftarrow \Gr\) is an equivalence if and only if both actions
  are basic and both anchor maps induce homeomorphisms \(\Bisp/\Gr
  \cong \Gr[H]^0\) and \(\Gr[H]{\backslash} \Bisp \cong \Gr^0\).
\end{theorem}

\begin{proof}
  Let~\(\Bisp[Y]\) be a groupoid correspondence that is inverse
  to~\(\Bisp\) up to isomorphism.  That is, there are equivariant
  homeomorphisms
  \(\varphi\colon \Bisp \Grcomp \Bisp[Y] \congto \Gr[H]\) and
  \(\psi\colon \Bisp[Y] \Grcomp \Bisp \congto \Gr\).  Then the maps
  \(\rg_{\Bisp}\colon \Bisp \to \Gr[H]^0\),
  \(\s_{\Bisp[Y]}\colon \Bisp[Y] \to \Gr[H]^0\),
  \(\rg_{\Bisp[Y]}\colon \Bisp[Y] \to \Gr^0\) and
  \(\s_{\Bisp}\colon \Bisp \to \Gr^0\) are surjective because the
  range and source maps on \(\Gr[H]\) and~\(\Gr\) are surjective.
  Next, \cite{Antunes-Ko-Meyer:Groupoid_correspondences}*{Lemma~5.4}
  provides a homeomorphism
  \(\Bisp[Y] \Grcomp (\Bisp/\Gr) \cong (\Bisp[Y] \Grcomp \Bisp)/\Gr
  \cong \Gr/\Gr \cong \Gr^0\).  The coordinate projection
  \(\Bisp[Y] \times_{\s,\Gr[H]^0,\rg} \Bisp \to \Bisp[Y]\) induces a
  canonical map \(\Bisp[Y] \Grcomp \Bisp \to \Bisp[Y]/\Gr[H]\).  It
  is surjective because~\(\rg_{\Bisp}\) is surjective.  It descends
  to a surjective continuous map~\(f\) on the orbit spaces from
  \((\Bisp[Y] \Grcomp \Bisp)/\Gr \cong \Gr/\Gr \cong \Gr^0\) to
  \(\Bisp[Y]/\Gr[H]\).  The composite map
  \(\rg_{\Bisp[Y],*}\circ f\colon \Gr^0 \to \Bisp[Y]/\Gr[H] \to
  \Gr^0\) is the identity map on~\(\Gr^0\) because
  \(\rg[y,x] \defeq \rg_{\Bisp[Y]}(y)\) for all
  \([y,x] \in \Bisp[Y] \Grcomp \Bisp\).  Since~\(f\) is surjective
  and \(f \rg_{\Bisp[Y],*} f = f\), it follows that
  \(\rg_{\Bisp[Y],*}\colon \Bisp[Y]/\Gr[H] \to \Gr^0\) and~\(f\) are
  homeomorphisms inverse to each other.  The same argument shows
  that \(\rg_{\Bisp,*}\colon \Bisp/\Gr \to \Gr[H]^0\) is a
  homeomorphism.

  Next we construct a homeomorphism
  \(\tau\colon \Bisp\congto \Bisp[Y]\) that satisfies
  \(\rg\circ\tau = \s\), \(\s\circ\tau = \rg\), and
  \(\tau(h x g) = g^{-1} \tau(x) h^{-1}\).  Since the
  right actions on \(\Bisp\) and~\(\Bisp[Y]\) are basic and the left
  anchor maps induce homeomorphisms, it then follows that both the
  right and the left action on~\(\Bisp\) are basic and that both the
  left and the right anchor maps for~\(\Bisp\) descend to
  homeomorphisms \(\Bisp/\Gr[H] \congto \Gr^0\) and
  \(\Gr\backslash\Bisp \congto \Gr[H]^0\).  This says that~\(\Bisp\)
  is a Morita equivalence.  It remains to construct the
  homeomorphism~\(\tau\).

  Let \(x\in\Bisp\).  We claim that there is a unique element
  \(y \in \Bisp[Y]\) with \(\s(x) = \rg(y)\) and
  \(\varphi[x,y]=1_{\rg(x)}\) for the homeomorphism
  \(\varphi\colon \Bisp \Grcomp \Bisp[Y] \congto \Gr[H]\).
  Since~\(\rg_{\Bisp[Y],*}\) is surjective, there is
  \(y_1\in\Bisp[Y]\) with \(\rg(y_1) = \s(x)\).  Then
  \(\s(\varphi[x,y_1]) = \s(y_1)\) and
  \(\rg(\varphi[x,y_1]) = \rg(x)\), so that
  \(y_1\cdot \varphi[x,y_1]^{-1}\) exists and
  \(\varphi[x,y_1\cdot \varphi[x,y_1]^{-1}]= \varphi[x,y_1]\cdot
  \varphi[x,y_1]^{-1}= 1_{\rg(x)}\).  This proves the existence of
  an element \(y \in \Bisp[Y]\) with \(\s(x) = \rg(y)\) and
  \(\varphi[x,y]=1_{\rg(x)}\).  Let \(y_2\in\Bisp[Y]\) be another
  element with these two properties.  Since~\(\varphi\) is
  injective, it follows that \([x,y] = [x,y_2]\) in
  \(\Bisp\Grcomp \Bisp[Y]\).  That is, there is \(g\in \Gr\) with
  \(\rg(g) = \s(x)\) and \([x,y_2] = [x g, g^{-1} y]\).  Since the
  right \(\Gr\)\nb-action on~\(\Bisp\) is basic, this is only
  possible for \(g=1_{\s(x)}\).  Then \(y=y_2\).  So the
  element~\(y\) above is unique.  We let
  \(\tau\colon \Bisp\to\Bisp[Y]\) be the map that sends~\(x\) to
  this unique~\(y\).  The map~\(\tau\) satisfies
  \(\rg\circ\tau = \s\), \(\s\circ\tau = \rg\), and
  \(\tau(h\cdot x\cdot g) = g^{-1} \tau(x) h^{-1}\) because
  \(\varphi(x g,g^{-1} \tau(x)) = \varphi(x,\tau(x)) = 1_{\rg(x)}\)
  and
  \(\varphi(h x, \tau(x) h^{-1}) = h \varphi(x,\tau(x)) h^{-1} = h
  1_{\rg(x)} h^{-1} = 1_{\rg(h)}\).

  Since~\(\s_{\Bisp}\) is surjective, an argument as above shows
  that for any \(y\in\Bisp[Y]\) there is \(x\in\Bisp\) with
  \(\varphi[x,y]=1_{\s(y)}\) or, equivalently, \(\tau(x)=y\).
  Assume that \(\tau(x_1) = \tau(x_2) = y\).  Then
  \([x_1,y]= [x_2,y]\).  Then there is \(g\in \Gr\) with
  \((x_2,y) = (x_1 g,g^{-1} y)\).  In particular, \(g^{-1} y = y\).
  There is \(x_3 \in\Bisp\) with \(\rg(x_3) = \s(y)\).  Then
  \(g^{-1}(y,x_3) = (y,x_3)\).  Since there is a
  \(\Gr\)\nb-equivariant map
  \(\psi\colon \Bisp[Y] \Grcomp \Bisp \congto \Gr\), this is only
  possible if \(g=1_{\rg(y)}\).  And then \(x_1 = x_2\).  Thus the
  map~\(\tau\) is also injective.

  Since~\(\rg_{\Bisp[Y],*}\) is a homeomorphism and the orbit space
  projection \(\Bisp[Y] \to \Bisp[Y]/\Gr\) is a local homeomorphism,
  it follows that~\(\rg_{\Bisp[Y]}\) is a local homeomorphism.
  Therefore, the choice of~\(y\) in the proof above may be done in a
  continuous fashion depending on~\(x\).  Therefore, the
  map~\(\tau\) is continuous.  The map~\(\s_{\Bisp[X]}\) is a local
  homeomorphism by assumption.  The bijection~\(\tau\) intertwines
  the local homeomorphisms \(\s_{\Bisp[X]}\) and~\(\rg_{\Bisp[Y]}\).
  This forces it to be a homeomorphism because it is also
  continuous.

  Finally, we must prove the converse, namely, that a groupoid
  correspondence~\(\Bisp\) is an equivalence provided both actions
  on~\(\Bisp\) are proper and the anchor maps induce homeomorphisms
  \(\Bisp/\Gr \cong \Gr[H]^0\) and
  \(\Gr[H]\backslash \Bisp \cong \Gr^0\).  Under these assumptions,
  the same set~\(\Bisp\) with
  \(g\cdot x\cdot h \defeq h^{-1} x g^{-1}\) is a groupoid
  correspondence \(\Bisp[Y]\colon \Gr \leftarrow \Gr[H]\).
  \longref{Proposition}{pro:innprod} gives a map
  \(\Bisp \times \Bisp \to \Gr\), \((x,y)\mapsto \braket{y}{x}\),
  which is uniquely determined by \(y\cdot \braket{y}{x} = x\).  The
  properties of this map imply that it descends to a homeomorphism
  \(\Bisp[Y] \Grcomp \Bisp \congto \Gr\).  Exchanging left and right
  gives a similar homeomorphism
  \(\Bisp \Grcomp \Bisp[Y] \congto \Gr[H]\).  Thus~\(\Bisp\) is an
  equivalence in~\(\Grcat\).
\end{proof}

\section{Diagrams of groupoid correspondences}
\label{sec:diagrams}

A diagram in a category~\(\Cat[D]\) is defined as a functor from a
small category to~\(\Cat[D]\).  Analogously, a diagram in a
bicategory~\(\Cat[D]\) is a homomorphism from a small
bicategory~\(\Cat\) to~\(\Cat[D]\).  We are interested in such
diagrams in the bicategory~\(\Grcat\) of groupoid correspondences.
We shall mostly study the case when the domain~\(\Cat\) is a
category, viewed as a bicategory with only identity \(2\)\nb-arrows.

Diagrams in a category form a category with natural
transformations as arrows.  Similarly, diagrams in a bicategory form a
bicategory; its arrows and \(2\)\nb-arrows are called
transformations and modifications.  These concepts are defined
succinctly in~\cite{Leinster:Basic_Bicategories}, including some
technical fine print.  Besides the homomorphisms of bicategories,
there are also ``morphisms'', where certain \(2\)\nb-arrows are not
required to be invertible.  And we shall actually restrict attention
to those homomorphisms that are strictly unital because this
assumption is no great loss of generality and simplifies the
discussion.  Similarly, there are different flavours of
transformations, and we shall only use the ``strong'' ones.  The
following propositions describe strictly unital homomorphisms
\(\Cat\to\Grcat\) for a category~\(\Cat\), strong transformations
between them, and modifications between these strong
transformations.  They may also be taken as definitions.  If we
allow~\(\Cat\) to be a \emph{bi}category, then the general
definitions in \cites{Johnson-Yau:2-Dim, Leinster:Basic_Bicategories} do not simplify
any more.  So there is no point in repeating them here.

After the definition, we examine some classes of diagrams,
identifying them with generalised dynamical systems that were
already considered in previous work.  First, we consider tight
diagrams of spaces, viewed as groupoids.  These are equivalent to
actions of categories on spaces by topological correspondences.
Secondly, we consider tight diagrams of groups.  These turn out to
be more general than complexes of groups.  Our conventions are,
however, somewhat different, so that formulas must be adapted a bit
to carry them over, and some assumptions on the shape of complexes
of groups are not needed for our diagrams.

Third, we look at group actions on groupoids.  These turn out to be
the same as actions by Morita equivalences.  These were already
studied in~\cite{Buss-Meyer:Actions_groupoids} in the locally
compact case, and they may be described through \(G\)\nb-graded
groupoids, where~\(G\) is the group that acts.  Our fourth class of
examples are actions of free monoids.  Here it suffices to specify
the action on the generators, which need not satisfy any relations.
The fifth class of examples are path categories of directed graphs,
which are treated in the same way, except that the free generators
of our category may now have different range and source objects.
Finally, the sixth class of examples are actions of free commutative
monoids.  These are more complicated because the commutativity
relation must be implemented in the diagram.  Our proofs are
somewhat sketchy here because analogous statements are known in the
theory of higher-rank graphs.

\begin{proposition}
  \label{pro:diagrams_in_Grcat}
  A \emph{\(\Cat\)\nb-shaped diagram} in~\(\Grcat\) or a
  \emph{strictly unital homomorphism} \(\Cat\to\Grcat\) consists of
  \begin{enumerate}
    \item groupoids~\(\Gr_x\) for all objects~\(x\) of~\(\Cat\);
    \item correspondences \(\Bisp_g\colon \Gr_x\leftarrow \Gr_y\) for all
    arrows \(g\colon x\leftarrow y\) in~\(\Cat\);
    \item isomorphisms of correspondences \(\mu_{g,h}\colon
    \Bisp_g\Grcomp_{\Gr_y} \Bisp_h\congto \Bisp_{g h}\) for all pairs of
    composable arrows \(g\colon z\leftarrow y\), \(h\colon y\leftarrow
    x\) in~\(\Cat\);
  \end{enumerate}
  such that
  \begin{enumerate}[label=\textup{(\ref*{pro:diagrams_in_Grcat}.\arabic*)},
    leftmargin=*,labelindent=0em]
    \item \label{en:diagrams_in_Grcat_1} \(\Bisp_x\) for an object~\(x\)
    of~\(\Cat\) is the identity correspondence~\(\Gr_x\) on~\(\Gr_x\);
    \item \label{en:diagrams_in_Grcat_2}
    \(\mu_{g,y}\colon \Bisp_g \Grcomp_{\Gr_y} \Gr_y \congto \Bisp_g\)
    and
    \(\mu_{x,g}\colon \Gr_x \Grcomp_{\Gr_x} \Bisp_g \congto \Bisp_g\)
    for an arrow \(g\colon x\leftarrow y\)
    in~\(\Cat\)
    are the canonical isomorphisms described in
    \eqref{eq:unitor_left}--\eqref{eq:unitor_right};
    \item \label{en:diagrams_in_Grcat_3} for all composable arrows
    \(g_{01}\colon x_0\leftarrow x_1\), \(g_{12}\colon x_1\leftarrow
    x_2\), \(g_{23}\colon x_2\leftarrow x_3\) in~\(\Cat\), the
    following diagram commutes:
    \begin{equation}
      \label{eq:coherence_category-diagram}
      \begin{tikzpicture}[yscale=1.5,xscale=3,baseline=(current bounding
      box.west)]
        \node (m-1-1) at (144:1)
        {\((\Bisp_{g_{01}}\Grcomp_{\Gr_{x_1}} \Bisp_{g_{12}})
          \Grcomp_{\Gr_{x_2}} \Bisp_{g_{23}}\)};
        \node (m-1-1b) at (216:1) {\(\Bisp_{g_{01}}\Grcomp_{\Gr_{x_1}}
          (\Bisp_{g_{12}}\Grcomp_{\Gr_{x_2}} \Bisp_{g_{23}})\)};
        \node (m-1-2) at (72:1)
        {\(\Bisp_{g_{02}}\Grcomp_{\Gr_{x_2}}\Bisp_{g_{23}}\)};
        \node (m-2-1) at (288:1)
        {\(\Bisp_{g_{01}}\Grcomp_{\Gr_{x_1}}\Bisp_{g_{13}}\)};
        \node (m-2-2) at (0:.8) {\(\Bisp_{g_{03}}\)};
        \draw[dar] (m-1-1) -- node[swap] {\(\scriptstyle\cong\)} node
        {\scriptsize\textup{associator}} (m-1-1b);
        \draw[dar] (m-1-1.north) -- node[very near end]
        {\(\scriptstyle\mu_{g_{01},g_{12}}\Grcomp_{\Gr_{x_2}}\id_{\Bisp_{g_{23}}}\)}
         (m-1-2.west);
        \draw[dar] (m-1-1b.south) -- node[swap,very near end]
        {\(\scriptstyle\id_{\Bisp_{g_{01}}}\Grcomp_{\Gr_{x_1}}\mu_{g_{12},g_{23}}\)}
        (m-2-1.west);
        \draw[dar] (m-1-2.south) -- node[inner sep=0pt]
        {\(\scriptstyle\mu_{g_{02},g_{23}}\)} (m-2-2);
        \draw[dar] (m-2-1.north) -- node[swap,inner sep=1pt]
        {\(\scriptstyle\mu_{g_{01},g_{13}}\)} (m-2-2);
      \end{tikzpicture}
    \end{equation}
    here \(g_{02}\defeq g_{01}\circ g_{12}\), \(g_{13}\defeq
    g_{12}\circ g_{23}\), and \(g_{03}\defeq g_{01}\circ g_{12}\circ
    g_{23}\).
    The diagram~\eqref{eq:coherence_category-diagram} commutes
    automatically if one of the arrows \(g_{01}\), \(g_{12}\)
    or~\(g_{23}\) is an identity arrow.
  \end{enumerate}
\end{proposition}

The maps~\(\mu_{g,h}\) are determined by the composite maps
\begin{equation}
  \label{eq:diagram_as_mult}
  \Bisp_g\times_{\s,\rg} \Bisp_h
  \prto \Bisp_g\Grcomp \Bisp_h
  \to \Bisp_{g h}
\end{equation}
for composable arrows \(g,h\) in~\(\Cat\), which we write
multiplicatively:
\(\gamma\cdot \eta\defeq \mu_{g,h}[\gamma,\eta]\in \Bisp_{g\cdot
  h}\) for \(\gamma\in \Bisp_g\), \(\eta\in\Bisp_h\) with
\(\s(\gamma)=\rg(\eta)\), where \([\gamma,\eta]\) denotes the image
of \((\gamma,\eta)\) in \(\Bisp_g\Grcomp \Bisp_h\).  If \(g=\rg(h)\)
or \(h=\s(g)\), this multiplication is the left
\(\Gr_{\rg(h)}\)\nb-action on~\(\Bisp_h\) and the right
\(\Gr_{\s(g)}\)\nb-action on~\(\Bisp_g\) by
\ref{en:diagrams_in_Grcat_1} and~\ref{en:diagrams_in_Grcat_2}.  Thus
a \(\Cat\)\nb-shaped diagram is determined by spaces~\(\Gr_x^0\) for
\(x\in\Cat^0\) and~\(\Bisp_g\) for \(g\in\Cat\) with continuous maps
\(\rg\colon \Bisp_g \to \Gr_{\rg(g)}\) and local homeomorphisms
\(\s\colon \Bisp_g \to \Gr_{\s(g)}\) for all \(g\in\Cat\) and with
associative multiplication maps as in~\eqref{eq:diagram_as_mult};
the associativity is equivalent to the condition
\ref{en:diagrams_in_Grcat_3}.  In addition, we need to assume that
\(\Gr_x^0\) and \(\Gr_x\defeq \Bisp_x\) with the given range and
source maps and multiplication form an étale groupoid for each~\(x\)
and that the right actions of these groupoids on the
spaces~\(\Bisp_g\) defined by the multiplication maps are basic.

\begin{proposition}
  \label{pro:trafo_category-diagram}
  Let \((\Gr_x^0,\Bisp_g^0,\mu_{g,h}^0)\) and
  \((\Gr_x^1,\Bisp_g^1,\mu_{g,h}^1)\) be two \(\Cat\)\nb-shaped
  diagrams in~\(\Grcat\).  A \emph{strong transformation} between
  them consists of
  \begin{enumerate}
  \item groupoid correspondences
    \(\Bisp[Y]_x\colon \Gr_x^1 \leftarrow \Gr_x^0\) for all
    objects~\(x\) of~\(\Cat\);
  \item isomorphisms of correspondences
    \(V_g\colon \Bisp_g^1\Grcomp_{\Gr^1_y} \Bisp[Y]_y \congto
    \Bisp[Y]_x\Grcomp_{\Gr^0_x} \Bisp_g^0\) for all arrows
    \(g\colon x\leftarrow y\) in~\(\Cat\);
  \end{enumerate}
  such that
  \begin{enumerate}[label=\textup{(\ref*{pro:trafo_category-diagram}.\arabic*)},
    leftmargin=*,labelindent=0em]
    \item for each object~\(x\)
    of~\(\Cat\), the isomorphism
    \(V_x\colon \Gr^1_x \Grcomp_{\Gr^1_x}
    \Bisp[Y]_x\congto\Bisp[Y]_x\Grcomp_{\Gr^0_x} \Gr^0_x \)
    is the canonical one mapping
    \([g_1,y\cdot g_2]\) to \([g_1\cdot y, g_2]\);
    \item for each pair of composable arrows \(g\colon x\leftarrow y\),
    \(h\colon y\leftarrow z\) in~\(\Cat\), the following diagram
    commutes:
    \begin{equation}
      \label{eq:trafo_category-diagram}
      \begin{tikzpicture}[yscale=1.5,xscale=3,baseline=(current bounding
      box.west)]
        \node (cbb) at (144:1)
        {\(\Bisp[Y]_x\Grcomp_{\Gr^0_x} \Bisp^0_g\Grcomp_{\Gr^0_y}
        \Bisp^0_h\)};
        \node (cb) at (216:1)
        {\(\Bisp[Y]_x\Grcomp_{\Gr^0_x} \Bisp^0_{g h}\)};
        \node (acb) at (72:1)
        {\(\Bisp^1_g\Grcomp_{\Gr^1_y} \Bisp[Y]_y \Grcomp_{\Gr^0_y}
        \Bisp^0_h\)};
        \node (ac) at (288:1)
        {\(\Bisp^1_{g h}\Grcomp_{\Gr^1_z} \Bisp[Y]_z\)};
        \node (aac) at (0:.8)
        {\(\Bisp^1_g\Grcomp_{\Gr^1_y} \Bisp^1_h \Grcomp_{\Gr^1_z}
        \Bisp[Y]_z\)};
        \draw[dar] (cbb) -- node
        {\(\scriptstyle\id_{\Bisp[Y]_x}\Grcomp_{\Gr^0_x} \mu^0_{g,h}\)} (cb);
        \draw[dar] (acb.west) -- node[swap,very near start]
        {\(\scriptstyle V_g\Grcomp_{\Gr^0_y} \id_{\Bisp^0_h}\)} (cbb.north);
        \draw[dar] (ac) -- node[near start] {\(\scriptstyle V_{g h}\)} (cb.south);
        \draw[dar] (aac.north) -- node[swap]
        {\(\scriptstyle\id_{\Bisp^1_g}\Grcomp_{\Gr^1_y} V_h\)} (acb.south);
        \draw[dar] (aac.south) -- node
        {\(\scriptstyle\mu^1_{g,h} \Grcomp_{\Gr^1_z} \id_{\Bisp[Y]_z}\)} (ac.north);
      \end{tikzpicture}
    \end{equation}
    Here we left out associators to keep the diagram readable.
    The diagram~\eqref{eq:trafo_category-diagram} commutes
    automatically if \(g\) or~\(h\) is an identity arrow.
  \end{enumerate}
\end{proposition}

We shall not use transformations that are not strong.

\begin{proposition}
  \label{pro:modification_category-diagram}
  Let \((\Gr_x^0,\Bisp_g^0,\mu_{g,h}^0)\) and
  \((\Gr_x^1,\Bisp_g^1,\mu_{g,h}^1)\) be \(\Cat\)\nb-shaped diagrams
  in~\(\Grcat\) and let \((\Bisp[Y]^1_x,V^1_g)\) and
  \((\Bisp[Y]^2_x,V^2_g)\) be strong transformations between them.  A
  \emph{modification} from \((\Bisp[Y]^1_x,V^1_g)\) to
  \((\Bisp[Y]^2_x,V^2_g)\) consists of continuous maps of groupoid
  correspondences \(W_x\colon \Bisp[Y]^1_x \to \Bisp[Y]^2_x\) for
  all objects~\(x\) in~\(\Cat\) such that the diagrams
  \begin{equation}
    \label{eq:modification_category-diagram}
    \begin{tikzpicture}[yscale=1.3,xscale=4,baseline=(current bounding
    box.west)]
      \node (add) at (0,1) {\(\Bisp[Y]^1_x \Grcomp_{\Gr^0_x}\Bisp^0_g\)};
      \node (ad) at (1,1) {\(\Bisp[Y]^2_x\Grcomp_{\Gr^0_x} \Bisp^0_g\)};
      \node (ddc) at (0,0) {\(\Bisp^1_g\Grcomp_{\Gr^1_y} \Bisp[Y]^1_y\)};
      \node (dc) at (1,0) {\(\Bisp^1_g\Grcomp_{\Gr^1_y} \Bisp[Y]^2_y\)};

      \draw[dar] (add) -- node {\(\scriptstyle W_x\Grcomp_{\Gr^0_x} \id_{\Bisp^0_g}\)}
      (ad);
      \draw[dar] (ddc) -- node {\(\scriptstyle V^1_g\)} (add);
      \draw[dar] (dc) -- node[swap] {\(\scriptstyle V^2_g\)} (ad);
      \draw[dar] (ddc) -- node[swap] {\(\scriptstyle \id_{\Bisp^1_g}\Grcomp_{\Gr^1_y}
      W_y\)}
      (dc);
    \end{tikzpicture}
  \end{equation}
  commute for all arrows \(g\colon x\leftarrow y\) in~\(\Cat\).  This
  diagram commutes automatically if~\(g\) is an identity arrow.
\end{proposition}

Two modifications \((\Bisp[Y]^1_x,V^1_g)\to (\Bisp[Y]^2_x,V^2_g)\to
(\Bisp[Y]^3_x,V^3_g)\) compose in the obvious way to a modification
\((\Bisp[Y]^1_x,V^1_g)\to (\Bisp[Y]^3_x,V^3_g)\).  This composition is
associative and unital, so the strong transformations between two diagrams
and the modifications between them form a category.

Transformations between diagrams are composed as follows.  Describe
\(\Cat\)\nb-shaped diagrams in~\(\Grcat\) by
\((\Gr_x^0,\Bisp_g^0,\mu_{g,h}^0)\),
\((\Gr_x^1,\Bisp_g^1,\mu_{g,h}^1)\) and
\((\Gr_x^2,\Bisp_g^2,\mu_{g,h}^2)\), and strong transformations between them
by \((\Bisp[Y]^{10}_x,V^{10}_g)\) and \((\Bisp[Y]^{21}_x,V^{21}_g)\) as
above.  The composite transformation is defined by \(\Bisp[Y]^{20}_x
\defeq \Bisp[Y]^{21}_x \Grcomp_{\Gr_x^1} \Bisp[Y]^{10}_x\) for
objects~\(x\) of~\(\Cat\) and
\begin{multline*}
  V^{20}_g\colon
  \Bisp_g^2 \Grcomp_{\Gr_y^2} \Bisp[Y]_y^{20}
  = \Bisp_g^2 \Grcomp_{\Gr_y^2} \Bisp[Y]_y^{21} \Grcomp_{\Gr_y^1}
  \Bisp[Y]_y^{10}
  \xrightarrow{V^{21}_g\Grcomp_{\Gr_y^1} \id_{\Bisp[Y]_y^{10}}}
  \Bisp[Y]_x^{21} \Grcomp_{\Gr_x^1} \Bisp_g^1 \Grcomp_{\Gr_y^1}
  \Bisp[Y]_y^{10}
  \\\xrightarrow{\id_{\Bisp[Y]_x^{21}} \Grcomp_{\Gr_x^1} V^{10}_g}
  \Bisp[Y]_x^{21} \Grcomp_{\Gr_x^1} \Bisp[Y]_x^{10} \Grcomp_{\Gr_x^0}
  \Bisp_g^0
  = \Bisp[Y]_x^{20} \Grcomp_{\Gr_x^0} \Bisp_g^0
\end{multline*}
for all arrows \(g\colon x\leftarrow y\) in~\(\Cat\).  Here we left
out associators to make the construction more readable.  These
\((\Bisp[Y]^{20}_x,V^{20}_g)\) indeed form a strong transformation.

There is also a horizontal product of modifications between
composable transformations, and there are canonical associators and
unit transformations.  This extra structure turns the
\(\Cat\)\nb-shaped diagrams in~\(\Grcat\), strong transformations
and modifications into a bicategory (see
\cite{Johnson-Yau:2-Dim}*{Section~4.4}).  We denote it
by~\(\Grcat^{\Cat}\).  We care about this bicategory because the
groupoid model construction is a homomorphism
\(\Grcat^{\Cat} \to \Grcat\) (see
\longref{Corollary}{cor:groupoid_model_functor}).

\begin{definition}
  Let~\(\Cat\) be a category.  A diagram of groupoid correspondences
  \(\Cat\to\Grcat\) described by the data
  \((\Gr_x,\Bisp_g,\mu_{g,h})\) is \emph{tight} or \emph{proper} if
  all the groupoid correspondences~\(\Bisp_g\) are tight or proper,
  respectively.  It is \emph{locally compact} if all the
  groupoids~\(\Gr_x\) and the correspondences~\(\Bisp_g\) are
  locally compact.

  A strong transformation~\((\Bisp[Y]_x,V_g)\) between two diagrams
  of groupoid correspondences \(\Cat\rightrightarrows\Grcat\) is
  \emph{tight} or \emph{proper} if the groupoid
  correspondences~\(\Bisp[Y]_x\) are tight or proper, respectively.
  It is \emph{locally compact} if both diagrams and all the
  correspondences~\(\Bisp[Y]_x\) are locally compact.
\end{definition}

The composite of two tight, proper or locally compact strong
transformations remains tight, proper or locally compact,
respectively, by the second sentence in
\longref{Proposition}{pro:groupoid_correspondences_composition}.
The identity transformation on a diagram of groupoids is tight and
\emph{a fortiori} proper, and it is locally compact if and only if
the diagram is locally compact.  Hence there are two subbicategories
of~\(\Grcat^{\Cat}\) that have all diagrams \(\Cat\to\Grcat\) as
objects, but only the tight or proper strong transformations as
arrows.  When we restrict to tight or proper diagrams and
transformations, we get the bicategories \(\Grcat_\tight^{\Cat}\)
and \(\Grcat_\prop^{\Cat}\), respectively, of all homomorphisms
\(\Cat\to\Grcat_\tight\) or \(\Cat\to\Grcat_\prop\).  Similarly,
there is a subbicategory \(\Grcat_\lc^{\Cat}\) of locally compact
diagrams \(\Cat\to\Grcat_\lc\) with locally compact transformations
between them as arrows and all modifications among those as
\(2\)\nb-arrows.

To make diagrams of groupoid correspondences more concrete, we now
describe some subbicategories of~\(\Grcat^{\Cat}\).  This means that
we specify a simpler bicategory that is equivalent to it.

\subsection{Some diagrams of spaces}
\label{sec:space_diagram}

Let \(\Gr_x\) and~\(\Gr_y\) be topological spaces, viewed as étale
groupoids with only identity arrows.  Then a groupoid correspondence
\(\Gr_x\leftarrow \Gr_y\) is a topological space~\(\Bisp\) with a
continuous map \(\rg\colon \Bisp\to\Gr_x=\Gr_x^0\) and a local
homeomorphism \(\s\colon \Bisp\to\Gr_y=\Gr_y^0\) (in the locally
compact case, this is
\cite{Antunes-Ko-Meyer:Groupoid_correspondences}*{Example~4.1}).
The groupoid correspondence is proper or tight if and only if the
map~\(\rg\) is proper or a homeomorphism, respectively.  The
correspondence is locally compact if and only if \(\Gr_x\)
and~\(\Gr_y\) are locally compact and~\(\Bisp\) is Hausdorff.  If
\(\Gr_x=\Gr_y\) and the correspondence is locally compact, then it
is called a ``topological graph''; this is the data used by
Katsura~\cite{Katsura:class_I} to define a topological graph
\(\Cst\)\nb-algebra.  We get an ordinary graph \(\Cst\)\nb-algebra
if~\(\Gr_x\) carries the discrete topology.

Let~\(\Cat\) be a monoid, viewed as a category with unique
object~\(x\).  Let \(F\colon \Cat\to\Grcat_\lc\) be a diagram such
that \(\Gr_x = F(x)\) is a locally compact space, viewed as an étale
locally compact groupoid with only identity arrows.  The
diagram~\(F\) is the same as an action of the monoid~\(\Cat\) on the
topological space~\(\Gr_x\) by topological correspondences as
studied by Albandik and the author~\cite{Albandik-Meyer:Product}.
If \(\Cat=(\N^k,+)\), this becomes equivalent to a topological
rank-\(k\) graph as defined by
Yeend~\cite{Yeend:Topological-higher-rank-graphs}.  If
\(\Cat=(\N,+)\), we may replace the diagram by a single
correspondence (see \longref{Lemma}{lem:action_one_correspondence}).
This gives a topological graph as above.

Let~\(\Cat\) be any category and let \(F\colon \Cat\to\Grcat_\lc\)
be a diagram such that \(\Gr_x = F(x)\) for \(x\in\Cat^0\) is a
\emph{discrete} space, viewed as an étale groupoid with the discrete
topology and only identity arrows.  It is proven
in~\cite{Antunes-Ko-Meyer:Groupoid_correspondences} that~\(F\)
becomes equivalent to a discrete Conduché fibration as defined by
Brown and Yetter~\cite{Brown-Yetter:Conduche}.

Each of the examples above has been used by \(\Cst\)\nb-algebraists
to define interesting generalisations of graph \(\Cst\)\nb-algebras.

We now specialise to the tight case.  We will see later that the
groupoid model is easier to construct in that case, and tight
diagrams of groupoid correspondences between spaces are the easiest
class of tight diagrams.  To study tight diagrams, we first prove
that the subbicategory of tight groupoid correspondences between
spaces is equivalent to an ordinary category, made a bicategory with
only identity \(2\)\nb-arrows.  This allows to simplify diagrams of
such correspondences.

\begin{proposition}
  \label{pro:tight_diagrams_spaces}
  The bicategory of tight groupoid correspondences between spaces is
  equivalent as a bicategory to the opposite category~\(\Cat[S]\) of
  spaces with local homeomorphisms as arrows, and only identity
  \(2\)\nb-arrows.
\end{proposition}

\begin{proof}
  Any tight correspondence \(\Gr[H]\leftarrow\Gr\) for two spaces
  \(\Gr[H]\) and~\(\Gr\) is isomorphic to one of the form
  \[
    \Gr[H]_f\colon
    \Gr[H] \xleftarrow{\id_{\Gr[H]}} \Gr[H] \xrightarrow{f} \Gr
  \]
  for a local homeomorphism \(f\colon \Gr\to\Gr[H]\).  A
  \(2\)\nb-arrow \(\Gr[H]_f \Rightarrow \Gr[H]_{f'}\) between two
  correspondences of this form only exists if \(f=f'\), and then the
  only \(2\)\nb-arrow is the identity \(2\)\nb-arrow.  The
  composition of correspondences is the composition of local
  homeomorphisms in reverse order.  That is, if
  \(f\colon \Gr[H]\to \Gr\) and \(g\colon \Gr\to\Gr[K]\) are local
  homeomorphisms, then there is a unique isomorphism of
  correspondences
  \(\Gr[H]_f \Grcomp_{\Gr} \Gr_g \cong \Gr[H]_{g\circ f}\).  This
  implies the asserted equivalence of bicategories.
\end{proof}

The equivalence of bicategories in
\longref{Proposition}{pro:tight_diagrams_spaces} induces an
equivalence between the corresponding bicategories of diagrams of
any shape~\(\Cat\).  That is, the restriction of
\(\Grcat^{\Cat}_\tight\) to spaces as objects is equivalent to the
bicategory that has functors \(\Cat\to\Cat[S]\) as objects, natural
transformations between them as arrows, and only identity
\(2\)\nb-arrows.  If, in addition, we assume~\(\Cat\) to be a
monoid, we get the classical category of dynamical systems, with a
monoid acting by local homeomorphisms -- except for a reversal in
the order of composition, which is due to conventions.

\subsection{Tight diagrams of groups vs complexes of groups}
\label{sec:tight_group_diagram}

Let \(\Gr[H]\) and~\(\Gr\) be groups.  Groupoid correspondences
\(\Gr[H] \leftarrow \Gr\) are described in
\cite{Antunes-Ko-Meyer:Groupoid_correspondences}*{Example~4.2}.  Up
to a nondegeneracy condition which we need not impose, a proper
groupoid correspondence on a group is the same as a self-similarity
of that group.  Nekrashevych~\cite{Nekrashevych:Cstar_selfsimilar}
defined \(\Cst\)\nb-algebras for such self-similar groups.

Once again, we specialise to a tight groupoid correspondence
\(\Gr[H] \leftarrow \Gr\).  Then there is a group homomorphism
\(\varphi\colon \Gr[H]\to \Gr\) so that the correspondence is
isomorphic to~\({}_\varphi \Gr\), the group~\(\Gr\) equipped with
the obvious right \(\Gr\)\nb-action by multiplication and the left
\(\Gr[H]\)\nb-action \(h\cdot g \defeq \varphi(h)g\) for
\(h\in \Gr[H]\), \(g\in\Gr\).  We are going to enrich this to an
equivalence of bicategories and use it to simplify diagrams of tight
groupoid correspondences between groups.  The resulting notions are
closely related to standard notions from the theory of complexes of
groups (see~\cite{Bridson-Haefliger}*{Chapter
  III.\(\mathcal{C}\).2}).

Let \(\varphi,\varphi'\colon \Gr[H]\rightrightarrows \Gr\) be group
homomorphisms.  The \(2\)\nb-arrows
\({}_\varphi\Gr \Rightarrow {}_{\varphi'} \Gr\)
between the tight correspondences defined by \(\varphi\) and~\(\varphi'\)
are maps of the form
\(g\mapsto u g\)
for some group element \(u\in\Gr\)
with \(u\varphi(h) = \varphi'(h) u\)
for all \(h\in \Gr[H]\);
we briefly write \(u\colon \varphi\Rightarrow \varphi'\).
Groups with homomorphisms as arrows and elements
\(u\colon \varphi\Rightarrow \varphi'\)
as above as \(2\)\nb-arrows
form a \(2\)\nb-category,
that is, a bicategory where the associators and unit transformations
are trivial.  The opposite of this \(2\)\nb-category
is equivalent to the bicategory of groups with tight groupoid
correspondences because for two group homomorphisms
\(\varphi\colon \Gr[H]\to\Gr\),
\(\psi\colon \Gr\to\Gr[K]\),
there is a canonical isomorphism of correspondences
\[
{}_\varphi \Gr \Grcomp {}_\psi \Gr[K] \cong
{}_{\psi\circ\varphi}\Gr[K],\qquad
(g,k)\mapsto \psi(g)\cdot k.
\]
Hence the restriction of \(\Grcat^{\Cat}_\tight\) to groups as objects
is equivalent to the \(2\)\nb-category of \(\Cat\)\nb-shaped diagrams
in the \(2\)\nb-category of groups just described.

We make this more explicit.  Consider a \(\Cat\)\nb-shaped diagram
of tight correspondences \((\Gr_x,\Bisp_g,\mu_{g,h})\) as in
\longref{Proposition}{pro:diagrams_in_Grcat}, where all~\(\Gr_x\)
are groups.  For each arrow \(g\in\Cat\), there are a group
homomorphism \(\varphi_g\colon \Gr_{\rg(g)} \to \Gr_{\s(g)}\) and an
isomorphism \(V_g\colon \Bisp_g \cong {}_{\varphi_g}\Gr_{\s(g)}\).
For composable arrows \(g\colon x\leftarrow y\) and
\(h\colon y\leftarrow z\), the isomorphism
\[
  {}_{\varphi_h \circ \varphi_g}\Gr_z
  \xrightarrow[\cong]{\mathrm{can.}} {}_{\varphi_g}\Gr_y \Grcomp {}_{\varphi_h}\Gr_z
  \xrightarrow[\cong]{V_g^{-1} \Grcomp V_h^{-1}} \Bisp_g \Grcomp \Bisp_h
  \xrightarrow[\cong]{\mu_{g,h}} \Bisp_{g h}
  \xrightarrow[\cong]{V_{g h}} {}_{\varphi_{g h}}\Gr_z
\]
corresponds to a \(2\)\nb-arrow
\(u_{g,h} \colon \varphi_h\circ \varphi_g \Rightarrow \varphi_{g
  h}\), that is, a group element \(u_{g,h} \in \Gr_z\) with
\(u_{g,h} \varphi_h(\varphi_g(\gamma)) = \varphi_{g h}(\gamma)
u_{g,h}\) for all \(\gamma\in \Gr_x\).  The associativity
condition~\eqref{eq:coherence_category-diagram} for the original
diagram~\((\Gr_x,\Bisp_g,\mu_{g,h})\) becomes the diagram
\[
  \begin{tikzcd}[column sep=large]
    \varphi_k\circ \varphi_h \circ \varphi_g
    \ar[r, Rightarrow, "\varphi_k(u_{g,h})"]
    \ar[d, Rightarrow, "u_{h,k}"'] &
    \varphi_k\circ \varphi_{g h}
    \ar[d, Rightarrow, "u_{g h,k}"] \\
    \varphi_{h k}\circ \varphi_g
    \ar[r, Rightarrow, "u_{g, h k}"]
    &
    \varphi_{g h k},
  \end{tikzcd}
\]
which translates to the condition
\begin{equation}
  \label{eq:group_diagram_cocycle}
  u_{g h,k}\cdot \varphi_k(u_{g,h}) = u_{g,h k} \cdot u_{h,k}
\end{equation}
for all composable arrows \(g,h,k\) in~\(\Cat\).  The requirement
\(\Bisp_x = \Gr_x\) for objects~\(x\) of~\(\Cat\) allows us to
choose \(\varphi_x = \id_{\Gr_x}\) for all objects~\(x\) of~\(\Cat\)
(which we do).  Then the condition on the isomorphisms
\(\mu_{\rg(g),g}\) and~\(\mu_{g,\s(g)}\) translates to the
normalisation conditions \(u_{\rg(g),g} = 1\) and \(u_{g,\s(g)}=1\)
for all arrows~\(g\) in~\(\Cat\).

The data \((\Gr_x,\varphi_g,u_{g,h})\) is an equivalent way to
describe our original diagram; here~\(\Gr_x\) are groups for
objects~\(x\) of~\(\Cat\),
\(\varphi_g\colon \Gr_{\rg(g)} \to \Gr_{\s(g)}\) are group
homomorphisms for \(g\in\Cat\), and \(u_{g,h} \in \Gr_{\s(g)}\) for
composable \(g,h\) in~\(\Cat\) satisfy
\(\Ad(u_{g,h})\circ \varphi_h \circ \varphi_g = \varphi_{g h}\) and
the cocycle conditions~\eqref{eq:group_diagram_cocycle}.  In
addition, we require the normalisation conditions
\(\varphi_x=\id_{\Gr_x}\) for unit arrows, and \(u_{g,\s(g)} = 1\)
and \(u_{\rg(g),g} = 1\) for each arrow \(g\in\Cat\).  This is very
close to a \emph{complex of groups} as defined in
\cite{Bridson-Haefliger}*{Chapter III.\(\mathcal{C}\).2}.  The
definition of a complex of groups requires, in addition, that all
group homomorphisms~\(\varphi_g\) should be injective and that the
underlying category~\(\Cat\) should have \emph{no loops}, that is,
\(\s(g)\neq \rg(g)\) for all \(g\in\Cat\), and \(g\cdot h\) is only
a unit arrow if both \(g\) and~\(h\) are units.  Furthermore, we
reverse the order of composition of group homomorphisms, so that we
replace~\(\Cat\) by the opposite category~\(\Cat^\op\), and our
group elements~\(u_{g,h}\) are the inverses of the elements used
in~\cite{Bridson-Haefliger}.  After these notational changes,
\eqref{eq:group_diagram_cocycle} becomes the cocycle condition
required for a complex of groups.

Let \((\Bisp[Y]_x,V_g)\) describe a tight strong transformation from the
tight diagram of group correspondences associated to the data
\((\Gr^1_x,\varphi^1_g,u^1_{g,h})\) to a similar diagram for
\((\Gr^2_x,\varphi^2_g,u^2_{g,h})\).  Then there are group
homomorphisms \(\psi_x\colon \Gr^2_x\to\Gr^1_x\) with \(\Bisp[Y]_x
\cong {}_{\psi_x} \Gr^1_x\).  The isomorphism~\(V_g\) for \(g\colon
x\leftarrow y\) induces an isomorphism
\[
{}_{\psi_y \circ \varphi^2_g} \Gr^1_y  \cong
{}_{\varphi^2_g} \Gr^2_x \Grcomp \Bisp[Y]_y
\xrightarrow[\cong]{V_g}
\Bisp[Y]_x \Grcomp {}_{\varphi^1_g} \Gr^1_y
\cong {}_{\varphi^1_g \circ \psi_x} \Gr^1_y,
\]
which corresponds to a \(2\)\nb-arrow
\(v_g \colon \psi_y \circ \varphi^2_g \Rightarrow \varphi^1_g \circ
\psi_x\)
for some \(v_g\in \Gr^1_y\).
The coherence condition~\eqref{eq:trafo_category-diagram} for a
strong transformation becomes
\[
\begin{tikzpicture}
  \matrix[cd, row sep=4ex, column sep = 3em] (m) {
    \psi_z \circ \varphi^2_h \circ \varphi^2_g &
    \varphi^1_h \circ \psi_y \circ \varphi^2_g &
    \varphi^1_h \circ \varphi^1_g \circ \psi_x \\
    \psi_z \circ \varphi^2_{g h} & &
    \varphi^1_{g h} \circ \psi_x, \\
  };
  \draw[dar] (m-1-1) -- node {\(\scriptstyle v_h\)} (m-1-2);
  \draw[dar] (m-1-2) -- node {\(\scriptstyle \varphi^1_h(v_g)\)} (m-1-3);
  \draw[dar] (m-1-1) -- node[swap] {\(\scriptstyle \psi_z(u^2_{g,h})\)} (m-2-1);
  \draw[dar] (m-1-3) -- node {\(\scriptstyle u^1_{g,h}\)} (m-2-3);
  \draw[dar] (m-2-1) -- node {\(\scriptstyle v_{g h}\)} (m-2-3);
\end{tikzpicture}
\]
which is equivalent to the condition
\begin{equation}
  \label{eq:group_trafo_cocycle}
  u^1_{g,h} \cdot \varphi^1_h(v_g)\cdot v_h
  = v_{g h} \cdot \psi_z (u^2_{g,h})
\end{equation}
for composable arrows \(h\colon z\leftarrow y\), \(g\colon
y\leftarrow x\) in~\(\Cat\).  The normalisation condition on~\(V_x\)
for an object~\(x\) translates to \(v_x=1\).

The data~\((\psi_x,v_g)\) is an equivalent way to describe tight
strong transformations between tight diagrams of groups of the special form
above.  Here \(\psi_x\colon \Gr_x^2 \to \Gr_x^1\) is a group
homomorphism for each object~\(x\), and the elements \(v_g\in
\Gr_{\s(g)}\) satisfy
\(\Ad(v_g)\circ \psi_{\s(g)} \circ \varphi^2_g = \varphi^1_g
\circ \psi_{\rg(g)}\) and the
condition~\eqref{eq:group_trafo_cocycle}, in addition to \(v_x = 1\)
for unit arrows.  When \((\Gr_x^i,\varphi_g^i,u_{g,h}^i)\) for
\(i=1,2\) come from complexes of groups as in
\cite{Bridson-Haefliger}*{Chapter III.\(\mathcal{C}\).2}, then such
strong transformations are equivalent to morphisms of complexes of groups
over the identity morphism on~\(\Cat\).

A modification~\((W_x)\) between two tight strong transformations
\((\psi^1_x,v^1_g)\) and \((\psi^2_x,v^2_g)\) of the special form
above is associated to elements \(w_x\in \Gr_x\) with \(w_x \colon
\psi^1_x \Rightarrow \psi^2_x\); the coherence
condition~\eqref{eq:modification_category-diagram} becomes
\[
\begin{tikzpicture}
  \matrix[cd, row sep=4ex, column sep = 3em] (m) {
    \psi^1_y\circ \varphi^2_g &
    \psi^2_y\circ \varphi^2_g \\
    \varphi^1_g\circ \psi^1_x &
    \varphi^1_g\circ \psi^2_x, \\
  };
  \draw[dar] (m-1-1) -- node {\(\scriptstyle w_y\)} (m-1-2);
  \draw[dar] (m-1-1) -- node[swap] {\(\scriptstyle v^1_g\)} (m-2-1);
  \draw[dar] (m-2-1) -- node {\(\scriptstyle \varphi^1_g(w_x)\)} (m-2-2);
  \draw[dar] (m-1-2) -- node {\(\scriptstyle v^2_g\)} (m-2-2);
\end{tikzpicture}
\]
which translates to the conditions
\begin{equation}
  \label{eq:group_modification_cocycle}
  \varphi^1_g(w_x)\cdot v^1_g = v^2_g\cdot w_y
\end{equation}
for all arrows \(g\colon x\leftarrow y\).  In the setting of complexes
of groups, these modifications are equivalent to homotopies between
morphisms of complexes of groups.

Summing up, diagrams in the bicategory of groups and tight groupoid
correspondences specialise to complexes of groups, strong
transformations between them specialise to morphisms of complexes of
groups, and modifications specialise to homotopies between such
morphisms.

\subsection{Group actions on groupoids}
\label{sec:group_action}

Let~\(\Cat\) be a group viewed as a category with only one object.
Describe a diagram \(F\colon \Cat\to\Grcat\) by a groupoid~\(\Gr\),
groupoid correspondences \(\Bisp_g\colon \Gr\leftarrow\Gr\) for
\(g\in\Cat\), and isomorphisms
\(\mu_{g,h}\colon \Bisp_g \Grcomp_{\Gr} \Bisp_h \congto \Bisp_{g
  h}\) for \(g,h\in \Cat\) as in
\longref{Proposition}{pro:diagrams_in_Grcat}.  The groupoid
correspondences~\(\Bisp_g\) are equivalences in the
bicategory~\(\Grcat\).  Hence they are groupoid equivalences in the
usual sense by \longref{Theorem}{the:groupoid_equivalence}.  Thus
our diagrams are a special case of inverse semigroup actions on
groupoids by partial equivalences, as studied
in~\cite{Buss-Meyer:Actions_groupoids}.  We recall the convenient
description of such actions through \(\Cat\)\nb-graded groupoids
(see \cite{Buss-Meyer:Actions_groupoids}*{Theorem~3.14}).

A \emph{\(\Cat\)\nb-graded groupoid} for a group~\(\Cat\) is a
groupoid \(\Gr[L]\) with a disjoint union decomposition
\(\Gr[L] = \bigsqcup_{g\in\Cat} \Gr[L]_g\) as a topological space,
such that \(\Gr[L]_g \cdot \Gr[L]_h = \Gr[L]_{g h}\) and
\(\Gr[L]_g^{-1} = \Gr[L]_{g^{-1}}\) (the last condition is, in fact,
redundant).  Then \(\Gr\defeq \Gr[L]_1\) is a clopen subgroupoid,
and each~\(\Gr[L]_g\) is a groupoid equivalence
\(\Gr[L]_1 \leftarrow \Gr[L]_1\).  The multiplication in~\(\Gr[L]\)
provides isomorphisms
\(\Gr[L]_g \Grcomp_{\Gr} \Gr[L]_h \cong \Gr[L]_{g h}\), which form
an action of~\(\Cat\) on~\(\Gr\) by groupoid equivalences.
Conversely, any action of~\(\Cat\) on a groupoid~\(\Gr\) by groupoid
equivalences is of this form for a \(\Cat\)\nb-graded groupoid,
which is unique up to a canonical grading-preserving isomorphism.
The groupoid~\(\Gr[L]\) is called the \emph{transformation groupoid}
(its construction in
\cite{Buss-Meyer:Actions_groupoids}*{Section~3.2} simplifies
if~\(\Cat\) is a group).

In particular, if~\(\Gr\) is also a group, then an action of~\(\Cat\)
on~\(\Gr\) by equivalences is the same as a group extension \(\Gr
\into \Gr[L] \prto \Cat\); here~\(\Gr[L]\) is the transformation
groupoid (see \cite{Buss-Meyer:Actions_groupoids}*{Theorem~3.15}).
This case is also contained in \longref{Section}{sec:tight_group_diagram}.

If~\(\Gr\) is a locally compact space, viewed as a groupoid, then
any groupoid equivalence \(\Gr\leftarrow\Gr\) is isomorphic to one
coming from a homeomorphism of the space~\(\Gr\), and the
multiplication isomorphisms are unique if they exist.  Thus a group
action on a space by equivalences is equivalent to an action by
homeomorphisms (see
\cite{Buss-Meyer:Actions_groupoids}*{Theorem~3.18}).  The
transformation groupoid \(\Gr\rtimes\Cat\) is the usual one.

\subsection{Free monoid actions}
\label{sec:free_monoid_action}

Now let~\(\Cat\) be the free
monoid on a set of generators~\(S\).  Already the case of a single
generator is interesting.  It gives the monoid~\((\N,+)\).  We are
going to describe \(S\)\nb-diagrams, \(S\)\nb-transformations and
\(S\)\nb-modifications so that the resulting bicategory is equivalent
to~\(\Grcat^{\Cat}\).

An \emph{\(S\)\nb-diagram} consists of a groupoid~\(\Gr\) and groupoid
correspondences~\(\Bisp_g\) for \(g\in S\).  An
\emph{\(S\)\nb-transformation} between two \(S\)\nb-diagrams
\((\Gr^1,\Bisp^1_g)\) and \((\Gr^2,\Bisp^2_g)\) consists of a groupoid
correspondence \(\Bisp[Y]\colon \Gr^2\leftarrow\Gr^1\) and
isomorphisms of groupoid correspondences \(V_g\colon \Bisp^2_g \Grcomp
\Bisp[Y] \congto \Bisp[Y] \Grcomp \Bisp^1_g\) for \(g\in S\).  An
\emph{\(S\)\nb-modification} between two \(S\)\nb-transformations
\((\Bisp[Y]^1,V^1_g)\) and \((\Bisp[Y]^2,V^2_g)\) from
\((\Gr^1,\Bisp^1_g)\) to \((\Gr^2,\Bisp^2_g)\) is a continuous map of
groupoid correspondences \(W\colon \Bisp[Y]^1 \to \Bisp[Y]^2\)
that intertwines the isomorphisms \(V^1_g\) and~\(V^2_g\) for \(g\in
S\), that is, the following diagram commutes:
\[
\begin{tikzpicture}[baseline=(current bounding box.west)]
  \matrix[cd, row sep=4ex, column sep = 3em] (m) {
    \Bisp_g^2 \Grcomp \Bisp[Y]^1 &
    \Bisp[Y]^1\Grcomp \Bisp_g^1 \\
    \Bisp_g^2 \Grcomp \Bisp[Y]^2 &
    \Bisp[Y]^2\Grcomp \Bisp_g^1 \\
  };
  \draw[dar] (m-1-1) -- node {\(\scriptstyle V_g^1\)} (m-1-2);
  \draw[dar] (m-1-1) -- node[swap] {\(\scriptstyle 1\Grcomp W\)} (m-2-1);
  \draw[dar] (m-2-1) -- node {\(\scriptstyle V_g^2\)} (m-2-2);
  \draw[dar] (m-1-2) -- node {\(\scriptstyle W\Grcomp 1\)} (m-2-2);
\end{tikzpicture}
\]

The bicategory structure on the above objects, arrows and
\(2\)\nb-arrows is defined in the obvious way.  Let~\(\Grcat^S\) denote
this bicategory.

\begin{proposition}
  \label{pro:free_monoid_bicategory}
  Restriction to the generators defines an equivalence of bicategories
  \(\Grcat^{\Cat} \cong\Grcat^S\).
\end{proposition}

\begin{proof}
  Restriction means that we restrict the data of diagrams and strong
  transformations to the unique object and to the generators~\(S\)
  among the arrows.  This is a strict homomorphism of bicategories
  by the definition of the compositions in \(\Grcat^{\Cat}\)
  and~\(\Grcat^S\).  By the Whitehead Theorem for Bicategories (see
  \cite{Johnson-Yau:2-Dim}*{Chapter~7}), this is an equivalence of
  bicategories if the following more precise statements hold:
  \begin{enumerate}
    \item any \(S\)\nb-diagram extends to a diagram \(\Cat\to\Grcat\);
    \item for two diagrams
      \(F^1,F^2\colon \Cat\rightrightarrows \Grcat\), any
      \(S\)\nb-transformation between their restrictions to~\(S\)
      extends \emph{uniquely} to a strong transformation
      \(F^1\to F^2\);
    \item for two diagrams
      \(F^1,F^2\colon \Cat\rightrightarrows \Grcat\) and strong
      transformations
      \(\Phi^1,\Phi^2\colon F^1\rightrightarrows F^2\), any
      \(S\)\nb-modification between the restrictions of \(\Phi^1\)
      and~\(\Phi^2\) to~\(S\) is also a modification
      \(\Phi^1\Rightarrow \Phi^2\).
  \end{enumerate}
  To prove~(1), we write an arbitrary arrow~\(g\) in~\(\Cat\) in the
  unique way as a word in the generators, \(g = a_1 a_2 \dotsm
  a_\ell\) with \(\ell\in\N\), \(a_1,\dotsc,a_\ell\in S\).  Then we
  let \(\Bisp_g\) be the composite of the groupoid
  correspondences~\(\Bisp_{a_i}\) for \(i=1,\dotsc,\ell\).  Formally,
  this depends on the way we put parentheses in the composition.
  Since we work in a bicategory, any two parallel maps built out of
  associators are equal by MacLane's Coherence Theorem.  Therefore,
  different ways of putting parentheses
  are linked by a unique isomorphism built out of associators.  So the
  choice of the parentheses does not matter.  The
  composition in~\(\Cat\) is the concatenation of words.  By
  definition, there are canonical multiplication maps \(\Bisp_g \Grcomp
  \Bisp_h \congto \Bisp_{g h}\) for \(g,h\in\Cat\).  More precisely,
  this isomorphism is the unique one constructed out of associators.  The
  correspondence~\(\Bisp_1\) is the identity correspondence and the
  multiplication maps \(\mu_{1,g}\) and~\(\mu_{g,1}\) are the unit
  transformations by construction.  The
  diagram~\eqref{eq:coherence_category-diagram} commutes for all
  triples of composable arrows because the maps either way are built
  out of associators.  Hence the data above is a \(\Cat\)\nb-shaped
  diagram.  It extends the given \(S\)\nb-diagram by construction.

  Let \((\Gr,\Bisp_g,\mu_{g,h})\) be any diagram.  The
  multiplication maps~\(\mu\) provide a unique isomorphism
  between~\(\Bisp_g\) for a word \(g = a_1 \dotsm a_\ell\) and the
  composite of the groupoid correspondences \(\Bisp_{a_i}\),
  \(i=1,\dotsc,\ell\).  The coherence condition for a
  diagram~\eqref{eq:coherence_category-diagram} ensures that this
  isomorphism does not depend on how we put parentheses in the
  decomposition of~\(g\).  Now consider two diagrams
  \((\Gr^i,\Bisp_g^i,\mu_{g,h}^i)\) for \(i=1,2\) and
  \(S\)\nb-transformation \((\Bisp[Y],V_a)\) between their
  restrictions to~\(S\).  Then we may well define isomorphisms
  \(\Bisp^2_g \Grcomp\Bisp[Y] \congto \Bisp[Y] \Grcomp\Bisp_g^1\)
  for all \(g\in \Cat\) by first identifying~\(\Bisp_g\) with a
  composite of~\(\Bisp_a\) for generators~\(a\) and then composing
  the isomorphisms in the \(S\)\nb-transformation for each
  factor~\(\Bisp_a\).  For instance, for the word~\(a_1 a_2\) of
  length~\(2\), the isomorphism is
  \[
    \Bisp^2_{a_1 a_2} \Grcomp \Bisp[Y]
    \cong \Bisp^2_{a_1} \Grcomp \Bisp^2_{a_2} \Grcomp \Bisp[Y]
    \xrightarrow{\id \Grcomp V_{a_2}}
    \Bisp^2_{a_1} \Grcomp \Bisp[Y] \Grcomp \Bisp^1_{a_2}
    \xrightarrow{V_{a_1}\Grcomp \id}
    \Bisp[Y] \Grcomp \Bisp^1_{a_1} \Grcomp \Bisp^1_{a_2}
    \cong \Bisp[Y] \Grcomp \Bisp^1_{a_1 a_2}.
  \]
  Associators are needed as well to apply~\(V_a\)
  for generators~\(a\) to parts of an expression.  The associators do
  not destroy commuting diagrams by MacLane's Coherence Theorem.
  The isomorphism~\(V_1\) is the canonical one by construction,
  and~\eqref{eq:trafo_category-diagram} commutes because both maps in
  that diagram differ at most by associators.  Hence this defines a
  strong transformation between \(\Cat\)\nb-shaped diagrams.  Moreover, it is
  the only strong transformation that extends the given
  \(S\)\nb-transformation.

  Finally, modifications of \(S\)\nb-diagrams and \(\Cat\)\nb-shaped
  diagrams are determined by the same data, namely, a map of
  groupoid correspondences
  \(\Bisp[Y]^1 \to \Bisp[Y]^2\).  A routine computation shows that
  this satisfies the coherence conditions
  \eqref{eq:modification_category-diagram} for all elements
  \(g\in\Cat\) once this holds for generators \(a\in S\).
\end{proof}

In particular, for \(\Cat=(\N,+)\), a \(\Cat\)\nb-shaped diagram of
groupoid correspondences is equivalent to a groupoid~\(\Gr\) with one
self-correspondence \(\Bisp\colon \Gr\leftarrow\Gr\); this type
of diagrams has been considered most frequently.  A
strong transformation between such pairs \((\Gr^i,\Bisp^i)\), \(i=1,2\), is a
groupoid correspondence \(\Bisp[Y]\colon \Gr^2\leftarrow\Gr^1\) with
an isomorphism \(V\colon \Bisp^2\Grcomp\Bisp[Y] \congto
\Bisp[Y]\Grcomp \Bisp^1\).  A modification between two
strong transformations \((\Bisp[Y]^i,V^i)\), \(i=1,2\), is a map of
correspondences \(W\colon \Bisp[Y]^1 \to \Bisp[Y]^2\) that
intertwines \(V^1\) and~\(V^2\).

\subsection{Path categories of directed graphs}
\label{sec:path_categories}

Let~\(\Gamma\) be a directed graph, consisting of sets \(E\)
and~\(V\) of edges and vertices with maps
\(\rg,\s\colon E\rightrightarrows V\) that send an edge to its range
and source, respectively.  Let~\(\Cat_\Gamma\) be the path category
of~\(\Gamma\); that is, \(\Cat_\Gamma^0=V\) and~\(\Cat_\Gamma^1\)
consists of paths in~\(\Gamma\), with concatenation of paths as
composition.  We now extend the description of diagrams for free
monoids in \longref{Section}{sec:free_monoid_action} to path
categories, where~\(E\) replaces the set of generators.

A \emph{\(\Gamma\)\nb-diagram}
of groupoid correspondences consists of groupoids~\(\Gr_v\)
for \(v\in V\)
and groupoid correspondences
\(\Bisp_e\colon \Gr_{\rg(e)}\leftarrow\Gr_{\s(e)}\)
for \(e\in E\).
A \emph{\(\Gamma\)\nb-transformation}
between two such \(\Gamma\)\nb-diagrams
\((\Gr^i_v,\Bisp^i_e)\),
\(i=1,2\),
consists of groupoid correspondences
\(\Bisp[Y]_v\colon \Gr^2_v\leftarrow \Gr^1_v\)
for \(v\in V\)
and isomorphisms of groupoid correspondences
\(V_e\colon \Bisp^2_e \Grcomp \Bisp[Y]_{\s(e)} \congto
\Bisp[Y]_{\rg(e)} \Grcomp \Bisp^1_e \)
for all \(e\in E\).
A \emph{\(\Gamma\)\nb-modification}
between two such \(\Gamma\)\nb-transformations
\((\Bisp[Y]_v^i,V_e^i)\),
\(i=1,2\),
consists of maps of groupoid correspondences
\(W_v\colon \Bisp[Y]_v^1 \to \Bisp[Y]_v^2\)
that intertwine the isomorphisms \(V_e^1\)
and~\(V_e^2\)
for \(e\in E\).
Let~\(\Grcat^\Gamma\)
by the bicategory with \(\Gamma\)\nb-diagrams,
\(\Gamma\)\nb-transformations
and \(\Gamma\)\nb-modifications
as objects, arrows and \(2\)\nb-arrows,
respectively, with the obvious composition of arrows and horizontal
and vertical composition of \(2\)\nb-arrows.

\begin{proposition}
  \label{pro:path_category_bicategory}
  Restriction to the generators is an equivalence of bicategories
  \(\Grcat^{\Cat_\Gamma} \cong\Grcat^\Gamma\).
\end{proposition}

More precisely, any \(\Gamma\)\nb-diagram extends to a diagram
\(\Cat_\Gamma\to\Grcat\); any \(\Gamma\)\nb-\alb{}transformation
between the restrictions of two diagrams \(\Cat_\Gamma\to\Grcat\)
extends uniquely to a strong transformation between \(\Cat_\Gamma\)\nb-shaped
diagrams; and restriction to~\(\Gamma\) is a bijection from
\(\Cat_\Gamma\)\nb-modifications to \(\Gamma\)\nb-modifications.  The
proof is
the same as for \longref{Proposition}{pro:free_monoid_bicategory}.

For instance, let
\(\Gamma=* \leftarrow * \rightarrow *\).
A \(\Gamma\)\nb-diagram
in~\(\Grcat\)
consists of three groupoids \(\Gr_1,\Gr_2,\Gr[H]\)
that are linked by groupoid correspondences
\[
\Gr_1\xleftarrow{\Bisp_1} \Gr[H] \xrightarrow{\Bisp_2} \Gr_2.
\]
Or let
\(\Gamma=* \rightrightarrows *\).  Then
a \(\Gamma\)\nb-diagram
in~\(\Grcat\)
consists of two groupoids \(\Gr\) and~\(\Gr[H]\)
that are linked by two parallel groupoid correspondences
\[
  \begin{tikzcd}
    \Gr[H] \ar[r, shift left, "\Bisp_1"]
    \ar[r, shift right, "\Bisp_2"'] &
    \Gr.
  \end{tikzcd}
\]
This is an equaliser diagram.  We will study diagrams of this shape
in \longref{Section}{sec:universal_nm}.

\subsection{Free commutative monoid actions}
\label{sec:free_commutative_monoid_action}

Let~\(\Cat\) be the free \emph{commutative} monoid on a set of
generators~\(S\).  We are going to describe \(\Cat\)\nb-shaped
diagrams using the generating set~\(S\).

A \emph{commutative \(S\)\nb-diagram} consists of a groupoid~\(\Gr\),
groupoid correspondences~\(\Bisp_g\) for \(g\in S\), and isomorphisms
of groupoid correspondences \(\Sigma_{g,h} \colon \Bisp_g \Grcomp
\Bisp_h \congto \Bisp_h \Grcomp \Bisp_g\) for \(g,h\in S\), such that
\(\Sigma_{g,g}\) is the identity for all \(g\in S\),
\(\Sigma_{g,h}\Sigma_{h,g}\) is the identity for all \(g,h\in S\), and
the following diagrams commute for all \(g,h,k\in S\):
\begin{equation}
  \label{eq:braiding_hexagon}
  \begin{tikzpicture}[baseline=(current bounding box.west)]
    \matrix[cd, row sep=4ex, column sep = 3em] (m) {
      & \Bisp_g \Bisp_k \Bisp_h &&
      \Bisp_k \Bisp_g \Bisp_h \\
      \Bisp_g \Bisp_h \Bisp_k && &&
      \Bisp_k \Bisp_h \Bisp_g \\
      & \Bisp_h \Bisp_g \Bisp_k &&
      \Bisp_h \Bisp_k \Bisp_g \\
    };
    \begin{scope}[cdar]
      \draw (m-1-2) -- node {\(\scriptstyle \Sigma_{g,k}\Grcomp 1\)} (m-1-4);
      \draw (m-1-4) -- node {\(\scriptstyle 1\Grcomp\Sigma_{g,h}\)} (m-2-5);
      \draw (m-3-4) -- node[swap] {\(\scriptstyle \Sigma_{h,k}\Grcomp 1\)} (m-2-5);
      \draw (m-2-1) -- node {\(\scriptstyle 1\Grcomp \Sigma_{h,k}\)} (m-1-2);
      \draw (m-2-1) -- node[swap] {\(\scriptstyle \Sigma_{g,h}\Grcomp 1\)} (m-3-2);
      \draw (m-3-2) -- node[swap] {\(\scriptstyle 1\Grcomp\Sigma_{g,k}\)} (m-3-4);
    \end{scope}
  \end{tikzpicture}
\end{equation}
Here we dropped the \(\Grcomp\)\nb-sign for composition of groupoid
correspondences.

A \emph{commutative \(S\)\nb-transformation} between two such diagrams
\((\Gr^i,\Bisp^i_g,\Sigma^i_{g,h})\) for \(i=1,2\), consists of a
groupoid correspondence \(\Bisp[Y]\colon \Gr^2\leftarrow \Gr^1\) and
isomorphisms \(V_g\colon \Bisp^2_g \Grcomp \Bisp[Y] \congto
\Bisp[Y]\Grcomp \Bisp^1_g\), such that the following diagrams commute
for all \(g,h\in S\):
\begin{equation}
  \label{eq:trafo_commutative_S-diagram}
  \begin{tikzpicture}[baseline=(current bounding box.west)]
    \matrix[cd, row sep=4ex, column sep = 3em] (m) {
      & \Bisp^2_g \Bisp[Y] \Bisp^1_h &&
      \Bisp[Y] \Bisp^1_g \Bisp^1_h \\
      \Bisp^2_g \Bisp^2_h \Bisp[Y] && &&
      \Bisp[Y] \Bisp^1_h \Bisp^1_g \\
      & \Bisp^2_h \Bisp^2_g \Bisp[Y]  &&
      \Bisp^2_h \Bisp[Y] \Bisp^1_g \\
    };
    \begin{scope}[cdar]
      \draw (m-1-2) -- node {\(\scriptstyle V_g\Grcomp 1\)} (m-1-4);
      \draw (m-1-4) -- node {\(\scriptstyle 1\Grcomp\Sigma^1_{g,h}\)} (m-2-5);
      \draw (m-3-4) -- node[swap] {\(\scriptstyle V_h\Grcomp 1\)} (m-2-5);
      \draw (m-2-1) -- node {\(\scriptstyle 1\Grcomp V_h\)} (m-1-2);
      \draw (m-2-1) -- node[swap] {\(\scriptstyle \Sigma^2_{g,h}\Grcomp 1\)} (m-3-2);
      \draw (m-3-2) -- node[swap] {\(\scriptstyle 1\Grcomp V_g\)} (m-3-4);
    \end{scope}
  \end{tikzpicture}
\end{equation}
A \emph{commutative \(S\)\nb-modification} between two such
commutative \(S\)\nb-transformations \((\Bisp[Y]^i,V_g^i)\),
\(i=1,2\), is a continuous map of groupoid correspondences \(W\colon
\Bisp[Y]^1 \to \Bisp[Y]^2\) that intertwines the isomorphisms
\(V^1_g\) and~\(V^2_g\) for all \(g\in S\).  This data also defines a
bicategory in an obvious way.  We claim that this bicategory is
equivalent to~\(\Gr^{\Cat}\), where~\(\Cat\) is the free commutative
monoid on~\(S\).  The equivalence restricts a \(\Cat\)\nb-shaped
diagram to~\(S\) and lets~\(\Sigma_{g,h}\) be the canonical
isomorphism
\[
\Bisp_g \Grcomp \Bisp_h \cong \Bisp_{g h} = \Bisp_{h g} \cong
\Bisp_h \Grcomp \Bisp_g.
\]
A direct computation shows that these maps
satisfy~\eqref{eq:braiding_hexagon} and the other conditions for a
commutative \(S\)\nb-diagram.  For strong transformations and modifications,
the restriction is the obvious one, and direct computations show that
this satisfies~\eqref{eq:trafo_commutative_S-diagram}.

Now totally order the set of generators.  It suffices to take
only~\(\Sigma_{g,h}\) for \(g<h\): the others are determined because
\(\Sigma_{g,g}=\id\) and \(\Sigma_{h,g} = \Sigma_{g,h}^{-1}\).
And~\eqref{eq:braiding_hexagon} commutes for all triples~\(g,h,k\),
once it does for \(g<h<k\).  Thus a commutative \(S\)\nb-diagram is
given by a groupoid~\(\Gr\), groupoid correspondences~\(\Bisp_g\) for
\(g\in S\), and isomorphisms \(\Sigma_{g,h}\colon \Bisp_g\circ\Bisp_h
\congto \Bisp_h \circ\Bisp_g\) for \(g,h\in S\) with \(g<h\) such
that~\eqref{eq:braiding_hexagon} commutes for \(g<h<k\).  And in the
definition of a commutative \(S\)\nb-transformation, it suffices to
require~\eqref{eq:trafo_commutative_S-diagram} to commute for \(g,h\in
S\) with \(g<h\).

Analogous statements to those in
\longref{Section}{sec:free_monoid_action} hold: (1)~any commutative
\(S\)\nb-diagram extends to a \(\Cat\)\nb-shaped diagram, (2)~any
commutative \(S\)\nb-transformation between restrictions to~\(S\)
extends uniquely to a strong transformation of \(\Cat\)\nb-shaped
diagrams, and (3)~restriction gives a bijection on the level of
modifications and commutative \(S\)\nb-modifications.  Statement~(1)
is proved in a different language in
\cite{Fowler-Sims:Product_Artin}*{Theorems 2.1--2}, and
\cite{Fowler-Sims:Product_Artin}*{Proposition 2.8} almost
proves~(2), except that the groupoid correspondence~\(\Bisp[Y]\) is
taken to be the identity correspondence.  We omit further details.

For a single generator, the free monoid~\((\N,+)\) is also the free
Abelian monoid, and we get the same statement as
in \longref{Section}{sec:free_monoid_action} for that case.  If \(S=\{1,2\}\)
with
\(1<2\), then there are no \(g,h,k\in S\) with \(g<h<k\), so that the
condition~\eqref{eq:braiding_hexagon} is trivial.  Hence any
isomorphism~\(\Sigma_{1,2}\) may be chosen.

The above discussion describes, in particular, actions of the free
commutative monoids~\((\N^k,+)\) on spaces and groups by topological
correspondences or self-similarities of groups, respectively.  The
first leads to the data of a higher-rank topological graph as
in~\cite{Yeend:Groupoid_models}, with arbitrary spaces instead of
locally compact spaces.  As far as we know, several commuting
self-similarities of groups have not yet been studied in general.

\section{Actions of diagrams and groupoid models}
\label{sec:actions}

In this section, we define the groupoid model of a diagram of
groupoid correspondences.  This is the main concept of this article.
We also prove some basic properties of groupoid models and
illustrate them by three examples.  The definition is based on the
category of actions of a diagram on topological spaces.  The first
two examples are rather trivial, namely, disjoint unions of
groupoids and groupoids graded by a group.  The third example is
much more interesting and leads to the \((m,n)\)-dynamical systems
studied by
Ara--Exel--Katsura~\cite{Ara-Exel-Katsura:Dynamical_systems}.

\subsection{Actions of correspondences and diagrams}
\label{sec:act_corr}

\begin{definition}
  \label{gpcorraction1}
  Let~\(\Gr\) be a groupoid, let \(\Bisp \colon \Gr \leftarrow \Gr\) be
  a groupoid correspondence, and let~\(Y\) be a topological space.
  An action of~\(\Bisp\) on~\(Y\) consists of
  \begin{itemize}
  \item a left \(\Gr\)\nb-action on~\(Y\); let
    \(\rg_Y\colon Y \to \Gr^0\) be its anchor map;
  \item an open, continuous, surjective map
    \(\alpha\colon \Bisp \times_{\s,\Gr^0,\rg_Y} Y \to Y\), denoted
    multiplicatively as \(\alpha(\gamma,y) = \gamma\cdot y\) for
    \(\gamma \in \Bisp\) and \(y \in Y\);
  \end{itemize}
  such that
  \begin{enumerate}[label=\textup{(\ref*{gpcorraction1}.\arabic*)},
    leftmargin=*,labelindent=0em]
  \item \label{en:gpaction1b}%
    \(x \cdot (g \cdot y) = (x \cdot g) \cdot y\) for all
    \(x\in \Bisp\), \(g\in\Gr\), \(y\in Y\) with \(\s(x) = \rg(g)\)
    and \(\s(g) = \rg_Y(y)\);
  \item \label{en:gpaction1}%
    \(\rg(x\cdot y) = \rg(x)\) for any \(x\in \Bisp\) and \(y \in Y\)
    with \(\s(x) = r_Y(y)\), and
    \(g \cdot (x \cdot y) = (g \cdot x) \cdot y\) for all \(g \in \Gr\),
    \(x\in \Bisp\), \(y\in Y\) with \(\s(g) = \rg(x)\) and
    \(\s(x) = r_Y(y)\);
  \item \label{en:gpaction2}%
    if \(\gamma\cdot y = \gamma'\cdot y'\) for
    \(\gamma,\gamma'\in \Bisp\) and \(y,y'\in Y\), then there is
    \(\eta\in \Gr\) with \(\gamma' = \gamma\cdot \eta\) and
    \(y = \eta\cdot y'\); equivalently, \(\Qu(\gamma)=\Qu(\gamma')\)
    for the orbit space projection \(\Qu\colon \Bisp\to \Bisp/\Gr\)
    --~so that \(\braket{\gamma}{\gamma'}\) is defined~-- and
    \(y = \braket{\gamma}{\gamma'} y'\).
  \end{enumerate}
\end{definition}

\begin{lemma}
  \label{homeo}
  The multiplication map~\(\alpha\) descends to a \(\Gr\)-equivariant
  homeomorphism \(\dot\alpha\colon \Bisp \Grcomp_{\Gr} Y \to Y\).
\end{lemma}

\begin{proof}
  Condition~\ref{en:gpaction1b} implies
  \(xh^{-1} \cdot hy = (xh^{-1}\cdot h) \cdot y = x \cdot y\) for any
  \(x \in \Bisp\), \(h \in \Gr\), \(y \in Y\) with \(\s(x) = \rg(h)\) and
  \(\s(h) = r_Y(y)\).  Therefore, all the elements in the same class
  \([x, y] \in \Bisp \Grcomp_{\Gr} Y\) have the same image
  under~\(\alpha\).  This gives a well defined map
  \(\dot\alpha\colon \Bisp \Grcomp_{\Gr} Y \to Y\).  Let \(k \in \Gr\)
  and \([x, y] \in \Bisp \Grcomp_{\Gr} Y\).  Then~\ref{en:gpaction1}
  implies
  \(\dot\alpha(k \cdot [x, y]) = \dot\alpha[kx, y] = (k \cdot x)
  \cdot y = k \cdot (x \cdot y) = k \cdot \dot\alpha[x, y]\).  This
  shows that~\(\dot\alpha\) is \(\Gr\)\nb-equivariant.  Next, we
  show that~\(\dot\alpha\) is injective.  Suppose that
  \(x_1 \cdot y_1 = x_2 \cdot y_2\) for some \(x_1, x_2 \in \Bisp\) and
  \(y_1, y_2 \in Y\) with \(\s(x_1) = r_Y(y_1)\) and \(\s(x_2) = r_Y(y_2)\).
  Then~\ref{en:gpaction2} implies \([x_1, y_1] = [x_2, y_2]\), as
  desired.  Since~\(\alpha\) is open, continuous and surjective, so
  is~\(\dot\alpha\).  Being injective as well, it is a homeomorphism.
\end{proof}

To prepare for actions of diagrams of groupoid correspondences, we
first treat groupoid correspondences with different domain and
codomain by reducing them to the special case above.  Let \(\Gr[H]\)
and~\(\Gr\) be étale groupoids.  Their disjoint union
\(\Gr[H] \sqcup \Gr\) is an étale groupoid as well.

\begin{definition}
  \label{gpcorraction2}
  A groupoid correspondence \(\Bisp \colon \Gr[H] \leftarrow \Gr\)
  induces a groupoid correspondence
  \(\Bisp \colon \Gr[H] \sqcup \Gr \leftarrow \Gr[H] \sqcup \Gr\) as
  follows.  The anchor maps \(r \colon \Bisp \to \Gr[H]^0\) and
  \(s \colon \Bisp \to \Gr^0\) compose with the injections
  \(\Gr[H]^0 \hookrightarrow \Gr[H]^0 \sqcup \Gr^0\) and
  \(\Gr^0 \hookrightarrow \Gr[H]^0 \sqcup \Gr^0\), respectively, and
  give anchor maps
  \(r' \colon \Bisp \hookrightarrow \Gr[H]^0 \sqcup \Gr^0\) and
  \(s' \colon \Bisp \hookrightarrow \Gr[H]^0 \sqcup \Gr^0\).  Let
  \(x\in\Bisp\) and \(\gamma\in \Gr[H] \sqcup \Gr\).  Then
  \(\gamma\cdot x\) is only defined if \(\gamma\in\Gr[H]\) and
  \(x\cdot \gamma\) is only defined if \(\gamma\in\Gr\).  Therefore,
  the left and right actions on~\(\Bisp\) already make~\(\Bisp\) into a
  bispace over \(\Gr[H] \sqcup \Gr\).  This inherits the properties of
  a groupoid correspondence.
\end{definition}

\begin{definition}
  An action of a groupoid correspondence
  \(\Bisp \colon \Gr[H] \leftarrow \Gr\) of étale groupoids \(\Gr[H]\) and
  \(\Gr\) on a topological space~\(Y\) is an action on~\(Y\) of the
  induced groupoid correspondence
  \(\Bisp \colon \Gr[H] \sqcup \Gr \leftarrow \Gr[H] \sqcup \Gr\).
\end{definition}

More concretely, the anchor map \(Y\to \Gr^0 \sqcup \Gr[H]^0\) gives
us a decomposition \(Y = Y_{\Gr} \sqcup Y_{\Gr[H]}\) into clopen
(that is, closed, open) subsets with continuous maps
\(Y_{\Gr} \to \Gr^0\) and \(Y_{\Gr[H]} \to \Gr[H]^0\), and the
action of \(\Gr\sqcup \Gr[H]\) is equivalent to an action of~\(\Gr\)
on~\(Y_{\Gr}\) and an action of~\(\Gr[H]\) on~\(Y_{\Gr[H]}\).  Then
the correspondence~\(\Bisp\) acts by a continuous, open, surjective
map \(\alpha\colon \Bisp_{\s,\Gr^0} \Grcomp Y_{\Gr} \to Y_{\Gr[H]}\) with
the algebraic properties in \longref{Definition}{gpcorraction1}.

Now we may move on to the actions of diagrams of groupoid
correspondences.  Let~\(\Cat\) be a category.  Describe a diagram
\(F\colon \Cat\to\Grcat\) by \((\Gr_x,\Bisp_g,\mu_{g,h})\) as in
\longref{Proposition}{pro:diagrams_in_Grcat}.  We write the
maps~\(\mu_{g,h}\) multiplicatively as
in~\eqref{eq:diagram_as_mult}.

\begin{definition}
  \label{def:diagram_dynamical_system}
  Let~\(Y\) be a space.  An \emph{\(F\)\nb-action} on~\(Y\) consists
  of
  \begin{itemize}
  \item a partition \(Y = \bigsqcup_{x\in\Cat^0} Y_x\) into clopen
    subsets;
    \item continuous maps \(\rg\colon Y_x \to \Gr_x^0\);
    \item open, continuous, surjective maps \(\alpha_g\colon \Bisp_g
    \times_{\s,\Gr_x^0,\rg} Y_x \to Y_{x'}\) for arrows \(g\colon
    x'\leftarrow x\) in~\(\Cat\), denoted multiplicatively as
    \(\alpha_g(\gamma,y) = \gamma\cdot y\);
  \end{itemize}
  such that
  \begin{enumerate}[label=\textup{(\ref*{def:diagram_dynamical_system}.\arabic*)},
    leftmargin=*,labelindent=0em]
    \item \label{en:diagram_dynamical_system1}
    \(\rg(\gamma_2\cdot y) = \rg(\gamma_2)\) and \(\gamma_1\cdot
    (\gamma_2\cdot y) = (\gamma_1 \cdot \gamma_2)\cdot y\) for
    composable arrows \(g_1,g_2\) in~\(\Cat\), \(\gamma_1\in
    \Bisp_{g_1}\), \(\gamma_2\in \Bisp_{g_2}\), and \(y\in
    Y_{\s(g_2)}\) with \(\s(\gamma_1) = \rg(\gamma_2)\),
    \(\s(\gamma_2) = \rg(y)\);
    \item \label{en:diagram_dynamical_system2}
    if \(\gamma\cdot y = \gamma'\cdot y'\) for \(\gamma,\gamma'\in
    \Bisp_g\), \(y,y'\in Y_{\s(g)}\), there is \(\eta\in \Gr_{\s(g)}\)
    with \(\gamma' = \gamma\cdot \eta\) and \(y = \eta\cdot y'\);
    equivalently, \(\Qu(\gamma)=\Qu(\gamma')\) for the orbit space
    projection \(\Qu\colon \Bisp_g \to \Bisp_g/\Gr_{\s(g)}\)
    and \(y = \braket{\gamma}{\gamma'} y'\).
  \end{enumerate}
\end{definition}

\begin{lemma}
  \label{lem:action_pieces}
  The multiplication map~\(\alpha_x\) is a left \(\Gr_x\)\nb-action
  on~\(Y_x\), where~\(x\) denotes the identity morphism on \(x \in \Cat^0\).
\end{lemma}

\begin{proof}
  Pick \(g_1 = g_2 = x\) in~\ref{en:diagram_dynamical_system1}.  We
  have \(\rg(\gamma\cdot y) = \rg(\gamma)\) for any
  \(\gamma \in \Bisp_x = \Gr_x\) and \(y \in Y_x\) with
  \(\s(\gamma) = \rg(y)\), and also
  \(\gamma_1\cdot (\gamma_2\cdot y) = (\gamma_1 \cdot \gamma_2)\cdot
  y\) for any \(\gamma_1\), \(\gamma_2 \in \Gr_x\) with
  \(\s(\gamma_1) = \rg(\gamma_2)\) and \(\s(\gamma_2) = \rg(y)\).
  Note that \(\gamma = \gamma \cdot \rg(y)\) for any
  \(\gamma \in \Gr_x\) with \(\s(\gamma) = \rg(y)\).  Therefore,
  \(\gamma \cdot y = \gamma \cdot \rg(y) \cdot y\) for any
  \(\gamma \in \Gr_x\) such that \(\s(\gamma) = \rg(y)\).  Picking
  \(\gamma'= \gamma \cdot \rg(y)\) in
  \ref{en:diagram_dynamical_system2} shows that
  \(\rg(y) \cdot y = y\) for any \(y \in Y_x\).  In addition, \(r\)
  and~\(\alpha_x\) are continuous by
  \longref{Definition}{def:diagram_dynamical_system}.
\end{proof}

\begin{lemma}
  The multiplication map~\(\alpha_g\) descends to a
  \(\Gr_{\rg(g)}\)\nb-equivariant homeomorphism
  \(\dot\alpha_g\colon \Bisp_g \Grcomp_{\Gr_{\s(g)}} Y_{\s(g)} \to
  Y_{\rg(g)}\).
\end{lemma}

\begin{proof}
  This follows from \longref{Lemma}{homeo} by turning
  each~\(\Bisp_g\) into a groupoid correspondence on
  \(\Gr_{\rg(g)} \sqcup \Gr_{\s(g)}\) as in
  \longref{Definition}{gpcorraction2} and then turning~\(\alpha_g\)
  into an action of the latter on the space
  \(Y_{\rg(g)} \sqcup Y_{\s(g)}\);
  \longref{Lemma}{lem:action_pieces} gives the actions of
  \(\Gr_{\rg(g)}\) and~\(\Gr_{\s(g)}\) on \(Y_{\rg(g)}\)
  and~\(Y_{\s(g)}\), respectively.
\end{proof}

\begin{definition}
  \label{def:F-invariant}
  A continuous map \(\varphi\colon Y\to Y'\) between two spaces with
  \(F\)\nb-actions is \emph{\(F\)\nb-equivariant} if
  \(\varphi(Y_x) \subseteq Y'_x\) for all \(x\in \Cat^0\), and
  \(\rg(\varphi(y))=\rg(y)\) and
  \(\varphi(\gamma\cdot y) = \gamma\cdot \varphi(y)\) for all
  \(g\in\Cat\), \(y\in Y_{\s(g)}\) and \(\gamma\in \Bisp_g\) with
  \(\s(\gamma)=\rg(y)\).

  Let \(Y\) and~\(Z\) be spaces and let~\(Y\) carry an
  \(F\)\nb-action.  A continuous map \(\varphi\colon Y\to Z\) is
  \emph{\(F\)\nb-invariant} if
  \(\varphi(\gamma\cdot y) = \varphi(y)\) for all \(g\in\Cat\),
  \(\gamma\in \Bisp_g\), \(y\in Y_{\s(g)}\) with
  \(\s(\gamma)=\rg(y)\).
\end{definition}

\begin{lemma}
  \label{lem:action_one_correspondence}
  A diagram \(F\colon (\N,+) \to \Grcat\) is completely described by
  a groupoid~\(\Gr\) and a single groupoid correspondence~\(\Bisp\)
  from~\(\Gr\) to itself.  An action of~\(F\) is the same as an
  action of the groupoid correspondence~\(\Bisp\) as in
  \longref{Definition}{gpcorraction1}.
\end{lemma}

\begin{proof}
  Since the category of the monoid~\((\N,+)\) has a single
  object~\(*\), the partition in an \(F\)\nb-action is trivial,
  \(Y=Y_*\).  The groupoid~\(\Gr\) acts on~\(Y\), and there is a
  further map \(\Bisp\times_{\s,\Gr^0,\rg} Y\to Y\),
  \((\gamma,y)\mapsto \gamma\cdot y\), describing the action of
  \(\Bisp=\Bisp_1\).  This is an open, continuous, surjective map
  that satisfies property~\ref{en:diagram_dynamical_system2} above.
  Furthermore, \(\rg(\gamma\cdot y) = \rg(\gamma)\) for all
  \((\gamma,y)\in \Bisp\times_{\s,\Gr^0,\rg} Y\), and
  \(\gamma_1\cdot (\gamma_2\cdot y) = (\gamma_1\cdot \gamma_2)\cdot
  y\) if
  \((\gamma_1,\gamma_2,y) \in \Gr\times_{\s,\Gr^0,\rg} \Bisp
  \times_{\s,\Gr^0,\rg} Y\) or
  \((\gamma_1,\gamma_2,y) \in \Bisp\times_{\s,\Gr^0,\rg} \Gr
  \times_{\s,\Gr^0,\rg} Y\).  Then there is a unique way to define
  the multiplication maps \(\Bisp_n\times_{\s,\Gr^0,\rg} Y \to Y\)
  for \(n\ge2\) so that~\ref{en:diagram_dynamical_system1} holds,
  namely,
  \[
  [\gamma_1,\gamma_2,\dotsc,\gamma_n]\cdot y \defeq
  \gamma_1 \cdot (\gamma_2 \cdot (\dotsb \cdot (\gamma_n\cdot y)))
  \]
  for \(\gamma_1,\dotsc,\gamma_n\in\Bisp\), \(y\in Y\) with
  \(\s(\gamma_i) = \rg(\gamma_{i+1})\) for \(i=1,\dotsc,n-1\) and
  \(\s(\gamma_n) = \rg(y)\).  It is routine to check that these maps
  have the properties required in
  \longref{Definition}{def:diagram_dynamical_system}.  Thus we have
  simplified an action of a diagram \(\N\to\Grcat\) to an action of
  the single correspondence~\(\Bisp_1\).
\end{proof}

\begin{lemma}
  \label{lem:action_self-similar_group}
  Let~\(\Gr\) be a group and build a groupoid correspondence
  on~\(\Gr\) from a discrete \(\Gr\)\nb-set~\(A\) with a
  \(1\)\nb-cocycle \(\Gr\times A\to \Gr\), \((g,x)\mapsto g|_x\).
  Extend this to a diagram \(F\colon (\N,+)\to \Grcat\).  Then an
  \(F\)\nb-action on a space~\(Y\) is equivalent to a
  \(\Gr\)\nb-action on~\(Y\) together with a homeomorphism
  \(\Psi\colon A\times Y\congto Y\)
  satisfying
  \begin{equation}
    \label{eq:self-similar_action}
    g\cdot \Psi(a, y) = \Psi(g\cdot a,g|_a\cdot y)
  \end{equation}
  for all \(g\in\Gr\), \(a\in A\), \(y\in Y\).
\end{lemma}

\begin{proof}
  As in the previous lemma, an action of~\(F\) on~\(Y\) simplifies
  to an action of the single groupoid correspondence
  \(\Bisp_1\colon \Gr \leftarrow \Gr\) on~\(Y\) built from the given
  \(\Gr\)\nb-set and \(1\)\nb-cocycle as in
  \cite{Antunes-Ko-Meyer:Groupoid_correspondences}*{Example~4.2}.
  That is, \(\Bisp_1 = A\times\Gr\) with the obvious right
  \(\Gr\)\nb-action and
  \(g\cdot (a,h) \defeq (g\cdot a,g|_a\cdot h)\).  We now simplify
  the action of~\(\Bisp_1\) on~\(Y\) to a ``self-similar action''
  of~\(\Gr\).  First, the range map \(Y\to \Gr^0\) becomes vacuous.
  The multiplication map \(\Bisp_1\times Y \to Y\) must map
  \(((a,g),y)\mapsto (a,1)\cdot (g\cdot y)\)
  by~\ref{en:diagram_dynamical_system1}.  So it is equivalent to a
  map \(\Psi\colon A\times Y\to Y\), \((a,y)\mapsto a\cdot y\),
  defined by \(a\cdot y\defeq (a,1)\cdot y\).  The corresponding
  multiplication map \(\Bisp_1\times Y\to Y\) maps
  \((a,g)\cdot y\mapsto a\cdot g\cdot y\).

  The multiplication map \(\Bisp_1\times Y \to Y\) is open,
  continuous and surjective if and only if~\(\Psi\) has these
  properties.  And property~\ref{en:gpaction2} says exactly
  that~\(\Psi\) is injective.  Thus it is a homeomorphism.  In fact,
  it is the homeomorphism
  \(\dot\alpha_1\colon \Bisp_1\Grcomp_{\Gr} Y\congto Y\) in
  \longref{Lemma}{homeo}.  The property~\ref{en:gpaction1b} is built
  into our Ansatz that \((a,g)\cdot y= \Psi(a, g\cdot y)\).  The
  property~\ref{en:gpaction1} is equivalent
  to~\eqref{eq:self-similar_action}.
\end{proof}  

Since~\(A\) is discrete, we may replace~\(\Psi\) above by a bunch of
partial homeomorphisms on~\(Y\).  Namely, let
\(Y_a\defeq \setgiven{\Psi(a,y)}{y\in Y} \subseteq Y\).  These
subsets of~\(Y\) form a clopen partition, that is,
\(Y = \bigsqcup_{a\in A} Y_a\).  The map~\(\Psi\) is equivalent to a
family of homeomorphisms \(\psi_a\colon Y\congto Y_a\),
\(\psi_a(y)\defeq \Psi(a,y)\).  The
condition~\eqref{eq:self-similar_action} is equivalent to
\(g\cdot Y_a = Y_{g\cdot a}\) for all \(g\in \Gr\), \(a\in A\) and
\(g\cdot \psi_a(y) = \psi_{g\cdot a}(g|_a\cdot y)\) for all
\(g\in \Gr\), \(a\in A\), \(y\in Y\).

Already the case where the group~\(\Gr\) is trivial is interesting.
Then an \(F\)\nb-action on~\(Y\) is equivalent to a homeomorphism
\(\Psi\colon A\times Y\congto Y\), or to a clopen partition
\(Y= \bigsqcup_{a\in A} Y_a\) together with homeomorphisms
\(\psi_a\colon Y\congto Y_a\) for all \(a\in A\).

\subsection{The groupoid model of a diagram}
\label{sec:groupoid_model}

Let~\(\Cat\) be a small category and let \(F\colon \Cat\to\Grcat\)
be a diagram of groupoid correspondences.  The actions of~\(F\)
on topological spaces together with the \(F\)\nb-equivariant maps
form a concrete category, that is, the product is the usual
composition of maps.  We use this category to define the universal
\(F\)\nb-action and the groupoid model of~\(F\).  Both are closely
related.

\begin{definition}
  \label{def:universal_F-action}
  An \(F\)\nb-action~\(\Omega\) is \emph{universal} if for any
  \(F\)\nb-action~\(Y\), there is a unique \(F\)\nb-equivariant map
  \(Y\to \Omega\).
\end{definition}

In other words, an \(F\)\nb-action is universal if and only if it is
a terminal object in the category of \(F\)\nb-actions.

\begin{definition}
  \label{def:universal_action}
  A \emph{groupoid model for \(F\)\nb-actions} is an étale
  groupoid~\(\Gr[U]\) with natural bijections between the sets of
  \(\Gr[U]\)\nb-actions and \(F\)\nb-actions on~\(Y\) for all
  spaces~\(Y\).
\end{definition}

The naturality of the bijections means that a continuous map
\(Y\to Y'\) is \(\Gr[U]\)\nb-equivariant if and only if it is
\(F\)\nb-equivariant.  These bijections for all spaces~\(Y\) form an
isomorphism between the categories of \(\Gr[U]\)\nb-actions and
\(F\)\nb-actions.

\begin{proposition}
  \label{pro:universal_F-action}
  Let~\(\Gr[U]\) be a groupoid model for \(F\)\nb-actions.  The
  object space~\(\Gr[U]^0\) of\/~\(\Gr[U]\) carries a canonical
  \(F\)\nb-action, and this is a universal \(F\)\nb-action.
\end{proposition}

\begin{proof}
  Give~\(\Gr[U]^0\) its canonical \(\Gr[U]\)\nb-action.  We claim
  that~\(\Gr[U]^0\) is terminal in the category of
  \(\Gr[U]\)\nb-actions.  Let~\(Y\) be a \(\Gr[U]\)\nb-space.  The
  anchor map \(s \colon Y\to \Gr[U]^0\) is
  \(\Gr[U]\)\nb-equivariant.  It is the only
  \(\Gr[U]\)\nb-equivariant map \(Y\to \Gr[U]^0\) because any
  \(\U\)\nb-equivariant map \(f \colon Y \to \U^0\) must intertwine
  the anchor maps to~\(\Gr[U]^0\).  Thus~\(\Gr[U]^0\) is terminal in
  the category of \(\Gr[U]\)\nb-actions.  By the definition of a
  groupoid model, we may turn a \(\Gr[U]\)\nb-action on a
  space~\(Y\) into an \(F\)\nb-action on~\(Y\), such that this
  defines an isomorphism of categories.  Therefore, it preserves
  terminal objects.
\end{proof}

Our next goal is to prove that groupoid models are unique up to
isomorphism.  This uses the following lemma.

\begin{lemma}
  \label{lem:groupoid_model_invariant_map}
  Let~\(\Gr[U]\) be a groupoid model for a diagram~\(F\).  The
  natural bijection between actions of \(F\) and~\(\Gr[U]\) on~\(Y\)
  also preserves invariant maps, that is, a continuous map
  \(Y\to Z\) for another space~\(Z\) is \(\Gr[U]\)\nb-invariant if
  and only if it is \(F\)\nb-invariant.
\end{lemma}

\begin{proof}
  Let \(g\colon Y\to Z\) be a continuous map.  If~\(g\) is
  \(F\)\nb-invariant, then the \(F\)\nb-action restricts to an
  action on~\(g^{-1}(z)\) for each \(z\in Z\), and the inclusion map
  \(g^{-1}(z) \hookrightarrow Y\) is \(F\)\nb-equivariant.
  Conversely, if these inclusion maps for \(z\in Z\) are
  \(F\)\nb-equivariant for some \(F\)\nb-actions on \(g^{-1}(z)\),
  then~\(g\) is \(F\)\nb-invariant.  Since the bijection between
  actions of \(F\) and~\(\Gr[U]\) is natural, the action of~\(F\)
  on~\(g^{-1}(z)\) induces an action of~\(\Gr[U]\), and the
  inclusion map \(g^{-1}(z) \hookrightarrow Y\) is
  \(\Gr[U]\)\nb-equivariant.  Then its image is
  \(\Gr[U]\)\nb-invariant, and then~\(g\) is
  \(\Gr[U]\)\nb-invariant.  This argument may be reversed.
\end{proof}

\begin{proposition}
  \label{pro:groupoid_model}
  Let \(\Gr[U]\) and~\(\Gr[U]'\) be two groupoid models for
  \(F\)\nb-actions.  There is a unique groupoid isomorphism
  \(\Gr[U]\cong\Gr[U]'\) that is compatible with the equivalence
  between actions of \(\Gr[U]\), \(\Gr[U]'\) and~\(F\).
\end{proposition}

\begin{proof}
  This follows from
  \longref{Lemma}{lem:groupoid_model_invariant_map} and
  \cite{Meyer-Zhu:Groupoids}*{Propositions 4.23 and~4.24}.  For the
  convenience of the reader, we prove this directly, without
  mentioning the category of groupoid actors.  The canonical actions
  of \(\Gr[U]\) and~\(\Gr[U]'\) on their object spaces correspond to
  actions of~\(F\), which are both universal.  Since any two
  terminal objects are isomorphic with a unique isomorphism, there
  is a unique \(F\)\nb-equivariant homeomorphism
  \(\Gr[U]^0 \cong (\Gr[U]')^0\).  To simplify notation, we identify
  \(\Gr[U]^0\) and \((\Gr[U]')^0\) through this isomorphism.

  A \(\Gr[U]\)\nb-action on a space~\(Y\) contains an anchor map
  \(Y\to \Gr[U]^0\).  Forgetting the rest of the action gives a
  forgetful functor from the category of \(\Gr[U]\)\nb-actions to
  the category of spaces over~\(\Gr[U]^0\).  This functor has a left
  adjoint, namely, the map that sends a space~\(Z\) with a
  continuous map \(\varrho\colon Z\to\Gr[U]^0\) to
  \(\Gr[U]\times_{\s,\Gr[U]^0,\varrho} Z\) with the obvious left
  \(\Gr[U]\)\nb-action,
  \(\gamma_1\cdot (\gamma_2,z) \defeq (\gamma_1\cdot \gamma_2,z)\).
  Being left adjoint to the forgetful functor means that
  \(\Gr[U]\)\nb-equivariant continuous maps
  \(\psi\colon \Gr[U]\times_{\s,\Gr[U]^0,\varrho} Z \to Y\) for a
  \(\Gr[U]\)\nb-space~\(Y\) are in natural bijection with maps
  \(\varphi\colon Z \to Y\) that satisfy
  \(r\circ \varphi = \varrho\).  Under this natural bijection,
  \(\psi\) corresponds to the map \(\psi^\flat\colon Z\to Y\),
  \(z\mapsto \psi(\varrho(z),y)\), whereas \(\varphi\colon Z\to Y\)
  corresponds to the map
  \(\varphi^\#\colon \Gr[U]\times_{\s,\Gr[U]^0,\varrho} Z \to Y\)
  defined by \(\varphi^\#(\gamma,z) \defeq \gamma\cdot \varphi(z)\).
  We give~\(\Gr[U]\) the left multiplication action of~\(\Gr[U]\),
  whose anchor map is~\(\rg\).  The adjunct of the unit map
  \(u\rg\colon \Gr[U] \to \Gr[U]\) is a \(\Gr[U]\)\nb-equivariant
  isomorphism
  \(u^\#\colon \Gr[U]\times_{\s,\Gr[U]^0,\varrho} \Gr[U]^0 \congto
  \Gr[U]\).  The identity map on~\(\Gr[U]\) and the map
  \(\gamma\mapsto \rg(\gamma)\), are maps over~\(\Gr[U]^0\).  Their
  adjuncts are \(\id^\#(\gamma,\eta) = \gamma\cdot \eta\), the
  multiplication map, and
  \((u\rg)^\#(\gamma,\eta) = \gamma\cdot \rg(\eta) = \gamma\), the
  first coordinate projection.

  Since \(\Gr[U]\) and~\(\Gr[U]'\) are both groupoid models
  for~\(F\), the categories of \(\Gr[U]\)\nb-actions and
  \(\Gr[U]'\)\nb-actions are isomorphic to the category of
  \(F\)\nb-actions.  It follows that there is a unique isomorphism
  between the left adjoint functors
  \(\Gr[U]\times_{\s,\Gr[U]^0,\varrho} \blank\) and
  \(\Gr[U]'\times_{\s,\Gr[U]^0,\varrho} \blank\) that preserves the
  adjunction \(\varphi\mapsto \varphi^\#\).  Plugging
  in~\(\Gr[U]^0\), this isomorphism gives an \(F\)\nb-equivariant
  homeomorphism \(h\colon \Gr[U] \congto \Gr[U]'\).  Since~\(h\) is
  \(F\)\nb-equivariant, it intertwines the range maps on \(\Gr[U]\)
  and~\(\Gr[U]'\).  If \(x\in\Gr[U]^0\), then the inclusion map
  makes~\(\{x\}\) a space over~\(\Gr[U]^0\), and the inclusion
  \(\{x\} \hookrightarrow \Gr[U]^0\) is a map of spaces
  over~\(\Gr[U]^0\).  By naturality, the homeomorphism~\(h\)
  intertwines the inclusion of
  \(\Gr[U]\times_{\s,\Gr[U]^0,x} \{x\} =
  \setgiven{\gamma\in\Gr[U]}{\s(\gamma)=x}\) into~\(\Gr[U]\) and the
  corresponding map for~\(\Gr[U]'\).  That is, \(h\)
  intertwines the source maps to~\(\Gr[U]^0\) as well.

  Next, the natural isomorphism between
  \(\Gr[U]\times_{\s,\Gr[U]^0,\varrho} \blank\) and
  \(\Gr[U]'\times_{\s,\Gr[U]^0,\varrho} \blank\) specialises to an
  \(F\)\nb-equivariant homeomorphism
  \(h_2\colon \Gr[U]\times_{\s,\Gr[U]^0,\rg} \Gr[U] \cong
  \Gr[U]'\times_{\s,\Gr[U]^0,\rg} \Gr[U]'\).  The compatibility with
  the adjunction \(\varphi\mapsto \varphi^\#\) for the maps \(\id\)
  and~\(u\rg\) discussed above shows that~\(h_2\) intertwines the
  multiplication maps and the first coordinate projections.  Since
  the inclusion map \(\{\eta\} \hookrightarrow \Gr[U]\) is a map of
  spaces over~\(\Gr[U]^0\), naturality implies that
  \(h_2(\gamma,\eta)\) has the form \((\gamma',h(\eta))\).  Since
  the first coordinate projection is intertwined, we get
  \(h_2(\gamma,\eta)=(h(\gamma),h(\eta))\).  Thus~\(h\) is a
  groupoid isomorphism.
\end{proof}

\subsection{Disjoint unions and graded groupoids as groupoid models}

\begin{proposition}
  \label{pro:union_groupoid}
  Let~\(\Cat\)
  be a set viewed as a category with only identity arrows.  A morphism
  \(F\colon \Cat\to\Grcat\)
  is equivalent to a collection of groupoids \((\Gr_x)_{x\in\Cat^0}\).
  The disjoint union groupoid \(\bigsqcup_{x\in \Cat^0} \Gr_x\)
  is a groupoid model for \(F\)\nb-actions.
\end{proposition}

\begin{proof}
  An \(F\)\nb-action on a space~\(Y\) is a decomposition into clopen
  subsets, \(Y = \bigsqcup_{x\in\Cat^0} Y_x\), with a
  \(\Gr_x\)\nb-action on~\(Y_x\) for each \(x\in\Cat^0\).  This is
  exactly the same as an action of the disjoint union groupoid
  \(\bigsqcup_{x\in \Cat^0} \Gr_x\).  That is, \(\bigsqcup_{x\in \Cat^0}
  \Gr_x\) is a groupoid model for \(F\)\nb-actions.
\end{proof}

\begin{proposition}
  \label{pro:group_action_limit}
  Let~\(\Cat\) be a group.  Turn a diagram \(F\colon \Cat\to\Grcat\)
  into a \(\Cat\)\nb-graded groupoid~\(\Gr[L]\) as
  in~\longref{Section}{sec:group_action}.  Then~\(\Gr[L]\) is a
  groupoid model for \(F\)\nb-actions.
\end{proposition}

\begin{proof}
  An \(F\)\nb-action on a space~\(Y\) is equivalent to an
  \(\Gr[L]\)\nb-action by restricting the latter to the pieces of
  different degree.
\end{proof}

\subsection{The universal \texorpdfstring{\((m,n)\)}{(m,n)}-dynamical system}
\label{sec:universal_nm}

Let~\(\Cat\) be the path category of the directed graph
\(1 \rightrightarrows 2\) (compare
\longref{Section}{sec:path_categories}).  That is, \(\Cat\) has two
objects~\(1,2\), two parallel nonidentity arrows
\(h,v\colon 1\rightrightarrows 2\), and the identity arrows on \(1\)
and~\(2\).  Consider a diagram in~\(\Grcat\) where \(1\) and~\(2\)
both go to the one-point groupoid~\(\mathds{1}\).  The arrows
\(h,v\) go to correspondences \(\mathds{1}\leftarrow \mathds{1}\).
Such correspondences are discrete sets \(\Bisp_h,\Bisp_v\) because
the source map is a local homeomorphism to the one-point set.  These
discrete sets are finite if and only if the two correspondences are
proper.  As a result, for each \(m,n\in\N\) there is, up to
isomorphism, a unique diagram~\(E_{m,n}\) in~\(\Grcat_\prop\) where
\(\Gr_1 = \Gr_2 = \mathds{1}\) and \(\abs{\Bisp_h}=n\),
\(\abs{\Bisp_v} = m\).

By definition, an \(E_{m,n}\)\nb-action on a space~\(Y\) is a
partition \(Y = Y_1\sqcup Y_2\) into two clopen
subsets \(Y_1,Y_2\subseteq Y\) with homeomorphisms
\[
\alpha_h\colon \{1,\dotsc,n\} \times Y_1 \congto Y_2,
\qquad
\alpha_v\colon \{1,\dotsc,m\} \times Y_1 \congto Y_2;
\]
here we identify \(\Bisp_h \Grcomp_{\mathds{1}} Y_1\) with the
product \(\Bisp_h \times Y_1\).  The map~\(\alpha_h\) is equivalent
to homeomorphisms \(h_1,\dotsc,h_n\) from~\(Y_1\) onto clopen
subsets that partition~\(Y_2\), that is,
\(Y_2 = \bigsqcup_{i=1}^n h_i(Y_1)\).  Similarly, \(\alpha_v\) is
equivalent to homeomorphisms \(v_1,\dotsc,v_m\) from~\(Y_1\) onto
clopen subsets of~\(Y_2\) that partition~\(Y_2\), that is,
\(Y_2 = \bigsqcup_{i=1}^m v_i(Y_1)\).  If \(n=1\) or \(m=1\), we may
use one of the maps \(v\) or~\(h\) to identify \(Y_1\) and~\(Y_2\),
and then we are left with a homeomorphism \(A\times Y \congto Y\)
(as in \longref{Lemma}{lem:action_self-similar_group} in the case
where~\(\Gr\) is the trivial group).  This case is much simpler than
the case \(n,m\ge2\).

\begin{definition}
  Let~\(Y\) be a topological space.  A \emph{partial homeomorphism}
  of~\(Y\) is a homeomorphism between two open subsets of~\(Y\).
  These are composed by the obvious formula: if \(f,g\) are partial
  homeomorphisms of~\(Y\), then \(f g\) is the partial homeomorphism
  of~\(Y\) that is defined on \(y\in Y\) if and only if \(g(y)\) and
  \(f(g(y))\) are defined, and then \((f g)(y) \defeq f(g(y))\).
  If~\(f\) is a partial homeomorphism of~\(Y\), we let~\(f^*\) be
  its ``partial inverse'', defined on the image of~\(f\) by
  \(f^*(f(y)) = y\) for all \(y\) in the domain of~\(f\).
\end{definition}

The maps \(h_i\) and~\(v_i\) above are examples of partial
homeomorphisms of the space~\(Y\).  It will be crucial later to
describe actions of general diagrams through certain partial
homeomorphisms.

A compact Hausdorff space~\(Y\) with maps \(h_i,v_j\) as above is
called an \emph{\((m,n)\)-\alb{}dynamical system} by Ara, Exel and
Katsura~\cite{Ara-Exel-Katsura:Dynamical_systems}.  We allow~\(Y\)
to be any topological space.  It will be shown
in~\cite{Ko-Meyer:Groupoid_models} that the universal
\((m,n)\)-\alb{}dynamical system takes place on a compact Hausdorff
space.  Thus it makes no difference whether we restrict to actions
on compact Hausdorff spaces or allow actions on arbitrary
topological spaces.

We follow~\cite{Ara-Exel-Katsura:Dynamical_systems}
in studying \(E_{m,n}\)\nb-actions through partial group actions,
which are a special class of actions of inverse semigroups by
partial homeomorphisms.  Let~\(Y\) be an \(E_{m,n}\)\nb-action.
There is a unique partial action~\(\vartheta\) of the free
group~\(\mathbb{F}_{n+m}\) on~\(Y\) where the generators act by
\(h_1,\dotsc,h_n,v_1,\dotsc,v_m\).  To construct it, write an
element \(g\in\mathbb{F}_{n+m}\) as a reduced word in the generators
and their inverses.  Let the generators and their inverses act by
the partial homeomorphisms
\(h_1^{\pm1},\dotsc,h_n^{\pm1}, v_1^{\pm1},\dotsc,v_m^{\pm1}\), and
compose the partial homeomorphisms for the letters of the reduced
word expressing~\(g\) to get~\(\vartheta_g\).  For the empty word,
\(\vartheta_\emptyset\) is the identity on~\(Y\).  The partial
homeomorphism \(\vartheta_g \vartheta_{g^{-1}}\) is contained in the
identity for each \(g\in\mathbb{F}_{n+m}\), and
\(\vartheta_g\vartheta_h = \vartheta_{g h}\) if \(g\), \(h\)
and~\(g h\) are reduced words.  This implies that~\(\vartheta\) is a
partial action of~\(\mathbb{F}_{n+m}\) on~\(Y\).  The construction
shows that the domain of~\(\vartheta_g\) is clopen for each
\(g\in\mathbb{F}_{n+m}\).

By construction, a map between \(E_{m,n}\)\nb-actions is
equivariant if and only if it is equivariant for the partial
homeomorphisms \(h_i,v_j\) on~\(Y\), if and only if it is
equivariant for the partial actions of~\(\mathbb{F}_{n+m}\).
There are,
however, more partial actions of~\(\mathbb{F}_{n+m}\)
than \(E_{m,n}\)\nb-actions.
A partial action~\(\varrho\)
of~\(\mathbb{F}_{n+m}\)
that comes from an \(E_{m,n}\)\nb-action
has the following extra properties:
\begin{enumerate}
\item \label{en:partial_action_from_nm_1}%
  \(\varrho_g\varrho_h = \varrho_{g h}\) if \(g\), \(h\)
  and~\(g h\) are reduced words;
\item \label{en:partial_action_from_nm_2}%
  \(\varrho_{a_1}^{-1} \varrho_{a_1} = \varrho_{a_i}^{-1}
  \varrho_{a_i}\) for \(i=2,\dotsc,n+m\), that is, the partial
  homeomorphisms~\(\varrho_{a_i}\) for the generators
  \(a_1,\dotsc,a_{n+m}\) all have the same domain, which we
  call~\(Y_1\);
\item \label{en:partial_action_from_nm_3}%
  \(\varrho_{a_i}\varrho_{a_j}=\emptyset\) for all
  \(1\le i,j\le n+m\), that is, \(Y_1\) is disjoint from the
  codomain of~\(\varrho_{a_j}\);
\item \label{en:partial_action_from_nm_4}%
  \(\varrho_{a_i}^{-1}\varrho_{a_j} = \emptyset\) if \(i\neq j\) and
  either \(1\le i,j\le n\) or \(n+1\le i,j\le n+m\), that is, the
  codomains of~\(\varrho_{a_i}\) for \(i=1,\dotsc,n\) are disjoint
  and the codomains of~\(\varrho_{a_i}\) for \(i=n+1,\dotsc,n+m\)
  are disjoint;
\item \label{en:partial_action_from_nm_5}%
  the (disjoint) union of~\(Y_1\) and the codomains of
  \(\varrho_{a_i}\) for \(i=1,\dotsc,n\) is~\(Y\), and so is the
  (disjoint) union of~\(Y_1\) and the codomains of \(\varrho_{a_i}\)
  for \(i=n+1,\dotsc,n+m\).
\end{enumerate}
Conversely, any partial action of~\(\mathbb{F}_{n+m}\) with these
extra properties comes from an \(E_{m,n}\)\nb-action.  The first
property says that the action is determined by what it does on the
generators.  The remaining conditions say that these generators act
by homeomorphisms from~\(Y_1\) to clopen subsets of \(Y_2
\defeq Y\setminus Y_1\) that partition~\(Y_2\) in two
different ways, exactly as for an \(E_{m,n}\)\nb-action.

Partial actions of a group~\(\Gamma\) are equivalent to actions
preserving \(0\) and~\(1\) of a certain inverse
semigroup~\(\IS(\Gamma)\) with \(0\) and~\(1\)
(see~\cite{Exel:Partial_actions}).  The conditions
\ref{en:partial_action_from_nm_1}--\ref{en:partial_action_from_nm_4}
above can be encoded through a congruence relation in the inverse
semigroup~\(\IS(\mathbb{F}_{n+m})\).  So partial actions
of~\(\mathbb{F}_{n+m}\) satisfying
\ref{en:partial_action_from_nm_1}--\ref{en:partial_action_from_nm_4}
are equivalent to actions of a certain quotient inverse
semigroup~\(\IS(F_{n+m})\) of~\(\IS(\mathbb{F}_{n+m})\).  The last
condition~\ref{en:partial_action_from_nm_5}, however, is not of this
form.  Thus we cannot model \(E_{m,n}\)\nb-actions purely as actions
of some inverse semigroup.

A ``standard \((m,n)\)\nb-dynamical
system''~\(\Omega\)
is constructed in~\cite{Ara-Exel-Katsura:Dynamical_systems} by
analysing a certain quotient of a separated graph \(\Cst\)\nb-algebra
(see \cite{Ara-Exel-Katsura:Dynamical_systems}*{Definition~3.2}).  This
is a universal \((m,n)\)\nb-dynamical
system by~\cite{Ara-Exel-Katsura:Dynamical_systems}*{Theorem~3.8}.
The space~\(\Omega\)
is a certain closed subset of~\(\{0,1\}^{\mathbb{F}_{n+m}}\),
equipped with the subspace topology from the product topology.  That
is, \(\Omega\)
is a set of subsets of~\(\mathbb{F}_{n+m}\).
The unique \(E_{m,n}\)\nb-equivariant
map \(Y\to \Omega\)
sends a point~\(x\)
in~\(Y\)
to the subset of all \(g\in \mathbb{F}_{n+m}\)
for which~\(\vartheta_g\)
is defined at~\(x\).
This defines a continuous map \(Y\to \Omega\)
because the domains of the~\(\vartheta_g\) are clopen.

Another way to describe~\(\Omega\)
uses the Paterson groupoid of the quotient inverse semigroup
\(S\defeq \IS(\mathbb{F}_{n+m})/{\sim}\)
mentioned above (see~\cite{Paterson:Groupoids}).  The object space of
this groupoid is the space~\(\hat{E}\)
of all characters \(E\to \{0,1\}\),
where~\(E\)
is the idempotent semilattice of~\(S\);
the topology is the product topology from~\(\{0,1\}^E\).
The Paterson groupoid has the universal property that actions of~\(S\)
on topological spaces with clopen domains are equivalent to actions of
the transformation groupoid \(\hat{E}\rtimes S\).
Here the anchor map sends \(\omega\in Y\)
to the character that maps \(e\in E\)
to~\(1\)
if and only if \(\omega\)
belongs to the domain of~\(e\)
in the partial action.  The last condition~(5) singles out a closed
subspace~\(\Omega\)
of~\(\hat{E}\),
so that an action of~\(S\)
satisfies~(5) if and only if the corresponding map
\(Y\to \hat{E}\) takes values in~\(\Omega\).

Summing up, an \(E_{m,n}\)\nb-action on~\(Y\) is equivalent to a
partial action of~\(\mathbb{F}_{n+m}\) satisfying the conditions
\ref{en:partial_action_from_nm_1}--\ref{en:partial_action_from_nm_5}
above.  This, in turn, is equivalent to a partial action of the
inverse semigroup~\(\IS(F_{n+m})\) together with an equivariant
continuous map \(Y\to \Omega\).  The latter is equivalent to an
action of the transformation groupoid
\(\Omega \rtimes \IS(F_{n+m}) \cong \Omega\rtimes\mathbb{F}_{n+m}\).
As a result, \(\Omega\rtimes\mathbb{F}_{n+m}\) is a groupoid model
for \(E_{m,n}\)\nb-actions.

\section{Encoding a diagram action through an inverse semigroup
  action}
\label{sec:encoding}

Eventually, we will describe the groupoid model of a diagram~\(F\)
as a transformation groupoid for an action of a certain inverse
semigroup \(\IS(F)\) on the space that underlies the universal
action of the diagram.  To prepare for this, we are going to encode
an action of a diagram on a space~\(Y\) through partial
homeomorphisms of~\(Y\).  We find the relations among these partial
homeomorphisms, and these relations lead to the definition of the
inverse semigroup~\(\IS(F)\).

\subsection{Slices of correspondences}
\label{sec:slices}

We introduced slices of groupoids and groupoid correspondences in
\longref{Definition}{def:correspondence_slices} and
\longref{Remark}{rem:slice_tight}.  Now we recall some results about
them from~\cite{Antunes-Ko-Meyer:Groupoid_correspondences}.

Let \(\Bisp\colon \Gr[H]\leftarrow \Gr\) be a groupoid
correspondence and let \(\U,\V\subseteq \Bisp\) be slices.  Recall
that \(\braket{x}{y}\) for \(x,y\in\Bisp\) with \(\Qu(x) = \Qu(y)\)
is the unique arrow in~\(\Gr\) with \(x\cdot \braket{x}{y} = y\).
The subset
\[
  \braket{\U}{\V} \defeq
  \setgiven{\braket{x}{y}\in\Gr}{x\in\U,\ y\in\V,\ \Qu(x)=\Qu(y)}
\]
is a slice in the groupoid~\(\Gr\) by
\cite{Antunes-Ko-Meyer:Groupoid_correspondences}*{Lemma~7.7}.  In
addition, if \(z\in \braket{\U}{\V}\), then the elements \(x\in\U\),
\(y\in\V\) with \(\Qu(x) = \Qu(y)\) and \(z=\braket{x}{y}\) are
unique.

Next, let \(\Bisp\colon \Gr[H]\leftarrow \Gr\) and
\(\Bisp[Y]\colon \Gr\leftarrow \Gr[K]\) be groupoid correspondences
and let \(\U\subseteq \Bisp\) and \(\V\subseteq \Bisp[Y]\) be
slices.  Then
\[
  \U \cdot \V\defeq
  \setgiven{[x,y] \in \Bisp\Grcomp_{\Gr} \Bisp[Y]}{x\in\U,\
    y\in\V,\ \s(x) = \rg(y)}
\]
is a slice in the composite groupoid correspondence
\(\Bisp\Grcomp_{\Gr} \Bisp[Y]\) by
\cite{Antunes-Ko-Meyer:Groupoid_correspondences}*{Lemma~7.14}.  In
addition, if \(z\in \U\cdot \V\), then the elements
\(x\in\U\), \(y\in\V\) with \(\s(x) = \rg(y)\) and
\(z=[x,y]\) are unique.  In particular, we may multiply a slice
in~\(\Bisp\) by slices in~\(\Gr[H]\) on the left and~\(\Gr\) on the
right, and this produces slices in~\(\Bisp\) again.

Let~\(F\) be a diagram of groupoid correspondences.  Let \(\Bis(F)\)
be the set of all slices of the correspondences~\(\Bisp_g\) for all
arrows \(g\in\Cat\), modulo the relation that we identify the empty
slices of \(\Bis(\Bisp_g)\) for all \(g\in \Cat\).  Given composable
arrows \(g,h\in\Cat\) and slices \(\U\subseteq \Bisp_g\),
\(\V\subseteq \Bisp_h\), then
\(\U\V \defeq \mu_{g,h}(\U\cdot \V)\) is a slice
in~\(\Bisp_{g h}\) by
\cite{Antunes-Ko-Meyer:Groupoid_correspondences}*{Lemma~7.14}.  When
we write~\(\mu_{g,h}\) multiplicatively, \(\U\V\) is the
subset of all products \(\gamma\cdot\eta\) for \(\gamma\in\U\),
\(\eta\in\V\) with \(\s(\gamma)=\rg(\eta)\).  If \(g,h\) are not
composable, then we let \(\U\V\) be the empty
slice~\(\emptyset\).  This turns~\(\Bis(F)\) into a semigroup with
zero element~\(\emptyset\).

\subsection{Encoding actions of diagrams through partial
  homeomorphisms}
\label{encode}

We are going to describe an action of a single groupoid
correspondence or of a diagram of groupoid correspondences on a
topological space~\(Y\) through certain partial homeomorphisms.

\begin{lemma}
  \label{lem:slice_acts}
  Let \(\Bisp\colon \Gr[H]\leftarrow \Gr\) be a groupoid
  correspondence and let~\(Y\) be a \(\Gr\)\nb-space.  Let
  \(\U\subseteq\Bisp\) be a slice.  Write \([\gamma,y]\) instead
  of \(\Qu(\gamma,y)\) for the image of
  \((\gamma,y) \in \Bisp\times_{\s,\Gr^0,\rg} Y\)
  in~\(\Bisp\Grcomp_{\Gr} Y\).  There is a
  homeomorphism~\(\U_*\) between the open subsets
  \(\rg^{-1}(\s(\U))\subseteq Y\) and
  \(\Qu(\U\times_{\s,\Gr^0,\rg} Y)\subseteq \Bisp\Grcomp_{\Gr}
  Y\), which maps \(y\in Y\) to \([\gamma,y]\) for the unique
  \(\gamma\in\U\) with \(\s(\gamma) = \rg(y)\).
\end{lemma}

\begin{proof}
  Since~\(\s|_\U\) is a homeomorphism
  \(\U \to \s(\U)\subseteq \Gr^0\), the map~\(\U_*\) is
  well defined and continuous.  Its image is
  \(\Qu(\U\times_{\s,\Gr^0,\rg} Y) \subseteq
  \Bisp\Grcomp_{\Gr}Y\).  Suppose \(\U_*(y_1)=\U_*(y_2)\)
  for \(y_1,y_2\in\rg^{-1}(\s(\U))\).  That is, there are
  \(\gamma_i\in\U\) with \(\s(\gamma_i)=\rg(y_i)\) for \(i=1,2\)
  and \([\gamma_1,y_1]=[\gamma_2,y_2]\) in \(\Bisp\Grcomp_{\Gr} Y\).
  Then there is \(\eta\in\Gr\) with
  \(\gamma_2 = \gamma_1 \eta\) and \(y_1 = \eta y_2\).  Then
  \(\Qu(\gamma_1)=\Qu(\gamma_2)\).  This implies
  \(\gamma_1 = \gamma_2\) because~\(\Qu\) is injective
  on~\(\U\).  Since the right \(\Gr\)\nb-action on~\(\Bisp\) is
  free, \(\eta\) is a unit and hence \(y_1=y_2\).  Thus~\(\U_*\)
  is injective.

  The quotient map
  \(\Qu\colon \Bisp\times_{\s,\Gr^0,\rg} Y \prto \Bisp\Grcomp_{\Gr}
  Y\) is a local homeomorphism by
  \longref{Lemma}{lem:basic_orbit_lh}.  So is the coordinate
  projection \(\pr_2\colon \Bisp\times_{\s,\Gr^0,\rg} Y \prto Y\) by
  \cite{Antunes-Ko-Meyer:Groupoid_correspondences}*{Lemma~2.9}
  because \(\s\colon \Bisp\to\Gr^0\) is a local homeomorphism.  The
  map~\(\U_*\) is the composite of the local section
  \(y\mapsto (\gamma,y)\) of \(\pr_2\) and the map~\(\Qu\).  Hence
  it is a local homeomorphism.  Being injective, it is a
  homeomorphism onto an open subset of \(\Bisp\Grcomp_{\Gr} Y\).
\end{proof}

An action of an étale groupoid induces an action of its slices by
partial homeomorphisms (see~\cite{Exel:Inverse_combinatorial}).  The
following construction generalises this to an action of a diagram of
groupoid correspondences~\(F\).  We describe~\(F\) as in
\longref{Proposition}{pro:diagrams_in_Grcat} by
\((\Gr_x,\Bisp_g,\mu_{g,h})\).  Let~\(Y\) with the partition
\(Y = \bigsqcup_{x\in\Cat^0} Y_x\) be an \(F\)\nb-action.  To
simplify notation, we write~\(\circ\) instead of~\(\circ_{\Gr_x}\).  The
multiplication maps of the \(F\)\nb-action induce
\(\Gr_x\)\nb-equivariant homeomorphisms
\(\dot\alpha_g\colon \Bisp_g \Grcomp Y_{x'} \xrightarrow {\cong}
Y_x\) for all arrows \(g\colon x\leftarrow x'\) in~\(\Cat\) by
\longref{Lemma}{homeo}.  A slice \(\U\subseteq\Bisp_g\)
induces a partial homeomorphism~\(\U_*\) from~\(Y_{x'}\) to
\(\Bisp_g \Grcomp Y_{x'}\) by \longref{Lemma}{lem:slice_acts}.
Combining these maps gives a partial homeomorphism
\[
\vartheta(\U)\defeq \dot\alpha_g\circ\U_* \colon
Y_{x'} \supseteq \rg^{-1}(\s(\U)) \to Y_x;
\]
it maps~\(y\) to \(\gamma\cdot y\) for the unique \(\gamma\in \U\)
with \(\s(\gamma)=\rg(y)\).  We also view~\(\vartheta(\U)\) as a
partial homeomorphism of~\(Y\).

\begin{lemma}
  \label{lem:theta_multiplicative}
  Let~\(Y\) be a topological space with an action of a diagram of
  groupoid correspondences~\(F\).  Define \(\vartheta(\U)\) for
  a slice \(\U\in \Bis(F)\) as above, and
  let~\(\vartheta(\U)^*\) be its partial inverse.  Then
  \(\vartheta(\U \V) = \vartheta(\U)\vartheta(\V)\)
  for all \(\U,\V\in \Bis(F)\), and
  \(\vartheta(\U_1)^*\vartheta(\U_2) =
  \vartheta(\braket{\U_1}{\U_2})\) if \(\U_1,\U_2\)
  are slices in~\(\Bisp_g\) for the same arrow \(g\in\Cat\).
\end{lemma}

\begin{proof}
  Let \(\U\subseteq\Bis(\Bisp_g)\), \(\V\subseteq\Bis(\Bisp_h)\) for
  arrows \(g\colon y\leftarrow x\) and \(h\colon x'\leftarrow z\)
  in~\(\Cat\).  If \(x\neq x'\), then \(Y_x\cap Y_{x'}=\emptyset\).
  Then \(\vartheta(\U)\vartheta(\V) = \emptyset\) because the
  codomain of \(\vartheta(\V)\) is contained in~\(Y_{x'}\) whereas the
  domain of \(\vartheta(\U)\) is contained in~\(Y_x\).  If \(x=x'\),
  then \(\vartheta(\U \V) = \vartheta(\U)\vartheta(\V)\) follows
  immediately from the
  \longref{condition}{en:diagram_dynamical_system1} for an
  \(F\)\nb-action.  Recall that \(\vartheta(\U_i)\) for \(i=1,2\) is
  the partial homeomorphism that maps \(y_i\in Y_x\) with
  \(\rg(y_i)\in\s(\U_i)\) to~\(\gamma_i\cdot y_i\) for the unique
  \(\gamma_i\in\U_i\) with \(\s(\gamma_i) = \rg(y_i)\).  Thus the
  composite partial homeomorphism
  \(\vartheta(\U_1)^*\vartheta(\U_2)\) maps~\(y_2\) to~\(y_1\) if
  and only if \(\gamma_1 \cdot y_1 = \gamma_2 \cdot y_2\).  By the
  \longref{condition}{en:diagram_dynamical_system2} for an
  \(F\)\nb-action, this happens if and only if there is
  \(\eta\in\Gr_x\) with \(\gamma_1 \eta = \gamma_2\) and
  \(\eta y_2 = y_1\).  If such an~\(\eta\) exists, then
  \(\eta=\braket{\gamma_1}{\gamma_2}\) and
  \(y_1 = \braket{\gamma_1}{\gamma_2}\cdot y_2 =
  \vartheta(\braket{\U_1}{\U_2}) y_2\).  Thus
  \(\vartheta(\U_1)^*\vartheta(\U_2) =
  \vartheta(\braket{\U_1}{\U_2})\).
\end{proof}

The following lemma generalises the description of \((m,n)\)-dynamical
systems in \longref{Section}{sec:universal_nm} to actions of
arbitrary diagrams of groupoid correspondences.  The point is to
recover an action of the diagram~\(F\) from the partial
homeomorphisms~\(\vartheta(\U)\) for \(\U\in \Bis(F)\).  As it turns
out, this is possible if the objects spaces~\(\Gr_x^0\) for
\(x\in\Cat^0\) are locally compact or, more generally, sober --~but
not always.  In general, we also need the map
\[
\rg\colon Y = \bigsqcup_{x\in\Cat^0} Y_x
\to \bigsqcup_{x\in\Cat^0} \Gr_x^0
\]
that we get by combining the maps \(\rg\colon Y_x \to \Gr_x^0\) for
\(x\in\Cat^0\).  This map must be ``compatible'' with the partial
homeomorphisms~\(\vartheta(\U)\) for \(\U\in\Bis(F)\).  To express
this, we need another construction.  For a groupoid correspondence
\(\Bisp \colon \Gr[H] \leftarrow \Gr\) and a slice \(\U\subseteq\Bisp\),
we define a partial map~\(\U_\dagger\) on~\(\Gr^0\) by mapping
\(x\in \s(\U) \subseteq\Gr^0\) to \(\rg(\gamma)\) for the unique
\(\gamma\in\U\) with \(\s(\U)=x\); here we use that
\(s|_\U \colon \U \to \Gr^0\) is injective.  The partial
map~\(\U_\dagger\) need not be a partial homeomorphism,
unless~\(\Bisp\) is tight.

\begin{lemma}
  \label{lem:F-action_from_theta}
  Let~\(Y\) be a space and let
  \(\rg\colon Y\to \bigsqcup_{x\in\Cat^0} \Gr_x^0\) and
  \(\vartheta\colon \Bis(F)\to I(Y)\) be maps.  These come from an
  \(F\)\nb-action on~\(Y\) if and only if
  \begin{enumerate}[label=\textup{(\ref*{lem:F-action_from_theta}.\arabic*)},
    leftmargin=*,labelindent=0em]
  \item \label{en:F-action_from_theta1}%
    \(\vartheta(\U \V) = \vartheta(\U)\vartheta(\V)\)
    for all \(\U,\V\in \Bis(F)\);
  \item \label{en:F-action_from_theta2}%
    \(\vartheta(\U_1)^*\vartheta(\U_2) =
    \vartheta(\braket{\U_1}{\U_2})\)
    for all \(g\in\Cat\), \(\U_1,\U_2\in\Bis(\Bisp_g)\);
  \item \label{en:F-action_from_theta3}%
    the images of~\(\vartheta(\U)\) for \(\U\in\Bisp_g\) cover
    \(Y_{\rg(g)} \defeq \rg^{-1}(\Gr_{\rg(g)}^0)\) for each
    \(g\in\Cat\);
  \item \label{en:F-action_from_theta5}%
    \(\rg\circ \vartheta(\U) = \U_\dagger\circ\rg\) as partial maps
    \(Y \to \Gr^0\) for any \(\U\in\Bis(F)\).
  \end{enumerate}
  The corresponding \(F\)\nb-action on~\(Y\) is unique if it exists,
  and it satisfies
  \begin{enumerate}[label=\textup{(\ref*{lem:F-action_from_theta}.\arabic*)},
    leftmargin=*,labelindent=0em,resume]
  \item \label{en:F-action_from_theta4}%
    for \(U\subseteq \Gr_x^0\) open, \(\vartheta(U)\) is the
    identity map on~\(\rg^{-1}(U)\);
  \item \label{en:F-action_from_theta6}%
    for any \(\U\in\Bis(F)\), the domain of \(\vartheta(\U)\) is
    \(\rg^{-1}(\s(\U))\).
  \end{enumerate}
\end{lemma}

\begin{proof}
  Conditions \ref{en:F-action_from_theta1}
  and~\ref{en:F-action_from_theta2} are necessary by
  \longref{Lemma}{lem:theta_multiplicative}.
  \longref{Condition}{en:F-action_from_theta5} is necessary by the
  construction of~\(\vartheta\).
  \longref{Condition}{en:F-action_from_theta3} means that any
  \(y\in Y_x\) is of the form \(\gamma\cdot y'\) for some
  \(\gamma\in \Bisp_g\), \(y'\in Y_{x'}\) with
  \(\s(\gamma)=\rg(y')\).  Thus it expresses the surjectivity of the
  multiplication map \(\Bisp_g\times_{\s,\rg} Y_{x'} \to Y_x\).  So
  these conditions are necessary for~\(\vartheta\) to come from an
  \(F\)\nb-action.

  Conversely, assume the conditions
  \ref{en:F-action_from_theta1}--\ref{en:F-action_from_theta5}.  Let
  \(U\subseteq \Gr_x^0\) be an open subset and view it as an element
  of~\(\Bis(F)\).  The domains of the partial maps
  \(\rg\circ \vartheta(U)\) and \(U_\dagger\circ\rg\) are equal to
  the domain of \(\vartheta(\U)\) and \(\rg^{-1}(U)\), respectively.
  \longref{Condition}{en:F-action_from_theta5} implies that these
  domains are equal.  \longref{Condition}{en:F-action_from_theta1}
  implies that \(\vartheta(U)\) is idempotent, so that it is the
  identity map on its domain.  This gives
  \longref{condition}{en:F-action_from_theta4}.  Let
  \(\U\in\Bis(\Bisp_g)\).  The domain of~\(\vartheta(\U)\) is the
  same as the domain of \(\vartheta(\U)^* \vartheta(\U)\), which is
  equal to \(\vartheta(\braket{\U}{\U})\) by
  \ref{en:F-action_from_theta2}.  Since
  \(\braket{\U}{\U} \subseteq \Gr_{\s(g)}\) is the idempotent slice
  \(\s(\U)\subseteq \Gr_{\s(g)}^0\), \ref{en:F-action_from_theta5}
  implies that \(\vartheta(\braket{\U}{\U})\) is the identity on
  \(\rg^{-1}(\s(\U))\).  This finishes the proof of
  \ref{en:F-action_from_theta6}.

  Since the domains of partial homeomorphisms are open,
  \ref{en:F-action_from_theta4} implies that the subsets
  \(Y_x \defeq \rg^{-1}(\Gr_x^0)\) are open and that the maps
  \(\rg|_{Y_x}\colon Y_x\to \Gr_x^0\) are continuous.  Since
  \(Y= \bigsqcup_{x\in \Cat^0} Y_x\) by construction, the
  subsets~\(Y_x\) are also closed.

  Let \(g\colon x\leftarrow x'\) be an arrow in~\(\Cat\) and let
  \(\gamma\in\Bisp_g\), \(y\in Y_{x'}\) satisfy
  \(\s(\gamma) = \rg(y)\).  There is a slice \(\U\in\Bis(F)\) with
  \(\gamma\in \U\).  \longref{Condition}{en:F-action_from_theta6}
  implies that \(\vartheta(\U)\) is defined at~\(y\).  We want to
  put \(\gamma\cdot y \defeq \vartheta(\U)(y)\).  This is the only
  possibility for an \(F\)\nb-action that induces the given maps
  \(\vartheta(\U)\) and~\(\rg\).  We must show that this formula
  gives a well defined \(F\)\nb-action.  If \(\U'\in\Bis(\Bisp_g)\)
  is another slice with \(\gamma\in\U'\), then
  \(\gamma\in \U\cap \U'\).  Let
  \(V = \s(\U\cap \U')\subseteq \Gr_x^0\), viewed as a slice of
  \(\Gr_x = \Bisp_x\).  Then
  \(\U \cdot V = \U\cap \U' = \U' \cdot V\) and
  \[
  \vartheta(\U)y
  = \vartheta(\U)\vartheta(V) y
  = \vartheta(\U\cap \U')y
  = \vartheta(\U')\vartheta(V) y
  = \vartheta(\U')y.
  \]
  Thus \(\gamma\cdot y\)
  is well defined.  The resulting map
  \(\Bisp_g \times_{\s,\rg} Y_x \to Y_{x'}\)
  is continuous and open because the maps~\(\vartheta(\U)\)
  are partial homeomorphisms, and it is surjective
  by~\ref{en:F-action_from_theta3}.

  The \longref{condition}{en:F-action_from_theta5} says that
  \(\rg(\gamma\cdot y) = \rg(\gamma)\) for all \(\gamma\in\U\),
  \(y\in Y\) with \(\s(\gamma) = \rg(y)\).
  Then~\ref{en:F-action_from_theta1} implies
  \(\gamma_1\cdot (\gamma_2\cdot y) = (\gamma_1\cdot \gamma_2)\cdot
  y\) for all composable \(\gamma_1,\gamma_2,y\).  The
  condition~\ref{en:F-action_from_theta2} implies that
  \(\gamma_1 \cdot y_1 = \gamma_2 \cdot y_2\) for
  \(\gamma_1,\gamma_2\in\Bisp_g\), \(y_1,y_2\in Y_x\) holds only if
  \(\Qu(\gamma_1) = \Qu(\gamma_2)\) and
  \(y_1 = \braket{\gamma_1}{\gamma_2} y_2\).  It follows that all
  conditions for an \(F\)\nb-action in
  \longref{Definition}{def:diagram_dynamical_system} are satisfied.
\end{proof}

The following lemma describes \(F\)\nb-equivariant maps between two
\(F\)\nb-actions through the partial homeomorphisms \(\vartheta(\U)\)
for \(\U\in\Bis(F)\) and \(r\colon Y\to \Gr^0\).  This is important to
characterise the groupoid model.

\begin{lemma}
  \label{lem:theta_gives_equivariant_invariant}
  Let \(Y\) and~\(Y'\) be \(F\)\nb-actions and let~\(Z\) be any
  space.  A continuous map \(\varphi\colon Y\to Y'\) is
  \(F\)\nb-equivariant if and only if \(\rg'\circ \varphi = \rg\)
  and \(\vartheta'(\U)\circ\varphi = \varphi\circ\vartheta(\U)\) for
  all \(\U\in \Bis(F)\).  A continuous map \(\psi\colon Y\to Z\) is
  \(F\)\nb-invariant if and only if
  \(\psi(y) = \psi\bigl(\vartheta(\U)(y)\bigr)\) for all
  \(\U\in\Bis(F)\) and~\(y\) in the domain of \(\vartheta(\U)\).
\end{lemma}

\begin{proof}
  This follows easily from the explicit formulas
  for~\(\vartheta(\U)\) given an action and for the maps
  \(\Bisp_g \times_{\s,\rg} Y_{\s(g)} \to Y_{\rg(g)}\) given the
  maps \(\vartheta(\U)\) and \(\rg\colon Y\to \Gr^0\).
\end{proof}

\begin{remark}
  In the situation of
  \longref{Lemma}{lem:theta_gives_equivariant_invariant}, the
  map~\(\vartheta\) determines \(\rg^{-1}(U)\) for all open subsets
  \(U\subseteq \Gr^0\).  If the space~\(\Gr^0\) is ``sober'', then
  this determines the map \(\rg\colon Y\to \Gr^0\) by Stone duality.
  In particular, this works if~\(\Gr^0\) is locally compact.  Even
  more, Stone duality says that the map~\(\rg\) required in
  \longref{Lemma}{lem:theta_gives_equivariant_invariant} exists
  provided the map \(U\mapsto \operatorname{dom} \vartheta(U)\)
  commutes with arbitrary unions of open subsets and finite
  intersections.  Actually, it always commutes with intersections
  by~\ref{en:F-action_from_theta1}.  We do not give more details
  because we will not use this simplification of
  \longref{Lemma}{lem:theta_gives_equivariant_invariant}.
\end{remark}

\subsection{Induced inverse semigroup action}

We are going to turn a collection of partial homeomorphisms as in
\longref{Lemma}{lem:F-action_from_theta} into an action of a
suitable inverse semigroup.  There is an established way to turn an
inverse semigroup action into an étale groupoid, and when we apply
all this to a universal action of a diagram~\(F\), we get a groupoid
model of~\(F\).  To begin with, we recall some basics about inverse
semigroups (see~\cite{Higgins:Techniques_semigroup_theory}).

An \emph{inverse semigroup} is a semigroup~\(S\) with the extra
property that for each \(a\in S\), there is a unique element~\(a^*\)
with \(a a^* a = a\) and \(a^* a a^* = a^*\).

An element \(e \in S\) is \emph{idempotent} if \(e = e^2\).  The set of
idempotents of~\(S\) is denoted by \(E(S)\).  The idempotent elements
form a commutative subsemigroup, and \(s e s^* \in E(S)\) for all
\(s\in S\), \(e\in E(S)\).

\begin{example}
  The semigroup of partial homeomorphisms of a topological
  space~\(Y\) is an inverse semigroup \(I(Y)\).
  If \(f\in I(Y)\), then~\(f^*\) is its partial inverse.  The
  idempotent elements are the identity maps of open subsets
  of~\(Y\).  The product on idempotent elements corresponds to the
  intersection of open subsets.  It is well known that any inverse
  semigroup embeds into \(I(Y)\) for some topological space~\(Y\).
\end{example}

The most important inverse semigroup for us is the following:

\begin{definition}
  Let~\(F\) be a diagram of groupoid correspondences.
  Let~\(\IS(F)\) be the inverse semigroup that is generated
  by the set~\(\Bis(F)\) and the relations in the first two
  conditions in \longref{Lemma}{lem:F-action_from_theta}, that is,
  \(\Theta(\U \V) = \Theta(\U)\Theta(\V)\) for all
  \(\U,\V\in \Bis(F)\) and
  \(\Theta(\U_1)^*\Theta(\U_2) = \Theta(\braket{\U_1}{\U_2})\) for
  all \(g\in\Cat\), \(\U_1,\U_2\in\Bis(\Bisp_g)\); here we
  write~\(\Theta(\U)\) for~\(\U\) viewed as a generator
  of~\(\IS(F)\).
\end{definition}

To construct \(\IS(F)\), we first let~\(\mathrm{Free}(\Bis(F))\) be
the free inverse semigroup on the set \(\Bis(F)\) \textup{(}see
\textup{\cite{Higgins:Techniques_semigroup_theory}*{Theorem~1.1.10})}.
Then we let~\(\IS(F)\) be the quotient of~\(\mathrm{Free}(\Bis(F))\)
by the smallest congruence relation that contains the relations in
\ref{en:F-action_from_theta1} and \ref{en:F-action_from_theta2}.
This is an inverse semigroup by
\cite{Higgins:Techniques_semigroup_theory}*{Corollary~1.1.8}.  By
construction, there is a natural bijection between inverse semigroup
homomorphisms \(\IS(F)\to \Gamma\) for another inverse
semigroup~\(\Gamma\) and maps \(\vartheta\colon \Bis(F)\to \Gamma\)
that satisfy \ref{en:F-action_from_theta1}
and~\ref{en:F-action_from_theta2}.

\begin{definition}
  \label{invsemigpact}
  Let~\(S\) be an inverse semigroup and let~\(Y\) be a topological
  space.  An \emph{action of~\(S\) on~\(Y\)} is a semigroup homomorphism
  \(\vartheta \colon S \to I(Y)\).

  Let \(Y\) and~\(Y'\) be two spaces with \(S\)\nb-actions and
  let~\(Z\) be a space.  A map \(\varphi\colon Y\to Y'\) is
  \emph{\(S\)\nb-equivariant} if, for all \(t\in S\), \(y\in Y\),
  \(t\cdot y\) is defined if and only if \(t\cdot \varphi(y)\) is
  defined and then \(\varphi(t\cdot y) = t\cdot \varphi(y)\).  Thus
  the domain of the partial homeomorphism of~\(Y\) associated to
  \(t\in S\) is the whole \(\varphi\)\nb-preimage of the
  corresponding domain in~\(Y'\).  A map \(\psi\colon Y\to Z\) is
  \emph{\(S\)\nb-invariant} if \(\psi(t\cdot y) = \psi(y)\) for all
  \(t\in S\), \(y\in \mathrm{dom} \vartheta_t\).
\end{definition}

By definition, an \(\IS(F)\)\nb-action on a space~\(Y\)
is equivalent to a map \(\vartheta\colon \Bis(F)\to I(Y)\)
that satisfies \ref{en:F-action_from_theta1}
and~\ref{en:F-action_from_theta2}.

\begin{lemma}
  \label{lem:SF_gives_equivariant_invariant}
  Let \(Y\) and~\(Y'\) be \(F\)\nb-actions and let~\(Z\) be any
  space.  Equip \(Y\) and~\(Y'\) with the induced
  \(\IS(F)\)\nb-actions.  A continuous map \(\varphi\colon Y\to Y'\)
  is \(F\)\nb-equivariant if and only if it is
  \(\IS(F)\)\nb-equivariant and \(\rg'\circ \varphi = \rg\).  A
  continuous map \(\psi\colon Y\to Z\) is \(F\)\nb-invariant if and
  only if it is \(\IS(F)\)\nb-invariant.
\end{lemma}

\begin{proof}
  First let~\(\varphi\) be \(F\)\nb-equivariant.  Then the
  conditions in
  \longref{Lemma}{lem:theta_gives_equivariant_invariant} hold.  The
  domain of~\(\vartheta(\U)\) for \(\U\in\Bis(F)\) is
  \(\rg^{-1}(\s(\U))\).  Since
  \(\rg\circ\varphi = \varphi\circ\rg\), the domain
  of~\(\vartheta(\U)\) in~\(Y'\) is the \(\varphi\)\nb-preimage
  of the corresponding domain in~\(Y\).  Since
  \(\vartheta(\U)\varphi(y) =
  \varphi\bigl(\vartheta(\U)(y)\bigr)\) for all \(y\in Y\) on
  which \(\vartheta(\U)\) is defined, the codomain
  of~\(\vartheta(\U)\) in~\(Y'\) is the whole
  \(\varphi\)\nb-preimage of the corresponding codomain in~\(Y\) as
  well.  Elements of~\(\IS(F)\) act by composites of the partial
  homeomorphisms \(\vartheta(\U)\) and~\(\vartheta(\U)^*\).
  An induction argument on the length of such a composite shows that
  \(t\cdot \varphi(y)\) is defined if and only if
  \(\varphi(t\cdot y)\) is defined, and
  \(\varphi(t\cdot y) = t\cdot \varphi(y)\) for all \(t\in \IS(F)\),
  \(y\in Y\).  Thus~\(\varphi\) is \(\IS(F)\)\nb-equivariant if it is
  \(F\)\nb-equivariant.  The converse implication and the statement
  for invariant maps follow immediately from
  \longref{Lemma}{lem:theta_gives_equivariant_invariant}.
\end{proof}

\begin{definition}[\cite{Exel:Inverse_combinatorial}]
  \label{trangp}
  Let~\(S\) be an inverse semigroup and let~\(Y\) be a topological space
  with an action \(\vartheta\colon S\to I(Y)\).  The
  \emph{transformation groupoid} \(S \ltimes Y\) is defined as
  follows.  Its object space is~\(Y\), and its set of arrows is the
  set of equivalence classes of pairs \((t,x)\) for \(t\in S\) and
  \(x\in \mathrm{dom}(\vartheta_t)\), where \((t,x) \sim (u,x')\) if
  \(x=x'\) and there is an idempotent element \(e\in E(S)\) such
  that \(\vartheta_e\) is defined at \(x=x'\) and \(t e = u e\).  We
  define the range and source maps by \(\s(t,x) \defeq x\),
  \(\rg(t,x) = \vartheta_t(x)\) and
  \((u,\vartheta_t(x))\cdot (t,x) = (u t,x)\) whenever this is
  defined.  There is a unique topology on the arrow space that makes
  \(S\ltimes Y\) an étale groupoid.
\end{definition}

The following proposition describes the transformation groupoid
through a universal property:

\begin{proposition}
  \label{pro:isg_trafo_universal}
  Let~\(S\) be an inverse semigroup and let~\(X\) be a space with an
  \(S\)\nb-action~\(\vartheta_X\) by partial homeomorphisms.
  Let~\(Y\) be a space.  There is a natural bijection between
  actions of the transformation groupoid \(S\ltimes X\) on~\(Y\) and
  pairs \((\vartheta_Y,f)\) consisting of an action~\(\vartheta_Y\)
  of~\(S\) on~\(Y\) and an \(S\)\nb-equivariant map
  \(f\colon Y\to X\).  A map \(Y\to Z\) is \(S\ltimes X\)-invariant
  if and only if it is \(S\)\nb-invariant.
\end{proposition}

\begin{proof}
  Assume that~\(S\ltimes X\) acts on~\(Y\).  Let \(f\colon Y\to X\)
  be the anchor map of the action.  Any \(t\in S\) gives a
  slice~\(\Theta_t\) of~\(S\ltimes X\).  We define a partial
  map~\(\vartheta_{Y,t}\) of~\(Y\) with domain
  \(f^{-1}(\vartheta_X(t^* t))\) by
  \(\vartheta_{Y,t}(y) \defeq \gamma\cdot y\) for the unique
  \(\gamma\in\Theta_t\) with \(\s(\gamma) = f(y)\) in~\(X\).  This
  is a partial homeomorphism because~\(\vartheta_{Y,t^*}\) is a
  partial inverse for it.  The partial
  homeomorphisms~\(\vartheta_{Y,t}\) for \(t\in S\) define an
  action~\(\vartheta_Y\) of~\(S\) on~\(Y\) such that~\(f\) is
  \(S\)\nb-equivariant.  Conversely, let \((\vartheta_Y,f)\) be
  given.  We are going to define an action of~\(S\ltimes X\) with
  anchor map~\(f\).  Let \(\gamma\in S\ltimes X\) and \(y\in Y\)
  satisfy \(\s(\gamma) = f(y)\).  There is \(t\in S\) with
  \(\gamma\in \Theta_t\).  Since \(\s(\Theta_t)\) is the open subset
  in~\(X\) corresponding to the idempotent element \(t^* t \in S\),
  \(f(y) = \s(\gamma) \in \vartheta_X(t^* t)\).  Then
  \(y\in \vartheta_Y(t^* t)\) because~\(f\) is \(S\)\nb-equivariant.
  So \(\gamma\cdot y\defeq \vartheta_{Y,t}(y)\) is defined.  If
  \(\gamma\in \Theta_t\cap \Theta_u\), then there is an idempotent
  element \(e\in S\) with \(t e = u e\) and
  \(\s(\gamma)\in \vartheta_X(e)\).  Then \(y\in \vartheta_Y(e)\) as
  well and hence
  \(\vartheta_{Y,t}(y) = \vartheta_{Y,t e}(y) = \vartheta_{Y,u e}(y)
  = \vartheta_{Y,u}(y)\).  Thus \(\gamma\cdot y\) does not depend on
  the choice of \(t\in S\) with \(y\in\Theta_t\).  The
  multiplication map \((S\ltimes X) \times_{\s,X,f} Y \to Y\) is
  continuous because this holds on each
  \(\Theta_t\times_{s,X,f} Y\).  Routine computations show that the
  multiplication satisfies
  \(\gamma_1\cdot (\gamma_2\cdot y) = (\gamma_1\cdot \gamma_2)\cdot
  y\) for
  \((\gamma_1,\gamma_2,y) \in (S\ltimes X)\times_{\s,X,\rg}
  (S\ltimes X) \times_{\s,X,f} Y\) and \(1_{f(y)}\cdot y = y\) for
  all \(y\in Y\).  Thus the pair \((\vartheta_Y,f)\) gives rise to
  an action of \(S\ltimes X\).  The two constructions above are
  inverse to each other, so that we get the desired bijection.  Both
  are natural; that is, a map \(\varphi\colon Y\to Y'\) is
  equivariant with respect to actions of \(S\ltimes X\) if and only
  if it is \(S\)\nb-equivariant and satisfies
  \(f'\circ \varphi = f\).  Similarly, a map \(\psi\colon Y\to Z\)
  is \(S\ltimes X\)\nb-invariant if and only if it is
  \(S\)\nb-invariant.
\end{proof}

We are ready to reduce the construction of a groupoid model to the
construction of a universal action:

\begin{proposition}
  \label{pro:groupoid_model_from_universal_F-action}
  Let~\(\Omega\) be a universal \(F\)\nb-action.  Equip~\(\Omega\)
  with the action of the inverse semigroup~\(\IS(F)\) associated to
  the \(F\)\nb-action.  The transformation groupoid
  \(\IS(F)\ltimes\Omega\) is a groupoid model for~\(F\).
\end{proposition}

\begin{proof}
  Let~\(Y\) be any space with an \(F\)\nb-action.  Since~\(\Omega\)
  is universal, there is a unique \(F\)\nb-equivariant, continuous
  map \(f\colon Y\to\Omega\).  This map is \(\IS(F)\)-equivariant by
  \longref{Lemma}{lem:SF_gives_equivariant_invariant}.  By
  \longref{Proposition}{pro:isg_trafo_universal}, \(f\) and the
  \(\IS(F)\)-action on~\(Y\) are equivalent to an action of the
  groupoid~\(\IS(F)\ltimes\Omega\).  Conversely, let~\(Y\) carry an
  action of \(\IS(F)\ltimes\Omega\).  We turn this into an action
  of~\(\IS(F)\) and an \(\IS(F)\)\nb-equivariant map
  \(f\colon Y\to\Omega\) by
  \longref{Proposition}{pro:isg_trafo_universal}.  We compose~\(f\)
  with the map
  \(\rg_\Omega \colon \Omega\to\bigsqcup_{x\in\Cat^0} \Gr_x^0\) to
  get \(\rg\colon Y\to\bigsqcup_{x\in\Cat^0} \Gr_x^0\).  We restrict
  the action of~\(\IS(F)\) to generators to get
  \(\vartheta\colon \Bis(F)\to I(Y)\).  We claim that \(\rg\)
  and~\(\vartheta\) satisfy the conditions in
  \longref{Lemma}{lem:F-action_from_theta}, so that they come from
  an \(F\)\nb-action on~\(Y\).

  Since~\(\vartheta\) comes from an action of~\(\IS(F)\), the
  conditions \ref{en:F-action_from_theta1}
  and~\ref{en:F-action_from_theta2} in
  \longref{Lemma}{lem:F-action_from_theta} hold.  The codomains
  of~\(\vartheta(\U)\) in~\(Y\) are the preimages of the
  corresponding codomains in~\(\Omega\) because~\(f\) is
  \(\IS(F)\)\nb-equivariant.  Hence~\(Y\)
  inherits~\ref{en:F-action_from_theta3} from~\(\Omega\).
  \longref{Condition}{en:F-action_from_theta5} holds because it
  holds for~\(r_\Omega\) and~\(f\) is equivariant.  Thus \(\rg\)
  and~\(\vartheta\) come from an \(F\)\nb-action.  The maps above
  between \(\IS(F)\ltimes\Omega\)-actions and \(F\)\nb-actions are
  inverse to each other.  The equivariant maps with respect to \(F\)
  and \(\IS(F)\ltimes \Omega\) are the same by
  \longref{Lemma}{lem:SF_gives_equivariant_invariant} and
  \longref{Proposition}{pro:isg_trafo_universal}.  Thus
  \(\IS(F)\ltimes \Omega\) is a groupoid model for~\(F\).
\end{proof}

\section{Groupoid models in the tight case}
\label{sec:limit_tight}

The construction of groupoid models for general diagrams contains
two complicated pieces that we do not understand well in general:
the inverse semigroup~\(\IS(F)\) and the space~\(\Omega\) that carries
the universal \(F\)\nb-action.  We are now going to show that the
\(F\)\nb-action~\(\Omega\) is very simple for diagrams of tight
correspondences.

Let~\(\Cat\) be a small category and let \(F=(\Gr_x,\Bisp_g,\mu_{g,h})\)
describe a \(\Cat\)\nb-shaped diagram of correspondences as in
\longref{Proposition}{pro:diagrams_in_Grcat}.  We assume that all
correspondences~\(\Bisp_g\) are tight.  Let \(\Omega_x \defeq \Gr_x^0\)
with the identity map \(\rg\colon \Omega_x\to \Gr_x^0\), and let
\(\Omega \defeq \bigsqcup_{x\in \Cat^0} \Omega_x = \bigsqcup_{x\in \Cat^0}
\Gr_x^0\).  For an arrow \(g\colon x'\leftarrow x\) in~\(\Cat\),
define
\[
\alpha_g\colon \Bisp_g \times_{\s,\Gr_x^0,\id} \Gr_x^0
\cong \Bisp_g \to \Gr_{x'}^0,\qquad
(\gamma, \s(\gamma)) \mapsto \rg(\gamma).
\]
That is, \(\gamma\cdot y\) is defined if and only if
\(y=\s(\gamma)\), and \(\gamma\cdot \s(\gamma) = \rg(\gamma)\) for
all \(\gamma\in\Bisp_g\).  This satisfies
\(\rg(\gamma\cdot y) = \rg(\gamma)\) and
\(\gamma_1\cdot (\gamma_2\cdot y) = (\gamma_1\cdot \gamma_2)\cdot
y\) because \(\rg(\gamma_1) = \rg(\gamma_1\cdot \gamma_2)\).  The
assumption that~\(\Bisp_g\) is tight means that~\(\alpha_g\)
descends to a homeomorphism from
\(\Bisp_g/\Gr_{\s(g)}^0 \cong \Bisp_g \Grcomp_{\Gr_{\s(g)}}
\Gr_{\s(g)}^0\) to~\(\Gr_{\rg(g)}^0\).  Hence the above data is an
\(F\)\nb-action as in
\longref{Definition}{def:diagram_dynamical_system}.  This breaks
down for diagrams that are not tight.

\begin{theorem}
  \label{the:universal_action_tight}
  For a tight diagram~\(F\),
  the \(F\)\nb-action
  on \(\Omega\defeq \bigsqcup_{x\in \Cat^0} \Gr_x^0\)
  defined above is universal.  The transformation groupoid
  \(\IS(F)\ltimes \Omega\)
  for the induced action of the inverse semigroup~\(\IS(F)\)
  on~\(\Omega\) is a groupoid model for the diagram~\(F\).
\end{theorem}

\begin{proof}
  Any \(F\)\nb-action~\(Y\) comes with a clopen partition \(Y =
  \bigsqcup_{x\in\Cat^0} Y_x\) with maps \(\rg\colon Y_x\to\Gr_x^0\),
  and an \(F\)\nb-equivariant map must respect these maps.  Thus the
  only map \(Y\to \bigsqcup_{x\in\Cat^0} \Gr_x^0\) that has a chance
  to be \(F\)\nb-equivariant is the map that restricts to~\(\rg\) on
  each~\(Y_x\).  This map is indeed \(F\)\nb-equivariant because
  \(\rg(\gamma\cdot y) = \rg(\gamma)\) for each \(\gamma\in\Bisp_g\),
  \(y\in Y_x\) with \(\s(\gamma)= \rg(y)\).  Thus~\(\Omega\) is the
  universal \(F\)\nb-action.  And \(\IS(F)\ltimes \Omega\) is a groupoid
  model by
  \longref{Proposition}{pro:groupoid_model_from_universal_F-action}.
\end{proof}

In \longref{Section}{sec:tight_Ore_model}, we will describe the
groupoid model differently for tight diagrams where~\(\Cat\)
satisfies the ``right Ore conditions''.  This alternative
description gives the arrow space of a groupoid directly, without
mentioning inverse semigroups.

% Since groupoid models are unique up to isomorphism by
% \longref{Lemma}{lem:groupoid_model}, the space \(\Omega =
% \bigsqcup_{x\in\Cat^0} \Gr_x^0\) for a tight diagram is homeomorphic
% to the space~\(\Omega_F\) in \longref{Theorem}{the:groupoid_model_exists}.

The monoid~\(\Bis(F)\) acts on~\(\Omega\) by partial homeomorphisms
by the constructions above.  These partial homeomorphisms generate a
pseudogroup by adding their partial inverses and closing again under
composition.  This pseudogroup gives rise to a germ groupoid~\(H\),
whose objects are points in~\(\Omega\) and whose arrows are the
germs of the partial homeomorphisms in this pseudogroup.  This germ
groupoid is pretty close to the original data and should therefore
be quite computable.  How is it related to our groupoid model
\(\IS(F)\ltimes \Omega\)?

The action of~\(\Bis(F)\) on~\(\Omega\) extends to a homomorphism
\(\IS(F) \to I(\Omega)\).  This induces a surjective functor
\(\IS(F)\ltimes\Omega \prto H\) that is the identity on objects.

\begin{proposition}
  \label{pro:tight_effective_quotient}
  The kernel of the surjective functor
  \(\IS(F)\ltimes\Omega \prto H\) is the open isotropy group bundle,
  that is, the interior of the set of arrows~\(\gamma\) with
  \(\s(\gamma)=\rg(\gamma)\) in \(\IS(F)\ltimes\Omega\).  Hence
  \(\IS(F)\ltimes\Omega \cong H\) if and only if \(\IS(F)\ltimes\Omega\)
  is effective.
\end{proposition}

\begin{proof}
  The germ groupoid~\(H\) is effective by the definition of the germ
  relation.  Hence the map \(\IS(F)\ltimes\Omega \prto H\) annihilates
  the open isotropy group bundle.  Conversely, if an arrow is not in
  the interior of the isotropy group bundle, then the germ
  of its action on~\(\Omega\) is non-trivial, and then its image
  in~\(H\) cannot be trivial.
\end{proof}

Thus~\(H\) is the \emph{effective quotient} of the étale topological
groupoid~\(\IS(F)\ltimes \Omega\), that is, the largest quotient
that is effective.  Several important results about groupoid
\(\Cst\)\nb-algebras hold only for effective groupoids.  In
particular, this is so for criteria when \(\Cst(\Gr)\) is simple or
purely infinite.  Hence examples where \(\IS(F)\ltimes \Omega\) is
different from the germ groupoid~\(H\) may be hard to analyse.

Ruy Exel suggested to consider the example of an action of a free
monoid on~\(n\) generators on a space~\(X\) by local homeomorphisms,
that is, tight groupoid correspondences \(X\leftarrow X\).  Such an
action is determined by \(n\)~arbitrary local homeomorphisms (see
\longref{Section}{sec:free_monoid_action}).  We do not know general criteria
for the resulting groupoid \(\IS(F)\ltimes \Omega\) to be
effective.

\section{Groupoid models and fundamental groups for complexes of groups}
\label{sec:group_complex}

We have seen in \longref{Section}{sec:tight_group_diagram} that a complex of
groups is equivalent to a tight diagram of groupoid correspondences
with some extra properties.  We are going to relate the groupoid
model of such a diagram to the fundamental group of the
corresponding complex of groups.  This fundamental group is defined
in \cite{Bridson-Haefliger}*{Definition 3.1 in Chapter
  III.\(\mathcal{C}\)} by generators and relations.  The following
definition adapts the one in~\cite{Bridson-Haefliger} to
our conventions:

\begin{definition}
  \label{def:fundamental_group_of_complex}
  Let~\(\Cat\) be a small category and let
  \((\Gr_x,\varphi_g,u_{g,h})\) describe a \(\Cat\)\nb-shaped
  diagram of tight correspondences between groups as in
  \longref{Section}{sec:tight_group_diagram}.  The \emph{fundamental group}
  of this diagram is the group~\(\Pi_1\) generated by the set
  \(\Cat\sqcup \bigsqcup_{x\in\Cat^0} \Gr_x\) with the following
  relations:
  \begin{enumerate}
  \item \label{def:fundamental_group_of_complex_1}%
    \(\gamma_1\cdot \gamma_2 = \gamma_1\gamma_2\) for all
    \(\gamma_1,\gamma_2\in\Gr_x\) for the same~\(x\); that is, the
    canonical maps \(\Gr_x \to \Pi_1\) are group homomorphisms;
  \item \label{def:fundamental_group_of_complex_2}%
    \(g\cdot \varphi_g(\gamma) = \gamma\cdot g\) for all
    \(g\in\Cat\), \(\gamma\in\Gr_{\rg(g)}\);
  \item \label{def:fundamental_group_of_complex_3}%
    \(g\cdot h = (g h) \cdot u_{g,h}\) for all
    \((g,h)\in\Cat^2\);
  \item \label{def:fundamental_group_of_complex_4}%
    \(x=1\) for each unit arrow \(x\in\Cat^0\subseteq \Cat\).
  \end{enumerate}
\end{definition}

%  the universal property
%   that for each group~\(H\), group homomorphisms \(\Pi_1\to H\) are
%   in natural bijection with families of maps \(\psi_x\colon \Gr_x\to
%   H\) for \(x\in\Cat^0\) and \(v_g\in H\) with the following
%   properties:
%   \begin{enumerate}
%   \item each~\(\psi_x\) is a group homomorphism;
%   \item \(\Ad(v_g)\circ \psi_{\s(g)} \circ \varphi_g =
%     \psi_{\rg(g)}\) for all \(g\in\Cat\);
%   \item \(v_g\cdot v_h = v_{g h} \cdot \psi_z (u_{g,h})\) for all
%     \((g,h)\in\Cat^2\);
%   \item \(v_x = 1\) for unit arrows;
%   \end{enumerate}
% \end{definition}

The generators \(x\in\Cat^0\) are redundant by~\ref{def:fundamental_group_of_complex_4}.  Relation~\ref{def:fundamental_group_of_complex_2}
for \(g\in\Cat^0\) follows from~\ref{def:fundamental_group_of_complex_4}, and so does~\ref{def:fundamental_group_of_complex_3} if
\(g\in\Cat^0\) or \(h\in\Cat^0\) because \(u_{g,h}=1\) in both
cases.  If~\(\Cat\) has no loops, as assumed
in~\cite{Bridson-Haefliger}, then \(g h\in\Cat^0\) only if
\(g\in\Cat^0\) and \(h\in\Cat^0\); hence we may just leave out the
units in~\(\Cat^0\) among the generators and leave out~\ref{def:fundamental_group_of_complex_4} and
require \ref{def:fundamental_group_of_complex_2} and~\ref{def:fundamental_group_of_complex_3} only for \(g,h\in\Cat\setminus\Cat^0\).
If~\(\Cat\) has loops, we may still leave out the units among the
generators and leave out~\ref{def:fundamental_group_of_complex_4}
and require \ref{def:fundamental_group_of_complex_2}
and~\ref{def:fundamental_group_of_complex_3} only for
\(g,h\in\Cat\setminus\Cat^0\).  But then we must clarify that~\ref{def:fundamental_group_of_complex_3}
says \(g\cdot h = u_{g,h}\) if~\(g h\) is a unit.  When we compare
our definition to that in~\cite{Bridson-Haefliger}, we must take
into account that we work with the reverse composition of arrows
in~\(\Cat\) and that our~\(u_{g,h}\) is the inverse of the twisting
element~\(g_{b,a}^{-1}\) in~\cite{Bridson-Haefliger}.  The
definitions become equivalent if we put \(g^- = g\) and \(g^+ =
g^{-1}\).

How is the fundamental group related to the groupoid model?
The description of a groupoid model for the diagram
\((\Gr_x,\varphi_g,u_{g,h})\) is like the presentation above, but
taking into account the ranges and sources of the generators.  First
construct a directed graph with vertex set~\(\Cat^0\) and set of
edges \(\Cat\sqcup \bigsqcup_{x\in\Cat^0} \Gr_x\), where \(\s(g)\)
and \(\rg(g)\) have the usual meaning for \(g\in\Cat\) and
\(\s(\gamma)= \rg(\gamma) = x\) if \(\gamma\in\Gr_x\).  Let~\(P\) be
the path groupoid of this directed graph; that is, an arrow in~\(P\)
is a reduced composable word, where each letter is an edge or its
formal inverse, and where ``reduced'' means that the combinations
\(a\cdot a^{-1}\) and \(a^{-1}\cdot a\) for generators~\(a\) are
forbidden.

\begin{definition}
  \label{def:complex_groups_to_groupoid}
  Let~\(\Gr[U]\) be the quotient of~\(P\) by the family of normal
  subgroups in~\(P\) that is generated by the following elements:
  \begin{enumerate}
  \item \(\gamma_1\cdot \gamma_2 \cdot (\gamma_1\gamma_2)^{-1}\) for
    all \(\gamma_1,\gamma_2\in\Gr_x\) for the same~\(x\);
  \item \(g\cdot \varphi_g(\gamma) \cdot g^{-1} \cdot \gamma^{-1}\)
    for all \(g\in\Cat\), \(\gamma\in\Gr_{\rg(g)}\);
  \item \(u_{g,h}^{-1}\cdot (g h)^{-1} \cdot g\cdot h\) for all
    \((g,h)\in\Cat^2\);
  \item \(x=1\) for each unit arrow \(x\in\Cat^0\).
  \end{enumerate}
\end{definition}

\begin{proposition}
  \label{pro:groupoid_model_tight_group_diagram}
  The groupoid~\(\Gr[U]\) described above is a groupoid model for
  the diagram of tight group correspondences described by
  \((\Gr_x,\varphi_g,u_{g,h})\).
\end{proposition}

We prove this in two ways, either directly or using the description
of the groupoid model in
\longref{Theorem}{the:universal_action_tight}.

\begin{proof}
  Let~\(Y\) be a space with an action of our diagram~\(F\).  The
  correspondences in our diagram are of the form \(\Bisp_g \defeq
  {}_{\varphi_g}\Gr_{\s(g)}\).  The multiplication map
  \[
    {}_{\varphi_g} \Gr_{\s(g)} \Grcomp Y_{\s(g)} \to Y_{\rg(g)},\qquad
    (\gamma,y)\mapsto \gamma\cdot y,
  \]
  is a homeomorphism that intertwines the canonical left actions
  of~\(\Gr_{\rg(g)}\) on
  \({}_{\varphi_g} \Gr_{\s(g)} \Grcomp Y_{\s(g)} \cong Y_{\s(g)}\)
  and~\(Y_{\rg(g)}\), which are given by
  \(\gamma\cdot y \defeq \varphi_g(\gamma) y\) and
  \(\gamma\cdot y' \defeq \gamma y'\) for \(\gamma\in\Gr_{\rg(g)}\),
  \(y\in Y_{\s(g)}\) and \(y'\in Y_{\rg(g)}\).  So the homeomorphism
  \({}_{\varphi_g} \Gr_{\s(g)} \Grcomp Y_{\s(g)} \congto
  Y_{\rg(g)}\) that is part of an \(F\)\nb-action on~\(Y\) is
  equivalent to a homeomorphism
  \(\psi_g\colon Y_{\s(g)} \congto Y_{\rg(g)}\) with
  \(\gamma\cdot \psi_g(y) = \psi_g(\varphi_g(\gamma)\cdot y)\) for
  all \(\gamma\in \Gr_{\rg(g)}\), \(y\in Y_{\s(g)}\).  The
  condition~\ref{en:diagram_dynamical_system1} in
  \longref{Definition}{def:diagram_dynamical_system} is equivalent
  to
  \(\psi_{g_1}\circ \psi_{g_2}(y) = \psi_{g_1 g_2}(u_{g_1,g_2}\cdot
  y)\) for all composable \(g_1,g_2\in \Cat\) and
  \(y\in Y_{\s(h)}\).

  Thus an \(F\)\nb-action on~\(Y\) is equivalent to the following data:
  \begin{enumerate}
  \item a disjoint union decomposition \(Y = \bigsqcup_{x\in \Cat^0}
    Y_x\);
  \item group actions of~\(\Gr_x\) on~\(Y_x\) for each \(x\in\Cat^0\);
  \item homeomorphisms \(\psi_g\colon Y_{\s(g)} \congto Y_{\rg(g)}\)
    for all \(g\in\Cat\),
  \end{enumerate}
  subject to the conditions \(\psi_x=\id\) for all \(x\in\Cat^0\),
  \(\gamma\cdot \psi_g(y) = \psi_g(\varphi_g(\gamma)\cdot y)\) for
  all \(\gamma\in \Gr_{\rg(g)}\), \(y\in Y_{\s(g)}\), and
  \(\psi_g\circ \psi_h(y) = \psi_{g h}(u_{g,h}\cdot y)\) for all
  composable \(g,h\in \Cat\) and \(y\in Y_{\s(h)}\).  We claim that
  this is equivalent to an action of the groupoid~\(\Gr[U]\)
  on~\(Y\).  First, the anchor map \(\rg\colon Y\to
  \Gr[U]^0=\Cat^0\) is equivalent to the disjoint union
  decomposition \(Y = \bigsqcup_{x\in\Cat^0} Y_x\) with \(Y_x =
  \rg^{-1}(\{x\})\).  Secondly, an action of~\(\Gr[U]\) is
  equivalent to an action of its generators, such that the relations
  hold.  We let \(g\in\Cat\) act by the map~\(\psi_g\) and the
  elements of~\(\Gr_x\) for \(x\in\Cat^0\) by the given group
  actions.  The relations that define~\(\Gr[U]\) are exactly those
  that are needed for this to give an \(F\)\nb-action.
\end{proof}

The general theory also quickly leads to the groupoid
model~\(\Gr[U]\) above.  First of all, the universal action
described in \longref{Theorem}{the:universal_action_tight} takes
place on the discrete set
\(\bigsqcup_{x\in\Cat^0} \Gr_x^0 \cong \Cat^0\).  Since
each~\(\Gr_x^0\) is a singleton, slices of the
correspondences~\(\Bisp_g\) are either singletons or empty.  When we
identify \(\Bisp_g \cong {}_{\varphi_g}\Gr_{\s(g)}\), then we may
decompose the slice corresponding to
\(\gamma\in\Bisp_g \cong \Gr_{\s(g)}\) as \(g\cdot \gamma\),
where~\(g\) denotes the slice of~\(\Bisp_g\) that is represented by
the neutral element \(1\in \Gr_{\s(g)}\).  Then one shows that the
semigroup of slices is the quotient of the free semigroup on the set
\(\Cat \sqcup \bigsqcup_{x\in\Cat^0} \Gr_x\sqcup \{\emptyset\}\) by
the congruence relation generated by the relations \(x=1\) for
\(x\in\Cat^0 \subseteq \Cat\),
\(\gamma_1\cdot \gamma_2 = \gamma_1\gamma_2\) for all
\(\gamma_1,\gamma_2\in\Gr_x\) for the same \(x\in\Cat^0\),
\(\gamma\cdot g = g\cdot \varphi(\gamma)\) for \(g\in\Cat\),
\(\gamma\in\Gr_{\rg(g)}\), \(g\cdot h = (gh) u_{g,h}\) for
composable \(g,h\in \Cat\), and
\(\gamma_1\cdot \gamma_2=\emptyset\), \(g\cdot \gamma=\emptyset\),
\(\gamma\cdot g=\emptyset\), \(g\cdot h=\emptyset\) in the cases not
listed above.  When we generate an inverse semigroup from this, we
adjoin the extra generators \(g^* = g^{-1}\) for \(g\in\Cat\).  We
need no extra generators~\(\gamma^*\) for \(\gamma\in\Gr_x\) because
\(\Gr_x\) is already a group and so
\(\gamma^* = \gamma^{-1} \in\Gr_x\).  The extra relations
\(\alpha^*\beta = \braket{\alpha}{\beta}\) for slices
\(\alpha,\beta\in\Bis(\Bisp_g)\) for the same \(g\in\Cat\) are
redundant in this case because we may write
\(\alpha=g\cdot \gamma_1\), \(\beta=g\cdot \gamma_2\), and then
\(\alpha^* = \gamma_1^{-1}\cdot g^{-1}\) with
\(\gamma_1^{-1} \in \Gr_{\s(g)}\) and
\(\alpha^* \beta = \gamma_1^{-1}\cdot g^{-1} \cdot g\cdot \gamma_2 =
\gamma_1^{-1} \gamma_2 = \braket{g\cdot \gamma_1}{g\cdot\gamma_2}\).
This explicit description of \(\Omega\) and
\(\IS(F)\) implies that \(\Omega \rtimes \IS(F)\) is isomorphic
to the groupoid~\(\Gr[U]\) described above.

There is usually no close relationship between the groupoid
model~\(\Gr[U]\) and the fundamental group~\(\Pi_1\).  For instance,
if the underlying category~\(\Cat\) has only identity arrows, then
the resulting complex of groups is simply a set of groups
\((\Gr_x)_{x\in\Cat^0}\) without further structure.  Then the
fundamental group is the free product of these groups, whereas the
groupoid model is their disjoint union by
\longref{Proposition}{pro:union_groupoid}.  We may, however, enlarge the diagram
so that the groupoid model of the enlarged diagram is equivalent
to~\(\Pi_1\).  Namely, let \(\Cat^\triangleright\) be the category
with object set \(\Cat^0\sqcup \{\infty\}\) and with arrows set
\(\Cat\times\{0,\infty\} \sqcup\{\id_\infty\}\), source and range
maps
\[
\s(g,0)=\s(g),\qquad
\rg(g,0)=\rg(g),\qquad
\s(g,\infty)=\s(g),\qquad
\rg(g,\infty)=\infty,
\]
and the composition \((g,0)\cdot (h,0) =(g h,0)\) and
\((g,\infty)\cdot (h,0) =(g h,\infty)\).  In other words, we take
the category generated freely by~\(\Cat\) and new arrows
\((x,\infty)\colon x\to\infty\) for all \(x\in\Cat^0\).  We impose
no relations besides those in~\(\Cat\), so the arrows
\((x,\infty)\circ (g,0) = (g,\infty)\) for \(g\in\Cat\) with
\(\rg(g)=x\) are all distinct.  By construction, the only arrow with
\(\s(g)=\infty\) is \(g=\id_\infty\).  We extend a given diagram
to~\(\Cat^\triangleright\) as follows.  Let
\(\Gr_\infty\defeq\{1\}\) and let \(\varphi_{(g,\infty)}\) for
\(g\in\Cat\) be the unique group homomorphism \(\{1\}\to \Gr_x\).
Let
\[
u_{(g,0),(h,0)} = u_{(g,\infty),(h,0)} = u_{g,h}
\]
for all composable \(g,h\in\Cat\) and let \(u_{\id_\infty,g} = 1\)
and \(u_{g,\id_\infty} = 1\) if \(\rg(g)=\infty\) or
\(\s(g)=\infty\), respectively.  This defines \(u_{g,h}\) for all
\((g,h) \in (\Cat^\triangleright)^2\), and the conditions needed for
a diagram of tight group correspondences hold because they hold for
the given diagram over~\(\Cat\).

\begin{theorem}
  \label{the:action_complex_of_groups}
  Let \((\Gr_x,\varphi_g,u_{g,h})\)
  be a diagram of tight group correspondences.
  Let~\(\Gr[U]^\triangleright\)
  be the groupoid model for the extension of this diagram
  to~\(\Cat^\triangleright\).
  The groupoid~\(\Gr[U]^\triangleright\)
  is isomorphic to the product of the fundamental group~\(\Pi_1\)
  of \((\Gr_x,\varphi_g,u_{g,h})\)
  with the pair groupoid of \(\Cat^0\sqcup\{\infty\}\).
  So it is equivalent to~\(\Pi_1\).
\end{theorem}

\begin{proof}
  For any object~\(x\)
  in~\(\Cat^\triangleright\),
  there is a canonical arrow \(h_x\colon x\to\infty\)
  in~\(\Cat^\triangleright\),
  namely, the arrow \((x,\infty)\colon x\to\infty\)
  for \(x\in\Cat^0\)
  or \(\id_\infty\)
  for \(x=\infty\).
  Hence the groupoid~\(\Gr[U]^\triangleright\)
  is transitive.  Thus it is isomorphic to the product of the pair
  groupoid on the set \(\Cat^0\sqcup \{\infty\}\)
  and the isotropy group \(\Gr[U]^\triangleright_\infty\)
  at the object~\(\infty\).
  Let~\(P\)
  be the path groupoid in the definition of~\(\Gr[U]^\triangleright\).
  Given a reduced path in~\(P\),
  put \(h_x^{-1} h_x\)
  for the appropriate object~\(x\)
  between any two factors.  This gives a product of
  factors~\(t^{\pm1}\),
  where~\(t\)
  is either \(h_{\rg(g)} (g,0) h_{\s(g)}^{-1}\)
  for some \(g\in\Cat\),
  or \(h_x \gamma h_x^{-1}\)
  for some \(x\in\Cat^0\),
  \(\gamma\in\Gr_x\),
  or \((g,\infty) h_{\s(g)}^{-1}\)
  for some \(g\in\Cat\).
  We may rewrite the last term as \(h_{\rg(g)} (g,0) h_{\s(g)}^{-1}\)
  because \(u(\rg(g),g)=1\)
  for all \(g\in\Cat\)
  and so \((g,\infty) = h_{\rg(g)} \cdot (g,0)\)
  holds in \(\Gr[U]^\triangleright_\infty\).
  So we only need the first two types of generators, which are in
  bijection with \( \Cat \sqcup \bigsqcup_{x\in\Gr^0} \Gr_x\).
  Decorating the relations defining the groupoid~
  \(\Gr[U]^\triangleright\)
  with \(h_x^{-1} h_x\)
  appropriately, we see that~\(\Gr[U]^\triangleright_\infty\)
  is the group generated by the set
  \(\bigsqcup_{x\in\Gr^0} \Gr_x \sqcup \Cat\) and the relations in
  \longref{Definition}{def:fundamental_group_of_complex}.
\end{proof}

We may also understand \longref{Theorem}{the:action_complex_of_groups} by
looking at actions of the extended diagram.  Such an action occurs
on a space \(Y=Y_\infty \sqcup \bigsqcup_{x\in\Cat^0} Y_x\).  The
arrows \((x,\infty)\colon x\to\infty\) in~\(\Cat^\triangleright\)
act by homeomorphisms \(Y_x\congto Y_\infty\).  We use these
homeomorphisms to identify all the~\(Y_x\) with~\(Y_\infty\), and to
transfer the actions of~\(\Gr_x\) on~\(Y_x\) and the homeomorphisms
\(Y_{\s(g)} \congto Y_{\rg(g)}\) for arrows \((g,0)\in
\Cat^\triangleright\) to actions of~\(\Gr_x\) on~\(Y_\infty\) and
homeomorphisms of~\(Y_\infty\).  The defining relations of~\(\Pi_1\)
are exactly the relations that these homeomorphisms of~\(Y_\infty\)
must satisfy to give an action of the extended diagram.

\section{Diagrams of Ore shape}
\label{sec:Ore_shape}

Let \(F=(\Gr_x,\Bisp_g,\mu_{g,h})\) be a \(\Cat\)\nb-shaped diagram
that is not tight.  Then \(\bigsqcup_{x\in\Cat^0} \Gr_x^0\) does not
carry an \(F\)\nb-action.  So the underlying space of the universal
\(F\)\nb-action must be more complicated.  We are going to construct
a better approximation to this universal \(F\)\nb-action.
If~\(\Cat\) satisfies certain conditions that generalise the
definition of an Ore monoid in~\cite{Albandik-Meyer:Product}, then
this better approximation is the universal \(F\)\nb-action.

Let \(Y= \bigsqcup_{x\in \Cat^0} Y_x\)
be a topological space with an \(F\)\nb-action.
This \(F\)\nb-action
is equivalent to a family of homeomorphisms
\(V_g\colon \Bisp_g \Grcomp Y_{\s(g)} \congto Y_{\rg(g)}\)
for all \(g\in\Cat\)
such that the multiplication defined by
\(\gamma\cdot y\defeq V_g[\gamma,y]\)
for \(\gamma\in\Bisp_g\),
\(y\in Y_{\s(g)}\)
with \(\s(\gamma) = \rg(y)\)
is associative as in \longref{Definition}{def:diagram_dynamical_system}.  The
spaces \(\Bisp_g \Grcomp Y_{\s(g)}\)
and~\(\Bisp_g/\Gr_{\s(g)}\)
are, respectively, defined as the quotients of
\(\Bisp_g \times_{\s,\Gr_{\s(g)}^0,\rg} Y_{\s(g)}\)
and~\(\Bisp_g\)
by the equivalence relations
\((\gamma,y) \sim (\gamma\cdot\eta,\eta^{-1}\cdot y)\)
and \(\gamma = \gamma\cdot \eta\)
for all \(\gamma\in \Bisp_g\),
\(y\in Y_{\s(g)}\),
\(\eta\in\Gr_{\s(g)}\)
with \(\s(\gamma)=\rg(y) = \rg(\eta)\).
Hence the coordinate projection
\(\Bisp_g \times_{\s,\Gr_{\s(g)}^0,\rg} Y_{\s(g)}\to \Bisp_g\)
descends to a well defined continuous map
\(\Bisp_g \Grcomp Y_{\s(g)} \to \Bisp_g/\Gr_{\s(g)}\).
Together with the homeomorphism~\(V_g\),
this produces a continuous map
\[
\varrho^g\colon Y_{\rg(g)} \xrightarrow[\cong]{V_g^{-1}}
\Bisp_g \Grcomp Y_{\s(g)} \to \Bisp_g/\Gr_{\s(g)}
\]
that maps \(\gamma\cdot y = V_g[\gamma,y]\)
to~\([\gamma]\)
for \(\gamma\in \Bisp_g\),
\(y\in Y_{\s(g)}\) with \(\s(\gamma)=\rg(y)\).

Similarly, the homeomorphism \(\mu_{g,h}\colon \Bisp_g\Grcomp
\Bisp_h \congto \Bisp_{g h}\), \([\gamma_1,\gamma_2]\mapsto
\gamma_1\cdot \gamma_2\), for composable \(g,h\in\Cat\) induces a
well defined continuous map
\begin{equation}
 \label{eq:projection_maps_Ore_diagrams}
\pi_{g,h}\colon \Bisp_{g h}/\Gr_{\s(g h)}
\xrightarrow[\cong]{\mu_{g,h}^{-1}}
\Bisp_g \Grcomp \Bisp_h/\Gr_{\s(h)} \to \Bisp_g/\Gr_{\s(g)},
\qquad
[\gamma_1\cdot \gamma_2]\mapsto [\gamma_1].
\end{equation}

\begin{lemma}
  \label{lem:varrho_g_compatible}
  These maps are related by
  \begin{equation}
    \label{eq:varrho_g_compatible}
    \pi_{g,h}\circ \varrho^{g h} = \varrho^g
  \end{equation}
  for all composable \(g,h\in\Cat\).
\end{lemma}

\begin{proof}
  Any element of~\(Y_{\rg(g)}\) is of the form
  \((\gamma_1\cdot \gamma_2) \cdot y = \gamma_1 \cdot (\gamma_2\cdot
  y)\) for some \(\gamma_1\in \Bisp_g\), \(\gamma_2\in\Bisp_h\),
  \(y\in Y_{\s(h)}\) with \(\s(\gamma_1) = \rg(\gamma_2)\) and
  \(\s(\gamma_2) = \rg(y)\).  The maps \(\varrho^{g h}\)
  and~\(\varrho^g\) on \(Y_{\rg(g)} = Y_{\rg(g h)}\) map
  \(\gamma_1 \cdot \gamma_2\cdot y\) to \([\gamma_1\cdot \gamma_2]\)
  and \([\gamma_1]\), respectively, and
  \(\pi_{g,h}[\gamma_1 \cdot \gamma_2] = [\gamma_1]\).  This
  proves~\eqref{eq:varrho_g_compatible}.
\end{proof}

We now use~\eqref{eq:varrho_g_compatible} to combine all
maps~\(\varrho^g\) for \(g\in\Cat\) with \(\rg(g) = x\) into a
single map from~\(Y_x\) to a certain limit space.  First we describe
the category~\(\Cat[D]_x\) over which we take the limit.  Its object
set is \(\Cat^x \defeq \setgiven{g\in\Cat}{\rg(g)=x}\), and an arrow
\(g_1\to g_2\) in~\(\Cat[D]_x\) is \(h\in\Cat\) with \(g_1 = g_2
h\).  The composition in~\(\Cat[D]_x\) is the multiplication
in~\(\Cat\).  The map that sends an object~\(g\) of~\(\Cat[D]_x\) to
the space~\(\Bisp_g/\Gr_{\s(g)}\) and an arrow \(h\colon g h\to g\)
to the continuous map~\(\pi_{g,h}\) is a diagram of topological
spaces indexed by~\(\Cat[D]_x\), that is,
\[
\pi_{g,h_1} \circ \pi_{g h_1,h_2} = \pi_{g, h_1 h_2}\colon
\Bisp_{g h_1 h_2}/\Gr_{\s(h_2)} \to \Bisp_{g}/\Gr_{\s(g)}
\]
if \(g,h_1,h_2\in\Cat\) are composable; this follows as in the proof
of \longref{Lemma}{lem:varrho_g_compatible}.  Let
\[
\Omega_x \defeq \varprojlim {}(\Bisp_g/\Gr_{\s(g)},\pi_{g,h})
\]
be the limit of this diagram.  The maps \(\varrho^g\colon Y_x \to
\Bisp_g/\Gr_{\s(g)}\) form a cone over the diagram
\((\Bisp_g/\Gr_{\s(g)},\pi_{g,h})\)
by~\eqref{eq:varrho_g_compatible}.  Hence they induce a map
\[
\varrho_x\defeq (\varrho^g)_{g\in\Cat^x} \colon Y_x \to
\varprojlim {}(\Bisp_g/\Gr_{\s(g)},\pi_{g,h}) = \Omega_x.
\]

The construction of \(\varrho^g\) and~\(\varrho_x\) is
natural with respect to \(F\)\nb-equivariant maps; that is, if
\(\varphi\colon Y_1 \to Y_2\) is \(F\)\nb-equivariant and
\(\varrho_i^g\colon Y_{i,x} \to \Bisp_g/\Gr_{\s(g)}\) and
\(\varrho_{i,x}\colon Y_{i,x} \to \Omega_x\) are these canonical
maps for \(i=1,2\), then \(\varrho_2^g \circ \varphi = \varrho_1^g\)
and \(\varrho_{2,x} \circ \varphi = \varrho_{1,x}\).

Let \(\Omega \defeq \bigsqcup_{x\in\Cat^0} \Omega_x\).  Is there
some kind of \(F\)\nb-action on~\(\Omega\)?  Each space
\(\Bisp_g/\Gr_{\s(g)}\) carries a left \(\Gr_{\rg(g)}\)\nb-action,
and the maps~\(\pi_{g,h}\) are equivariant for these actions.  Hence
the limit space~\(\Omega_x\) inherits a canonical action of the
groupoid~\(\Gr_x\).  Its anchor map is the projection to the factor
\(\Bisp_x/\Gr_x \cong \Gr_x^0\) at \(x\in\Cat^x\) in our diagram.

We also need multiplication maps \(\Bisp_g
\times_{\s,\Gr_{\s(g)}^0,\rg} \Omega_{\s(g)} \to \Omega_{\rg(g)}\)
for non-identity arrows \(g\in \Cat\).  Let \(\gamma\in\Bisp_g\),
\(x\defeq \s(g)\), and let \(\omega = (\omega_h)_{h\in\Cat^x} \in
\Omega_x\) with \(\s(\gamma) = \rg(\omega) = \rg(\omega_h)\) for all
\(h\in\Cat^x\).  Then \(\omega_h\in\Bisp_h/\Gr_{\s(h)}\) and so we
get \(\gamma\cdot\omega_h\in\Bisp_{g\cdot h}/\Gr_{\s(g\cdot h)}\) by
choosing a representative of~\(\omega_h\) in~\(\Bisp_h\) and
multiplying.  But to define a point in~\(\Omega_{\rg(g)}\), we need
elements \((\gamma\cdot\omega)_k\) for all \(k\in\Cat^x\).  The
formula above only produces \((\gamma\cdot\omega)_{g\cdot h}\) for
fixed~\(g\) and \(h\in\Cat^{\s(g)}\).  In general, this is not
enough and so our Ansatz does not lead to a natural action of~\(F\)
on~\(\Omega\).  To get such an action, we now assume that all the
categories~\(\Cat[D]_x^\op\) are filtered, compare
also~\cite{Albandik-Meyer:Product}.  The following lemma makes
explicit the usual definition of a filtered category:

\begin{lemma}
  \label{lem:Ore_conditions}
  The categories~\(\Cat[D]_x^\op\) are filtered for all \(x\in\Cat^0\)
  if and only if the following two conditions hold:
  \begin{enumerate}[label=\textup{(\ref*{lem:Ore_conditions}.\arabic*)},
    leftmargin=*,labelindent=0em]
  \item if \(g_1,g_2\in\Cat\) satisfy \(\rg(g_1) = \rg(g_2)\), then
    there are \(h_1,h_2\in\Cat\) with \(\s(g_i)=\rg(h_i)\) for
    \(i=1,2\) and \(g_1\cdot h_1 = g_2\cdot h_2\);
  \item if \(g,h_1,h_2\in\Cat\) satisfy \(\s(g) = \rg(h_1) =
    \rg(h_2)\) and \(g\cdot h_1 = g\cdot h_2\), then there is \(k\in
    \Cat\) with \(\rg(k) = \s(h_1) = \s(h_2)\) and \(h_1\cdot k =
    h_2\cdot k\).\qed
  \end{enumerate}
\end{lemma}

We briefly say that~\(\Cat\) \emph{satisfies the right Ore
  conditions}.

\begin{lemma}
  \label{lem:Ore_gives_cofinal}
  Let~\(\Cat\) satisfy the right Ore conditions and let
  \(g\in\Cat\).  Then the functor
  \(g_*\colon \Cat[D]_{\s(g)} \to \Cat[D]_{\rg(g)}\),
  \(h\mapsto g\cdot h\), is final, so that
  \[
  \Omega_{\rg(g)} \cong
  \varprojlim_{\Cat[D]_{\s(g)}} {}(\Bisp_{g h}/\Gr_{\s(h)},\pi_{g h,k}).
  \]
  The canonical map
  \[
  \Bisp_g \times_{\s,\Gr_{\s(g)}^0,\rg} \Omega_{\s(g)}
  \to \varprojlim_{\Cat[D]_{\s(g)}} {}
  (\Bisp_{g h}/\Gr_{\s(h)},\pi_{g h,k}) \cong \Omega_{\rg(g)}
  \]
  descends to a homeomorphism \(\Bisp_g \Grcomp \Omega_{\s(g)}
  \congto \Omega_{\rg(g)}\).  These maps define an \(F\)\nb-action
  on~\(\Omega\).
\end{lemma}

\begin{proof}
  We prove finality implicitly by proving that the limits
  \[
    \Omega_{\rg(g)} \defeq \varprojlim_{\Cat[D]_{\rg(g)}} {}
    (\Bisp_h/\Gr_{\s(h)},\pi_{h,k}),\qquad
    \Omega' \defeq \varprojlim_{\Cat[D]_{\s(g)}} {}
    (\Bisp_{g h} / \Gr_{\s(h)},\pi_{g h,k})
  \]
  are isomorphic.  This consequence is all that we are going to use
  anyway.  The functor
  \(g_*\colon \Cat[D]_{\s(g)} \to \Cat[D]_{\rg(g)}\) induces a
  canonical continuous map
  \[
  \Omega_{\rg(g)} \to \Omega',\qquad
  (\omega_k)_{k\in\Cat^{\rg(g)}}\mapsto
  (\omega_{g h})_{h\in\Cat^{\s(g)}}.
  \]
  We claim that it is a homeomorphism.

  Let \(\omega_{g h}\in \Bisp_{g h}/\Gr_{\s(g h)}\) be given for all
  \(h\in\Cat^{\s(g)}\).  Assume \(\pi_{g h,k}(\omega_{g h k}) =
  \omega_{g h}\) for all composable \(h,k\).  If \(\rg(k) =
  \rg(g)\), then there are \(h_1,h_2\in\Cat\) with \(k h_1 = g h_2\)
  by the first Ore condition.  Pick \(h_1,h_2\) as above for
  each~\(k\) and define \(\omega_k \defeq \pi_{k,h_1}(\omega_{g
    h_2}) \in \Bisp_k/\Gr_{\s(k)}\).  If \((\omega_k) \in
  \Omega_{\rg(g)}\), then \(\omega_k = \pi_{k,h_1}(\omega_{k
    h_1})\), so this is the only possible preimage of \((\omega_{g
    h})\in\Omega'\).  We will prove below that \((\omega_k)_{k\in
    \Cat^{\rg(g)}}\) belongs to~\(\Omega_{\rg(g)}\), that is,
  \(\pi_{k, h}(\omega_{k h}) = \omega_k\) for all composable \(k,h\)
  with \(\rg(k) = \rg(g)\).  The maps sending~\((\omega_{g h})\)
  to~\(\omega_k\) are continuous for each \(k\in\Cat^{\rg(g)}\).
  Hence we have found a continuous two-sided inverse for the
  canonical continuous map \(\Omega_{\rg(g)} \to \Omega'\).  Thus
  \(\Omega_{\rg(g)} \cong\Omega'\).

  We still have to prove \(\pi_{k, h}(\omega_{k h}) = \omega_k\) for
  composable \(k,h\) with \(\rg(k) = \rg(g)\).  By definition,
  \(\omega_k = \pi_{k,h_1}(\omega_{g h_2})\) and \(\omega_{k h} =
  \pi_{k h,h_3}(\omega_{g h_4})\) for certain \(h_i\in\Cat\) with
  \(k h_1 = g h_2\) and \(k h h_3 = g h_4\).  Since \(\rg(g h_2) =
  \rg(g h_4)\), the first Ore condition allows us to choose
  \(h_5,h_6\in\Cat\) with \(g h_2 h_5 = g h_4 h_6\).  Then \(k h h_3
  h_6 = g h_4 h_6 = g h_2 h_5 = k h_1 h_5\).  The second Ore
  condition allows us to choose \(h_7\in\Cat\) with \(h h_3 h_6 h_7
  = h_1 h_5 h_7\).  Then
  \begin{multline*}
    \omega_k
    = \pi_{k, h_1}(\omega_{g h_2})
    = \pi_{k, h_1}\bigl(\pi_{g h_2, h_5 h_7}(\omega_{g h_2 h_5 h_7})\bigr)
    = \pi_{k, h_1 h_5 h_7}(\omega_{g h_2 h_5 h_7})
    \\= \pi_{k, h h_3 h_6 h_7}(\omega_{g h_4 h_6 h_7})
    = \pi_{k, h}\circ \pi_{k h, h_3}(\omega_{g h_4})
    = \pi_{k, h}(\omega_{k h}).
  \end{multline*}

  Since limits commute with fibre products, there is a canonical
  homeomorphism
  \[
  \Bisp_g \times_{\s,\Gr_{\s(g)}^0,\rg} \Omega_{\s(g)} \cong
  \varprojlim_{\Cat[D]_{\s(g)}} {}
  (\Bisp_g\times_{\s,\Gr_{\s(g)}^0,\rg} \Bisp_h/\Gr_{\s(h)},
  \id\times\pi_{h,k}).
  \]
  Taking the quotient by the inner conjugation action
  of~\(\Gr_{\s(g)}\) on each factor in the limit on the right gives
  a continuous map
  \begin{equation}
    \label{eq:limit_comparison_map}
    \Bisp_g \times_{\s,\Gr_{\s(g)}^0,\rg} \Omega_{\s(g)} \to
    \varprojlim_{\Cat[D]_{\s(g)}} {}
    (\Bisp_g\Grcomp \Bisp_h/\Gr_{\s(h)},
    \id\Grcomp\pi_{h,k}).
  \end{equation}
  Its codomain is homeomorphic to the space~\(\Omega'\), which we have
  already
  identified with \(\Omega_{\rg(g)}\).  We claim that the map above
  is the orbit space projection for the canonical
  \(\Gr_{\s(g)}\)\nb-action on \(\Bisp_g
  \times_{\s,\Gr_{\s(g)}^0,\rg} \Omega_{\s(g)}\)
  by \(\eta\cdot (\gamma,\omega) = (\gamma \eta^{-1},\eta\omega)\)
  for \(\gamma\in\Bisp_g\), \(\eta\in\Gr_{\s(g)}\),
  \(\omega\in\Omega_{\s(g)}\) with \(\s(\gamma) = \rg(\omega) =
  \s(\eta)\).

  An element of \(\varprojlim_{\Cat[D]^{\s(g)}} {}
  (\Bisp_g\Grcomp\Bisp_h/\Gr_{\s(h)})\) is a family of \(\omega_{g
    h}\in (\Bisp_g \Grcomp \Bisp_h)/\Gr_{\s(h)} \cong \Bisp_g
  \Grcomp (\Bisp_h/\Gr_{\s(h)})\) with \(\pi_{g h,k} (\omega_{g h
    k}) = \omega_{g h}\) for all \(h,k\) with composable~\(g,h,k\).  In
  particular, \(\pi_{g, h} (\omega_{g h}) = \omega_{g \cdot \s(g)}\)
  for all \(h\in\Cat^{\s(g)}\).  We may write \(\omega_{g h} =
  \omega_{1,h}\cdot \omega_{2,h}\) with \(\omega_{1,h}\in \Bisp_g\),
  \(\omega_{2,h}\in \Bisp_h/\Gr_{\s(h)}\).  Then \([\omega_{1,h}] =
  [\omega_{1,\s(g)}]\) in \(\Bisp_g/\Gr_{\s(g)}\) for all \(h\in
  \Cat^{\s(g)}\).  Since the right \(\Gr_{\s(g)}\)\nb-action
  on~\(\Bisp_g\) is basic, there are unique \(\eta_h\in\Gr_{\s(g)}\)
  with \(\omega_{1,h} = \omega_{1,\s(g)} \cdot \eta_h\).  Thus we
  may rewrite \(\omega_{g h} = \omega_{1,\s(g)}\cdot \omega'_{2,h}\)
  with \(\omega_{2,h}' = \eta_h\cdot \omega_{2,h}\in
  \Bisp_h/\Gr_{\s(h)}\).

  The equality of \(\omega_{1,\s(g)}\cdot \omega_{2,h}'\) and
  \(\pi_{g h,k}(\omega_{1,\s(g)}\cdot \omega_{2,h k}')\) implies
  \(\pi_{h,k}(\omega_{2,h k}') = \omega_{2,h}'\) for all composable
  \(g,h,k\).  Thus \((\omega_{2,h}')\in
  \Omega_{\s(g)}\) and we have lifted \((\omega_{g h}) \in \Omega'\)
  to an element of \(\Bisp_g \times_{\s,\rg} \Omega_{\s(g)}\).  Here
  the image of~\(\omega_{1,\s(g)}\) in~\(\Bisp_g/\Gr_{\s(g)}\) is
  unique and so the lifting of \((\omega_{g h}) \in \Omega'\) to
  \(\Bisp_g \times_{\s,\rg} \Omega_{\s(g)}\) is unique up to the
  \(\Gr_{\s(g)}\)\nb-action by inner conjugation.  That is, the
  map~\eqref{eq:limit_comparison_map} induces a continuous bijection
  from the orbit space \(\Bisp_g \Grcomp \Omega_{\s(g)}\)
  onto~\(\Omega'\).

  The projection \(\Bisp_g\prto \Bisp_g/\Gr_{\s(g)}\) is
  a local homeomorphism by \longref{Lemma}{lem:basic_orbit_lh}.  Hence the
  choice of the lifting~\(\omega_{1,\s(g)}\) above can be made
  continuously in a sufficiently small neighbourhood of the class
  of~\(\omega_{1,\s(g)}\) in~\(\Bisp_g/\Gr_{\s(g)}\).  The preimage
  of this neighbourhood in~\(\Omega'\) is open, and the construction
  above gives a continuous map from this open subset to \(\Bisp_g
  \times_{\s,\rg} \Omega_{\s(g)}\).  Hence the bijection \(\Bisp_g
  \Grcomp \Omega_{\s(g)}\cong\Omega'\) is a homeomorphism.
\end{proof}

\begin{theorem}
  \label{the:Ore_universal_action}
  Let~\(\Cat\) satisfy the right Ore conditions.  Then the
  \(F\)\nb-action on~\(\Omega\) constructed above is universal.
\end{theorem}

\begin{proof}
  A direct computation shows that the canonical map
  \(\varrho\colon Y\to\Omega\) for an \(F\)\nb-action on~\(Y\) is
  \(F\)\nb-equivariant when we equip~\(\Omega\) with the
  \(F\)\nb-action in \longref{Lemma}{lem:Ore_gives_cofinal}.  Let
  \(\varphi\colon Y\to\Omega\) be any \(F\)\nb-equivariant map.  We
  claim that \(\varphi=\varrho\).  The construction of the maps
  \(\varrho_g\colon Y_x\to \Bisp_g/\Gr_{\s(g)}\) applied to the
  \(F\)\nb-action on \(\Omega\) gives the coordinate projections
  \(\pi^\Omega_g\colon \Omega_x \to \Bisp_g/\Gr_{\s(g)}\).  Since
  this map is natural, \(\pi^\Omega_g\circ \varphi = \varrho_g\) for
  all \(g\in \Cat\).  This implies
  \(\varphi = (\varrho_g)_{g\in\Cat} = \varrho\) as asserted.
\end{proof}

\begin{corollary}
  The \(F\)\nb-action on~\(\Omega\) induces an action of the inverse
  semigroup~\(\IS(F)\), and its transformation groupoid
  \(\IS(F)\ltimes \Omega\) is a groupoid model for~\(F\)
\end{corollary}

\begin{proof}
  This follows from \longref{Theorem}{the:Ore_universal_action} and
  \longref{Proposition}{pro:groupoid_model_from_universal_F-action}.
\end{proof}

This description of the groupoid model is useful, but it is not the
only one.  Our next goal is another description of it that does not
use the inverse semigroup~\(\IS(F)\).  This construction
follows~\cite{Albandik:Thesis} and proceeds in two steps.  The first
step reduces to the case of tight diagrams, and the second step
describes an explicit groupoid model for a tight diagram over a
category with right Ore conditions.

\subsection{Reduction to the tight case}
\label{sec:Ore_tight_reduction}

The \(F\)\nb-action on~\(\Omega_x\) contains a canonical
\(\Gr_x\)\nb-action.  Let
\[
\Gr\Omega_x \defeq \Gr_x\ltimes \Omega_x,\qquad
\Bisp\Omega_g \defeq \Bisp_g \times_{\s,\Gr_x^0,\rg} \Omega_{\s(g)}
\]
for \(x\in\Cat^0\), \(g\in\Cat\).  By definition, \(\Gr\Omega_x\) is
an étale topological groupoid.  We are going to turn~\(\Bisp\Omega_g\)
into a groupoid correspondence
\(\Gr\Omega_{\rg(g)} \leftarrow \Gr\Omega_{\s(g)}\).  The right
anchor map is the coordinate projection \(\s\colon \Bisp\Omega_g
\to \Omega_{\s(g)}\), \((\gamma,\omega)\mapsto \omega\), and the
left anchor map is the multiplication map \(\rg\colon \Bisp\Omega_g
\to \Omega_{\rg(g)}\), \((\gamma,\omega)\mapsto \gamma\cdot
\omega\), constructed in \longref{Lemma}{lem:Ore_gives_cofinal}.  The left
and right actions of \(\Gr_{\rg(g)}\) and~\(\Gr_{\s(g)}\) are
\[
\gamma_1 \cdot (\gamma_2,\omega) \defeq
(\gamma_1\cdot \gamma_2,\omega),\qquad
(\gamma_2,\omega)\cdot \gamma_3 \defeq
(\gamma_2\cdot \gamma_3,\gamma_3^{-1}\cdot \omega)
\]
for \(\gamma_1\in\Gr_{\rg(g)}\), \(\gamma_2\in\Bisp_g\),
\(\gamma_3\in\Gr_{\s(g)}\), \(\omega\in\Omega_{\s(g)}\) with
\(\s(\gamma_1) = \rg(\gamma_2)\) and \(\s(\gamma_2) = \rg(\gamma_3)
= \rg(\omega)\).  These anchor maps and actions satisfy
\begin{alignat*}{2}
  \s(\gamma_1\cdot (\gamma_2,\omega))
  &= \s(\gamma_2,\omega),&\qquad
  \s((\gamma_2,\omega)\cdot \gamma_3)
  &= \gamma_3^{-1}\cdot \s(\gamma_2,\omega),\\
  \rg(\gamma_1\cdot (\gamma_2,\omega))
  &= \gamma_1\cdot \rg(\gamma_2,\omega),&\qquad
  \rg((\gamma_2,\omega)\cdot \gamma_3) &=
  \rg(\gamma_2,\omega),\\
  \gamma_1\cdot ((\gamma_2,\omega)\cdot\gamma_3) &=
  (\gamma_1\cdot (\gamma_2,\omega))\cdot\gamma_3.
\end{alignat*}
Hence we may combine them to commuting actions of
\(\Gr\Omega_{\rg(g)}\) and~\(\Gr\Omega_{\s(g)}\)
on~\(\Bisp\Omega_g\).  The right action of~\(\Gr_{\s(g)}\)
on~\(\Bisp_g\) is basic.  This is inherited by the right action
of~\(\Gr_{\s(g)}\) on~\(\Bisp\Omega_g\) because of the
\(\Gr_{\s(g)}\)\nb-equivariant map \(\Bisp\Omega_g \to \Bisp_g\),
\((\gamma,\omega)\mapsto \gamma\) (see
\cite{Meyer-Zhu:Groupoids}*{Proposition~7.4}).  Hence the right
action of the transformation groupoid~\(\Gr\Omega_{\s(g)}\)
on~\(\Bisp\Omega_g\) is basic by \cite{Meyer-Zhu:Groupoids}*{Lemma
  5.17}.  Thus \(\Bisp\Omega_g\) is a groupoid correspondence
\(\Gr\Omega_{\rg(g)}\leftarrow\Gr\Omega_{\s(g)}\).

\begin{lemma}
  \label{lem:compose_with_BispOmega}
  Let~\(Y\) carry a left \(\Gr\Omega_{\s(g)}\)-action.  Then
  \(\Bisp\Omega_g \Grcomp_{\Gr\Omega_{\s(g)}} Y \cong \Bisp_g
  \Grcomp_{\Gr_{\s(g)}} Y\).  In particular,
  \[
  \Bisp\Omega_g \Grcomp_{\Gr\Omega_{\s(g)}} \Bisp\Omega_h
  \cong \Bisp_g \Grcomp_{\Gr_{\s(g)}} \Bisp_h \Grcomp_{\Gr_{\s(h)}} \Omega_h
  \xrightarrow[\cong]{\mu_{g,h}} \Bisp_{g h} \Grcomp_{\Gr_{\s(h)}} \Omega_h
  = \Bisp\Omega_{g h}.
  \]
  These isomorphisms \(\Bisp\Omega_g \Grcomp_{\Gr\Omega_{\s(g)}}
  \Bisp\Omega_h \congto \Bisp\Omega_{g h}\) form a diagram of tight
  groupoid correspondences.
\end{lemma}

\begin{proof}
  The space \(\Bisp\Omega_g \Grcomp_{\Gr\Omega_{\s(g)}} Y\) is
  defined as the orbit space of the inner conjugation action
  of~\(\Gr\Omega_{\s(g)}\) on the fibre product
  \[
  \Bisp\Omega_g \times_{\s,\Omega_{\s(g)},\rg} Y
  \defeq \Bisp_g \times_{\s,\Gr_x^0,\rg} \Omega_{\s(g)}
  \times_{\id,\Omega_{\s(g)},\rg} Y
  \cong \Bisp_g \times_{\s,\Gr_x^0,\rg} Y.
  \]
  The orbit space for an action of~\(\Gr\Omega_{\s(g)}\) is the same
  as for the restriction to~\(\Gr_{\s(g)}\) by
  \cite{Meyer-Zhu:Groupoids}*{Lemma 5.18}.  Hence \(\Bisp\Omega_g
  \Grcomp_{\Gr\Omega_{\s(g)}} Y \cong \Bisp_g \Grcomp_{\Gr_{\s(g)}}
  Y\).  Similarly, the orbit space for the right action
  of~\(\Gr\Omega_{\s(g)}\) on~\(\Bisp\Omega_g\) is \(\Bisp_g \Grcomp
  \Omega_{\s(g)}\).  Since~\(\rg\) descends to a homeomorphism
  \(\Bisp_g \Grcomp \Omega_{\s(g)} \congto \Omega_{\rg(g)}\) by
  \longref{Lemma}{lem:Ore_gives_cofinal}, \(\Bisp\Omega_g\) is a tight
  groupoid correspondence.  The multiplication isomorphism
  \(\mu\Omega_{g,h}\colon \Bisp\Omega_g \Grcomp_{\Gr\Omega_{\s(g)}}
  \Bisp\Omega_h \congto \Bisp\Omega_{g h}\) is defined by
  \[
  (\gamma_1,\gamma_2\cdot \omega_2)\cdot (\gamma_2,\omega_2)
  = (\mu_{g,h}(\gamma_1,\gamma_2),\omega_2).
  \]
  These maps are associative because the maps~\(\mu_{g,h}\) are so.
  The conditions for unit arrows are trivial.  Thus
  \((\Gr\Omega_x,\Bisp\Omega_g,\mu\Omega_{g,h})\) is a diagram of
  tight groupoid correspondences.
\end{proof}

\begin{theorem}
  \label{the:Ore_tight_same_limit}
  Let~\(\Cat\) be a small category that satisfies the right Ore
  conditions and let \(F = (\Gr_x,\Bisp_g,\mu_{g,h})\) be a
  \(\Cat\)\nb-shaped diagram of groupoid correspondences.  A
  groupoid model for actions of this diagram is also a groupoid
  model for actions of the corresponding tight diagram \(F\Omega
  \defeq (\Gr\Omega_x,\Bisp\Omega_g,\mu\Omega_{g,h})\), and vice
  versa.
\end{theorem}

\begin{proof}
  We are going to construct a bijection between actions of the
  diagrams \(F\)
  and~\(F\Omega\)
  on~\(Y\)
  that preserves equivariant maps.  Hence both diagrams have
  isomorphic groupoid models by the uniqueness of groupoid models.

  First let~\(Y\) carry an \(F\)\nb-action.  This consists of
  \(\Gr_x\)\nb-actions on the spaces~\(Y_x\) and
  \(\Gr_{\rg(g)}\)\nb-equivariant homeomorphisms \(V_g\colon \Bisp_g
  \Grcomp_{\Gr_{\s(g)}} Y_{\s(g)}\congto Y_{\rg(g)}\) for all
  \(g\in\Cat\) with some relations between \(V_g\), \(V_h\)
  and~\(V_{g h}\) for composable arrows \(g,h\in\Cat\) and a
  condition for~\(V_x\) for unit arrows~\(x\).  Since the
  \(F\)\nb-action on~\(\Omega\) is universal by
  \longref{Theorem}{the:Ore_universal_action}, there is a unique
  \(F\)\nb-equivariant map \(\varrho\colon Y \to \Omega\).  Its
  restriction to~\(Y_x\) is a \(\Gr_x\)\nb-equivariant map
  \(\varrho^x\colon Y_x \to \Omega_x\).  The action of~\(\Gr_x\) and
  the map~\(\varrho^x\) are equivalent to an action of the
  transformation groupoid \(\Gr\Omega_x = \Gr_x\ltimes \Omega_x\)
  on~\(Y_x\).  So we get canonical actions of~\(\Gr\Omega_x\)
  if~\(Y\) carries an \(F\)\nb-action and, conversely, an
  \(F\Omega\)\nb-action provides actions of~\(\Gr_x\) on~\(Y_x\).

  \longref{Lemma}{lem:compose_with_BispOmega} provides
  canonical homeomorphisms \(\Bisp\Omega_g
  \Grcomp_{\Gr\Omega_{\s(g)}} Y_{\s(g)} \congto \Bisp_g
  \Grcomp_{\Gr_{\s(g)}} Y_{\s(g)}\).  Composing with~\(V_g\) gives
  homeomorphisms \(\Bisp\Omega_g \Grcomp_{\Gr\Omega_{\s(g)}}
  Y_{\s(g)} \congto Y_{\rg(g)}\).  A direct computation shows that
  these homeomorphisms form an action of~\(F\Omega\) if and only if
  the original maps~\(V_g\) form an action of~\(F\).
\end{proof}

\subsection{A groupoid model for tight diagrams of Ore shape}
\label{sec:tight_Ore_model}

Let~\(\Cat\) be a small category that satisfies the right Ore
conditions.  If \(F=(\Gr_x,\Bisp_g,\mu_{g,h})\) is a
\(\Cat\)\nb-shaped diagram of groupoid correspondences, then
\((\Gr\Omega_x,\Bisp\Omega_g,\mu\Omega_{g,h})\) is a tight
\(\Cat\)\nb-shaped diagram with the same groupoid model by
\longref{Theorem}{the:Ore_tight_same_limit}.  It remains to find a
groupoid model when the diagram \((\Gr_x,\Bisp_g,\mu_{g,h})\) is
already tight.

Let \(\Gr \defeq \bigsqcup_{x\in\Cat^0} \Gr_x\).  This is a groupoid
with object space \(\Gr^0 \defeq \bigsqcup_{x\in\Cat^0} \Gr_x^0\).
Since the diagram~\(F\) is tight, \(\Gr^0\) carries a universal
action of the diagram.  Let \(\vartheta\colon \IS(F)\to I(\Gr^0)\) denote
the canonical \(\IS(F)\)\nb-action generated by the \(F\)\nb-action.
The transformation groupoid \(\IS(F)\ltimes_\vartheta \Gr^0\) is a
groupoid model for \(F\)\nb-actions by
\longref{Theorem}{the:universal_action_tight}.  We are going to describe this
groupoid in a different way, without using the inverse
semigroup~\(\IS(F)\).

\begin{lemma}
  \label{lem:Ore_tight_simplify_groupoid}
  Let \(x\in\Cat^0\) and \(g,h\in\Cat\) with \(x=\rg(g)=\rg(h)\).
  Let \(\alpha\in\Bis(\Bisp_g)\), \(\beta\in\Bis(\Bisp_h)\) and let
  \(y\in\Gr_{\s(h)}^0\) belong to the domain of \(\vartheta(\alpha)^*
  \vartheta(\beta)\).  There are \(\tau,\sigma\in\Bis(F)\) such that
  \(\vartheta(\sigma)^*\) is defined in a neighbourhood of~\(y\) and
  \(\alpha^*\cdot \beta\cdot \sigma = \tau\) in~\(\IS(F)\).  Thus
  \((\alpha^*\cdot \beta,y)\) and \((\tau\cdot\sigma^*,y)\)
  represent the same arrow in \(\IS(F)\ltimes_\vartheta \Gr^0\).
\end{lemma}

\begin{proof}
  By the first Ore condition, there are \(k,l\in\Cat\) with
  \(\rg(k) = \s(g)\), \(\rg(l)=\s(h)\) and \(g k = h l\).  Since
  slices cover~\(\Bisp_l\) and
  \(\rg\colon \Bisp_l \to \Gr^0_{\rg(l)}\) is surjective, there is a
  slice \(\sigma\in\Bis(\Bisp_l)\) with
  \(y\in \rg(\sigma) = \rg(\vartheta(\sigma))\);
  shrinking~\(\sigma\), we may arrange that \(U\defeq \rg(\sigma)\)
  is contained in the domain of
  \(\vartheta(\alpha)^*\vartheta(\beta)\).  The subset~\(U\) is open
  because~\(\Bisp_l\) is a tight correspondence.  The relations
  defining~\(\IS(F)\) imply
  \(\alpha^* \cdot \beta \cdot \sigma = \alpha^* \cdot
  (\beta\sigma)\) with \(\beta\sigma \in \Bis(\Bisp_{h l})\).  Since
  \(\Bisp_{h l} = \Bisp_{g k}\), slices of the form
  \(\tau_1 \tau_2\) with \(\tau_1\in\Bis(\Bisp_g)\),
  \(\tau_2\in\Bis(\Bisp_k)\) cover~\(\Bisp_{g k}\).
  Shrinking~\(\sigma\) further if necessary, we may arrange that
  \(\beta \sigma = \tau_1 \tau_2\) is itself of this form.  Then
  \(\alpha^* \cdot \beta \cdot \sigma = \alpha^* \cdot \tau_1\cdot
  \tau_2 = \braket{\alpha}{\tau_1}\cdot \tau_2 = \tau\) with
  \(\tau\defeq\braket{\alpha}{\tau_1}\tau_2\) belonging to
  \(\Bis(\Bisp_k) \subseteq \Bis(F)\).  Then
  \((\alpha^*\cdot \beta,y)\) and \((\tau\cdot\sigma^*,y)\)
  represent the same arrow in \(\IS(F)\ltimes_\vartheta \Gr^0\)
  because~\(\sigma \cdot \sigma^*\) is an idempotent element
  of~\(\IS(F)\) defined at~\(y\) and
  \(\alpha^*\cdot \beta\cdot \sigma \cdot \sigma^* = \tau\cdot
  \sigma^*\).
\end{proof}

Arrows in \(\IS(F)\ltimes \Gr^0\) are equivalence classes of
pairs~\((w,x)\) with \(w\in \IS(F)\), \(x\in \Dom \vartheta(w)\), where
\((w_1,x_1)\sim (w_2,x_2)\) if and only if there is an idempotent
element \(e\in \IS(F)\) with \(w_1\cdot e = w_2\cdot e\) and
\(x_1=x_2\in \Dom \vartheta(e)\).  Arrows with fixed~\(w\) form a
slice in~\(\IS(F)\ltimes\Gr^0\), which we denote by~\(\Theta(w)\).
We distinguish the action~\(\vartheta\) on~\(\Gr^0\) from~\(\Theta\)
because the groupoid \(\IS(F)\ltimes\Gr^0\) need not be effective.

\begin{lemma}
  \label{lem:Ore_tight_simplify_groupoid2}
  Slices of the form \(\Theta(\tau)\Theta(\sigma)^*\) for
  \(\sigma,\tau\in\Bis(F)\) cover \(\IS(F)\ltimes \Gr^0\).
\end{lemma}

\begin{proof}
  The inverse semigroup~\(\IS(F)\) is generated by~\(\Bis(F)\).
  Appending unit slices of~\(\Gr_x\) for suitable~\(x\) if
  necessary, we may write any element~\(w\) of~\(\IS(F)\) as a word
  \(\tau_1\cdot \sigma^*_1\cdot \tau_2\cdot \sigma_2^* \dotsm
  \tau_\ell \cdot \sigma_\ell^*\) for certain \(\tau_i,\sigma_i\in
  \Bis(F)\).  Let \(y\in\Gr^0\) be in the domain of~\(\vartheta(w)\).
  Then \longref{Lemma}{lem:Ore_tight_simplify_groupoid} gives
  \(\xi,\eta\in\Bis(F)\) such that~\(\vartheta(\eta)^*\) is defined
  at~\(\vartheta(\sigma_\ell)^*y\) and \([\sigma_{\ell-1}^*\cdot
  \tau_\ell, \vartheta(\sigma_\ell)^*y] = [\xi\cdot
  \eta^*,\vartheta(\sigma_\ell)^*y]\) in \(\IS(F)\ltimes \Gr^0\).  Hence
  \[
  [\tau_1\cdot \sigma^*_1\cdot \tau_2\cdot \sigma_2^* \dotsm
  \tau_\ell \cdot \sigma_\ell^*,y]
  =
  [\tau_1\cdot \sigma^*_1\cdot \tau_2\cdot \sigma_2^* \dotsm
  \sigma_{\ell-2}^* \cdot (\tau_{\ell-1}\xi) \cdot (\sigma_\ell \eta)^*,y].
  \]
  This reduces the length of the word~\(w\).  Induction on the
  length of~\(w\) shows that any arrow is contained in
  \(\Theta(\tau)\Theta(\sigma)^*\) for some
  \(\sigma,\tau\in\Bis(F)\).
\end{proof}

Now we are going to derive a nice form for arrows
in~\(\IS(F)\ltimes\Gr^0\).  Let \(g,h\in\Cat\) with \(\s(g)=\s(h)\)
and let \((\gamma_g,\gamma_h) \in \Bisp_g \times_{\s,\Gr_{\s(g)},\s}
\Bisp_h\) be given.  There are slices \(\tau\in\Bis(\Bisp_g)\)
and \(\sigma\in\Bis(\Bisp_h)\) with \(\gamma_g\in\tau\),
\(\gamma_h\in\sigma\).  Then the partial homeomorphisms
\(\vartheta(\sigma^*)\) and~\(\vartheta(\tau)\) map \(\rg(\gamma_h)\mapsto
\s(\gamma_h) = \s(\gamma_g) \mapsto \rg(\gamma_g)\).  Thus
\[
\psi(\gamma_g,\gamma_h) \defeq (\tau\sigma^*,\rg(\gamma_h))
\]
is an arrow in the groupoid \(\IS(F)\ltimes \Gr^0\).

\begin{lemma}
  \label{lem:Ore_tight_simplify_groupoid3}
  The construction above defines a surjective local homeomorphism
  \[
  \psi\colon \bigsqcup_{\substack{g,h\in\Cat\\ \s(g)=\s(h)}} \Bisp_g
  \times_{\s,\Gr_{\s(g)},\s} \Bisp_h \to \IS(F)\ltimes\Gr^0.
  \]
\end{lemma}

\begin{proof}
  We first show that~\(\psi\) is well defined.
  Let \(\gamma_g\in \tau\cap \tau'\) and \(\gamma_h\in \sigma\cap
  \sigma'\).  Then there is an open neighbourhood~\(U\)
  of~\(\rg(\gamma_h)\) in~\(\Gr_{\rg(h)}^0\) such that
  \(\vartheta(\tau\cap \tau') \vartheta(\sigma\cap \sigma)^*\) is defined
  on~\(U\).  Then \(\tau \cdot \sigma^* \cdot U = (\tau\cap
  \tau')\cdot (\sigma\cap \sigma')^*\cdot U = \tau'\cdot
  (\sigma')^*\cdot U\) holds in~\(\IS(F)\), where we view \(U=U^*\) as
  an element of \(\Bis(F)\subseteq \IS(F)\).  Hence
  \([\tau\cdot\sigma^*,\rg(\gamma_h)] = [\tau'\cdot
  (\sigma')^*,\rg(\gamma_h)]\) as arrows of \(\IS(F)\ltimes\Gr^0\).
  That is, the map~\(\psi\) is well defined.

  \longref{Lemma}{lem:Ore_tight_simplify_groupoid2} shows that any arrow
  in~\(\IS(F)\ltimes\Gr^0\)
  is of the form \((\tau\cdot \sigma^*,x)\)
  for some \(\sigma,\tau\in\Bis(F)\)
  and~\(x\)
  in the domain of \(\vartheta(\tau) \vartheta(\sigma)^*\).
  There are \(g,h\in\Cat\)
  with \(\tau\in\Bis(\Bisp_g)\),
  \(\sigma\in\Bis(\Bisp_h)\).
  Then \(\Dom \vartheta(\tau)\)
  and \(\Dom \vartheta(\sigma)\)
  are contained in \(\Gr^0_{\s(g)}\)
  and~\(\Gr^0_{\s(h)}\),
  respectively.  We have \(\s(g)=\s(h)\)
  because \(\vartheta(\tau) \vartheta(\sigma)^* = \emptyset\)
  for \(\s(g)\neq\s(h)\).
  By the definition of the action~\(\vartheta\),
  there are unique \(\gamma_g\in\tau\subseteq \Bisp_g\)
  and \(\gamma_h\in\sigma\subseteq \Bisp_h\)
  with \(\rg(\gamma_h) = x\),
  \(\s(\gamma_g)=\s(\gamma_h)\),
  and \(\rg(\gamma_g) = \vartheta(\tau)\vartheta(\sigma)^* (x)\).
  Then \((\tau\cdot\sigma^*,x) = \psi(\gamma_g,\gamma_h)\).
  That is, \(\psi\) is surjective.

  Fix \(g,h\in\Cat\) and \(\tau\in\Bis(\Bisp_g)\),
  \(\sigma\in\Bis(\Bisp_h)\).  Then
  \(\tau\times_{\s,\s} \sigma\subseteq \Bisp_g\times_{\s,\s}
  \Bisp_h\) is an open subset.  Since we assume~\(\Bisp_h\) to be
  tight, the map \(\rg\colon \Bisp_h\to \Gr_{\rg(h)}^0\) is
  equivalent to the orbit space projection
  \(\Bisp_h \to \Bisp_h/\Gr_{\s(h)}\).  Hence~\(\rg\) restricts to
  an injective, open map on~\(\sigma\).  The source anchor map also
  restricts to an injective, open map on~\(\tau\).  Then the map
  \((\gamma_g,\gamma_h)\mapsto \rg(\gamma_h)\) restricts to an
  injective, continuous, open map on \(\tau\times_{\s,\s} \sigma\).
  The source map also restricts to an injective, continuous, open
  map on the slice \(\Theta(\tau\cdot\sigma^*)\) of
  \(\IS(F)\ltimes\Gr^0\).  The map~\(\psi\) maps
  \(\tau\times_{\s,\s}\sigma\) bijectively onto
  \(\Theta(\tau\cdot\sigma^*)\) and intertwines
  \((\gamma_g,\gamma_h)\mapsto \rg(\gamma_h)\) with the source map
  on \(\Theta(\tau\cdot\sigma^*)\).  Hence~\(\psi\) is a
  homeomorphism from \(\tau\times_{\s,\s}\sigma\) onto the open
  subset \(\Theta(\tau\cdot\sigma^*)\) of~\(\IS(F)\ltimes\Gr^0\).
  Thus~\(\psi\) is a local homeomorphism.
\end{proof}

It remains to analyse when \((\gamma_g,\gamma_h)\in \Bisp_g
\times_{\s,\s}\Bisp_h\) and \((\gamma_{g'},\gamma_{h'})\in
\Bisp_{g'} \times_{\s,\s}\Bisp_{h'}\) for \(g,h,g',h'\in\Cat\) with
\(\s(g)=\s(h)\) and \(\s(g')=\s(h')\) give the same arrow in the
groupoid~\(\IS(F)\ltimes\Gr^0\).  We first describe some obvious
sufficient conditions for this.  If \(k\in\Cat\) and
\(\eta\in\Bisp_k\), then
\((\eta,\eta)\in\Bisp_k\times_{\s,\s}\Bisp_k\) represents the unit
arrow on~\(\rg(\gamma)\): if \(\eta\in\alpha\) for some
\(\alpha\in\Bis(\Bisp_k)\), then \(\alpha\cdot \alpha^*\) is
idempotent in~\(\IS(F)\) and hence its elements represent unit arrows
in any transformation groupoid with~\(\IS(F)\).  Since the map
\(\Bis(F)\to \IS(F)\) is multiplicative, this implies
\[
\psi(\gamma_g,\gamma_h)
= \psi (\gamma_g \cdot \eta,\gamma_h\cdot \eta)
\]
for all \(g,h,k\in\Cat\)
with \(\s(g)=\s(h)=\rg(k)\)
and all \((\gamma_g,\gamma_h)\in \Bisp_g \times_{\s,\s}\Bisp_h\),
\(\eta\in\Bisp_k\)
with \(\s(\gamma_g) = \s(\gamma_h) = \rg(\eta)\).
We write \((\gamma_g,\gamma_h) \sim (\gamma_{g'},\gamma_{h'})\)
if
\((\gamma_g \cdot\eta,\gamma_h\cdot \eta) =
(\gamma_{g'}\cdot\eta',\gamma_{h'}\cdot\eta')\)
for some \(k,k'\in\Cat\),
\(\eta\in\Bisp_k\),
\(\eta\in\Bisp_{k'}\)
with \(\rg(k) = \s(g)\),
\(\rg(k') = \s(g')\),
\(\s(\gamma_g) = \rg(\eta)\),
and \(\s(\gamma_{g'}) = \rg(\eta')\).
The discussion above shows that
\((\gamma_g,\gamma_h) \sim (\gamma_{g'},\gamma_{h'})\)
implies \(\psi(\gamma_g,\gamma_h) = \psi(\gamma_{g'},\gamma_{h'})\).

\begin{lemma}
  \label{lem:Ore_tight_simplify_groupoid4}
  The relation~\(\sim\) is an equivalence relation.
\end{lemma}

\begin{proof}
  The relation is clearly symmetric and reflexive.  We must prove
  that it is transitive.  So let
  \(g,g',g'',h,h',h'',k,k',l,l'\in\Cat\) be such that
  \(\s(g)=\s(h)=\rg(k)\), \(\s(g')=\s(h')=\rg(k') = \rg(l)\),
  \(\s(g'')=\s(h'')=\rg(l')\) and \(g k = g' k'\), \(h k = h' k'\),
  \(g' l = g'' l'\), \(h' l = h'' l'\); let \(\gamma_g\in \Bisp_g\),
  \(\gamma_{g'}\in\Bisp_{g'}\), \(\gamma_{g''}\in\Bisp_{g''}\),
  \(\gamma_h\in \Bisp_h\), \(\gamma_{h'}\in\Bisp_{h'}\),
  \(\gamma_{h''}\in\Bisp_{h''}\), and let \(\eta\in\Bisp_k\),
  \(\eta'\in\Bisp_{k'}\), \(\xi\in\Bisp_l\), \(\xi'\in\Bisp_{l'}\)
  be such that \(\s(\gamma_g)=\rg(\eta)\),
  \(\s(\gamma_{g'})=\rg(\eta') = \rg(\xi)\),
  \(\s(\gamma_{g''})=\rg(\xi')\) and
  \[
  (\gamma_g \cdot\eta,\gamma_h\cdot \eta) =
  (\gamma_{g'}\cdot\eta',\gamma_{h'}\cdot\eta'),\qquad
  (\gamma_{g'} \cdot\xi,\gamma_{h'}\cdot \xi) =
  (\gamma_{g''}\cdot\xi',\gamma_{h''}\cdot\xi').
  \]
  This says that \((\gamma_g,\gamma_h) \sim
  (\gamma_{g'},\gamma_{h'})\) and \((\gamma_{g'},\gamma_{h'}) \sim
  (\gamma_{g''},\gamma_{h''})\).  The first Ore condition for
  \(k',l\) gives \(m,m'\in\Cat\) with \(\s(l)=\rg(m)\),
  \(\s(k')=\rg(m')\) and \(l m = k' m'\).  Since \(\rg\colon \Bisp_m
  \to \Gr^0_{\rg(m)}\) and \(\rg\colon \Bisp_{m'} \to
  \Gr^0_{\rg(m')}\) are surjective, there are \(\kappa\in \Bisp_m\),
  \(\kappa'\in \Bisp_{m'}\) with \(\rg(\kappa) = \s(\eta')\) and
  \(\rg(\kappa') = \s(\xi)\).  Then \(\eta' \cdot \kappa,
  \xi\cdot\kappa'\) are well defined elements of \(\Bisp_{k' m'} =
  \Bisp_{l m}\) with the same range \(\s(\gamma_{g'}) =
  \s(\gamma_{h'})\).  Since \(\rg\colon \Bisp_{k' m'} \to
  \Gr^0_{\rg(k')}\) is equivalent to the orbit space projection for
  the right action, there is \(\kappa''\in\Gr_{\s(m)}\) with \(\eta'
  \cdot \kappa= \xi\cdot\kappa'\kappa''\).  Replacing \(\kappa'\)
  by~\(\kappa' \cdot \kappa''\), we may arrange that already \(\eta'
  \cdot \kappa= \xi\cdot\kappa'\).  Then
  \[
  (\gamma_g \cdot\eta \kappa,\gamma_h\cdot \eta \kappa)
  = (\gamma_{g'}\cdot\eta' \kappa,\gamma_{h'}\cdot\eta' \kappa)
  = (\gamma_{g'}\cdot\xi \kappa',\gamma_{h'}\cdot\xi \kappa')
  = (\gamma_{g''}\cdot\xi' \kappa',\gamma_{h''}\cdot\xi' \kappa').
  \]
  Thus the relation~\(\sim\) is transitive.
\end{proof}

Let~\(\Gr[H]\) be the set of \(\sim\)\nb-equivalence classes in
\(\bigsqcup_{g,h\in\Cat: \s(g)=\s(h)} \Bisp_g \times_{\s,\s}
\Bisp_h\) with the quotient topology.  The discussion above shows
that~\(\psi\) descends to a surjective continuous map
\(\Gr[H] \to \IS(F)\ltimes\Gr^0\).

\begin{lemma}
  \label{lem:Ore_tight_simplify_groupoid5}
  The quotient map from \(\bigsqcup_{g,h\in\Cat: \s(g)=\s(h)}
  \Bisp_g \times_{\s,\s} \Bisp_h\) to~\(\Gr[H]\) and the map
  \(\Psi\colon \Gr[H] \to \IS(F)\ltimes\Gr^0\) induced
  by~\(\psi\) are local homeomorphisms.
\end{lemma}

\begin{proof}
  Let \(\tau\in\Bis(\Bisp_g)\), \(\sigma\in\Bis(\Bisp_h)\).  We have
  already seen in the proof of
  \longref{Lemma}{lem:Ore_tight_simplify_groupoid3} that~\(\psi\)
  restricts to a homeomorphism from
  \(\tau\times_{\s,\s}\sigma \subseteq \Bisp_g \times_{\s,\s}
  \Bisp_h\) onto the slice \(\Theta(\tau\sigma^*)\)
  in~\(\IS(F)\ltimes \Gr^0\).  The set of all pairs
  \((\gamma_1,\gamma_2)\) that are \(\sim\)\nb-equivalent to an
  element of \(\tau\times_{\s,\s}\sigma\) is open because the
  multiplication maps
  \(\Bisp_g \times_{\s,\rg}\Bisp_k \to \Bisp_{g k}\) are open.
  Equivalently, the image of \(\tau\times_{\s,\s}\sigma\)
  in~\(\Gr[H]\) is open in the quotient topology.  The quotient map
  to~\(\Gr[H]\) and~\(\Psi\) are continuous, and both are inverse to
  each other when restricted to the embedded image of
  \(\tau\times_{\s,\s}\sigma\).  Hence the quotient map restricts to
  a homeomorphism from \(\tau\times_{\s,\s}\sigma\) to its open
  image in~\(\Gr[H]\) and~\(\Psi\) restricts to a homeomorphism from
  that image onto the slice \(\Theta(\tau\sigma^*)\) in
  \(\IS(F)\ltimes\Gr^0\).  This implies the assertions because open
  subsets of the form \(\tau\times_{\s,\s}\sigma\) cover
  \(\Bisp_g \times_{\s,\s} \Bisp_h\).
\end{proof}

Let \(\Gr[H]^0\defeq \Gr^0\) and define \(\rg,\s\colon
\Gr[H]\rightrightarrows \Gr[H]^0\) by \(\rg[\gamma_g,\gamma_h]
\defeq \rg(\gamma_g)\), \(\s[\gamma_g,\gamma_h] \defeq
\rg(\gamma_h)\); this indeed descends to \(\sim\)\nb-equivalence
classes.  Equip~\(\Gr[H]\) with the partially defined
multiplication
\begin{equation}
  \label{eq:Ore_tight_multiplication}
  [\gamma_g,\gamma_h]\cdot [\gamma_h,\gamma_k] \defeq
  [\gamma_g,\gamma_k]
\end{equation}
for all \(g,h,k\in\Cat\) with \(\s(g)=\s(h)=\s(k)\),
\(\gamma_g\in\Bisp_g\), \(\gamma_h\in\Bisp_h\),
\(\gamma_k\in\Bisp_k\) with
\(\s(\gamma_g)=\s(\gamma_h)=\s(\gamma_k)\).

\begin{lemma}
  \label{lem:Ore_tight_simplify_groupoid6}
  The multiplication above is defined on
  \(\Gr[H]\times_{s,\Gr[H]^0,r}\Gr[H]\)
  and turns \(\rg,\s\colon \Gr[H]\rightrightarrows \Gr[H]^0\)
  into an étale topological groupoid.  The identity map on objects
  and~\(\Psi\)
  on arrows give a continuous functor
  \(\Gr[H] \to \IS(F)\ltimes \Gr^0\).
\end{lemma}

\begin{proof}
  In the situation of~\eqref{eq:Ore_tight_multiplication}, we have
  \(\s[\gamma_g,\gamma_h] = \rg(\gamma_h) =
  \rg[\gamma_h,\gamma_k]\).  Since \(\s,\rg\) on~\(\Gr[H]\) are
  well defined, the product \([\gamma_g,\gamma_h]\cdot
  [\gamma_{h'},\gamma_k]\) can only be defined if
  \(\s[\gamma_g,\gamma_h] = \rg(\gamma_h)\) and
  \(\s[\gamma_{h'},\gamma_k] = \rg(\gamma_{h'})\) are equal.
  Conversely, let \(\rg(\gamma_h) = \rg(\gamma_{h'})\) and hence
  \(\rg(h) = \rg(h')\).  Then the first Ore condition gives
  \(l,l'\in \Cat\) with \(\s(h)=\rg(l)\), \(\s(h') = \rg(l')\) and
  \(h\cdot l = h'\cdot l'\).  Since \(\rg\colon \Bisp_l \to
  \Gr_{\s(l)}^0\) is surjective, there are \(\eta\in \Bisp_l\) and
  \(\eta'\in \Bisp_{l'}\) with \(\rg(\eta) = \s(\gamma_h)\),
  \(\rg(\eta') = \s(\gamma_{h'})\).  Then \(\gamma_h\cdot \eta\) and
  \(\gamma_{h'} \cdot \eta'\) are well defined elements of
  \(\Bisp_{h l} = \Bisp_{h' l'}\) with the same range.  Therefore,
  there is \(\eta''\in\Gr_{\s(h l)}\) with \(\gamma_h \cdot \eta
  \eta'' = \gamma_{h'}\cdot\eta'\); replacing \(\eta\)
  by~\(\eta\eta''\), we arrange that already \(\gamma_h \cdot \eta =
  \gamma_{h'}\cdot\eta'\).  Then
  \[
  (\gamma_g,\gamma_h) \sim
  (\gamma_g\cdot\eta,\gamma_h\cdot\eta),\qquad
  (\gamma_{h'},\gamma_k)\sim
  (\gamma_{h'}\cdot\eta',\gamma_k\cdot\eta'),
  \]
  and so \([\gamma_g,\gamma_h]\cdot [\gamma_{h'},\gamma_k]\) is
  defined and equal to \([\gamma_g\cdot\eta,\gamma_k\cdot\eta']\).

  The multiplication on~\(\Gr[H]\) is associative because
  \[
  [\gamma_g,\gamma_h]\cdot ([\gamma_h,\gamma_k] \cdot [\gamma_k,\gamma_l])
  = [\gamma_g,\gamma_l] =
  ([\gamma_g,\gamma_h]\cdot [\gamma_h,\gamma_k]) \cdot [\gamma_k,\gamma_l].
  \]
  Arrows of the form \([\gamma,\gamma]\) are unit arrows, and
  \([\gamma_g,\gamma_h]^{-1} = [\gamma_h,\gamma_g]\).
  Thus~\(\Gr[H]\) is a groupoid.  Since
  \(\psi(\gamma_g,\gamma_h)\cdot \psi(\gamma_h,\gamma_k) =
  \psi(\gamma_g,\gamma_k)\), the map~\(\Psi\) on arrows and the
  identity on objects form a functor \(\Gr[H]\to \IS(F)\ltimes
  \Gr^0\).

  The map \(\Psi\colon \Gr[H] \to \IS(F)\ltimes\Gr^0\) is a
  local homeomorphism by \longref{Lemma}{lem:Ore_tight_simplify_groupoid5},
  and so are \(\s,\rg\colon \IS(F)\ltimes\Gr^0\rightrightarrows
  \Gr^0\) by construction.  Since \(\rg\circ\Psi = \rg\) and
  \(\s\circ\Psi = \s\), the maps \(\rg,\s\colon \Gr[H] \to
  \Gr[H]^0 = \Gr^0\) are local homeomorphisms as well.  The
  multiplication map on~\(\Gr[H]\) is continuous when restricted
  to \((\tau\times_{\s,\s}\sigma)\times_{\s_{\Gr[H]},\rg_{\Gr[H]}}
  (\sigma\times_{\s,\s}\varrho)\).  So the multiplication is
  continuous everywhere because such subsets cover \(\Gr[H]
  \times_{\s,\Gr[H]^0,\rg} \Gr[H]\).  The inversion map is continuous
  for a similar reason.  Thus~\(\Gr[H]\) is an étale topological groupoid.
\end{proof}

\begin{theorem}
  \label{the:Ore_tight_groupoid_model}
  Let~\(\Cat\)
  be a category that satisfies the right Ore conditions and let
  \(F=(\Gr_x,\Bisp_g,\mu_{g,h})\)
  be a \(\Cat\)\nb-shaped
  diagram of groupoid correspondences.  Let
  \(F\Omega = (\Gr\Omega_x,\Bisp\Omega_g,\mu\Omega_{g,h})\)
  be the corresponding tight diagram as in
  \longref{Theorem}{the:Ore_tight_same_limit}.  The
  groupoid~\(\Gr[H]\)
  constructed from~\(F\Omega\) is a groupoid model for~\(F\).
\end{theorem}

\begin{proof}
  We may assume that~\(F\) is already tight and \(F=F\Omega\):
  \longref{Theorem}{the:Ore_tight_same_limit} reduces the general case to
  this special case.  Then \(\IS(F)\ltimes \Gr^0\) is a groupoid model
  by \longref{Theorem}{the:Ore_universal_action}.  We have
  already constructed a functor \(\Psi\colon \Gr[H]\to \IS(F)\ltimes
  \Gr^0\), which is the identity on objects and a local
  homeomorphism on arrows by
  \longref{Lemma}{lem:Ore_tight_simplify_groupoid5}.  It remains to prove
  that~\(\Psi\) is injective.  Since we only understand~\(\IS(F)\)
  through its universal property, we prove this by constructing a
  homomorphism from~\(\IS(F)\) to the inverse semigroup of slices
  of~\(\Gr[H]\) and checking that the resulting functor
  \(\IS(F)\ltimes \Gr^0 \to \Gr[H]\) is a one-sided inverse
  to~\(\Psi\).

  Let \(g\in\Cat\) and \(\alpha\in\Bis(\Bisp_g)\).  If
  \(\gamma_g\in\alpha\), then \((\gamma_g,1_{\s(\gamma_g)}) \in \Bisp_g
  \times_{\s,\s} \Bisp_{\s(g)}\) because
  \(\Bisp_{\s(g)}=\Gr_{\s(g)}\).  These arrows form a slice
  in~\(\Gr[H]\), which we denote by~\(t_\alpha\).

  We claim that the map \(\Bis(F)\to \Bis(\Gr[H])\), \(\alpha\mapsto
  t_\alpha\), extends to a unital homomorphism \(\IS(F)\to
  \Bis(\Gr[H])\).  By the defining universal property of~\(\IS(F)\),
  it suffices to check \(t_\alpha t_\beta = t_{\alpha\beta}\) for
  all \(\alpha,\beta\in\Bis(F)\) and \(t_\alpha^* t_\beta =
  t_{\braket{\alpha}{\beta}}\) if \(\alpha,\beta\in\Bis(\Bisp_g)\)
  for the same \(g\in\Cat\).  For the first relation, let
  \(\alpha\in\Bis(\Bisp_g)\), \(\beta\in\Bis(\Bisp_h)\).  If
  \(\s(g)\neq\rg(h)\), then both \(t_\alpha t_\beta\)
  and~\(t_{\alpha\beta}\) are empty.  So assume \(\s(g)=\rg(h)\).
  If \(\gamma_g\in\alpha\), \(\gamma_h\in\beta\), then
  \([\gamma_g,1_{\s(\gamma_g)}]\cdot [\gamma_h,1_{\s(\gamma_h)}]\)
  is defined if and only if \(\rg(1_{\s(\gamma_g)}) = \s(\gamma_g)\)
  and \(\rg(\gamma_h)\) are equal, and then
  \[
  [\gamma_g,1_{\s(\gamma_g)}]\cdot [\gamma_h,1_{\s(\gamma_h)}]
  =
  [\gamma_g\cdot \gamma_h,1_{\s(\gamma_g)}\cdot \gamma_h]\cdot
  [\gamma_h,1_{\s(\gamma_h)}]
  =  [\gamma_g\cdot \gamma_h,1_{\s(\gamma_g \gamma_h)}].
  \]
  This implies \(t_\alpha t_\beta = t_{\alpha\beta}\).  Now let
  \(g\in\Cat\) and \(\alpha,\beta\in \Bis(\Bisp_g)\).  Then elements
  of \(t_\alpha^*\) are of the form \([1_{\s(\gamma_g)},\gamma_g]\)
  with \(\gamma_g\in\alpha\).  The product of
  \([1_{\s(\gamma_g)},\gamma_g]\) and
  \([\gamma_h,1_{\s(\gamma_h)}]\) for \(\gamma_g\in\alpha\),
  \(\gamma_h\in\beta\) is defined if and only if \(\rg(\gamma_g) =
  \rg(\gamma_h)\); then there is a unique \(\eta\in\Gr_{\s(g)}\)
  with \(\gamma_g \eta = \gamma_h\), namely, \(\eta=
  \braket{\gamma_g}{\gamma_h}\).  Then
  \begin{multline*}
    [1_{\s(\gamma_g)},\gamma_g]\cdot [\gamma_h,1_{\s(\gamma_h)}]
    = [1_{\s(\gamma_g)}\cdot \braket{\gamma_g}{\gamma_h},\gamma_g\cdot
    \braket{\gamma_g}{\gamma_h}]\cdot
    [\gamma_h,1_{\s(\gamma_h)}]
    \\= [\braket{\gamma_g}{\gamma_h},1_{\s(\gamma_h)}].
  \end{multline*}
  This implies \(t_\alpha^* t_\beta = t_{\braket{\alpha}{\beta}}\).

  The argument above gives a unique homomorphism \(t\colon
  \IS(F)\to\Bis(\Gr[H])\) mapping \(\alpha\in\Bis(F)\subseteq \IS(F)\)
  to~\(t_\alpha\).  We may write \(\Gr[H] \cong \Bis(\Gr[H])\ltimes
  \Gr^0\).  Since~\(t\) preserves the actions on~\(\Gr^0\), it
  induces a continuous functor \(t_*\colon \IS(F)\ltimes\Gr^0 \to
  \Bis(\Gr[H]) \ltimes \Gr^0 = \Gr[H]\).  If \(\alpha\in\Bis(F)\)
  and \(x\in\Dom \vartheta(\alpha) \subseteq \Gr^0\), then the arrow
  \([\alpha,x] \in \IS(F)\ltimes\Gr^0\) is mapped by~\(t_*\) to
  \([\gamma_g,1_{\s(\gamma_g)}]\) for the unique \(\gamma_g\in
  \alpha\) with \(\s(\gamma_g) = x\).  Hence
  \[
  t_*(\psi(\gamma_g,\gamma_h))
  = [\gamma_g,1_{\s(\gamma_g)}]\cdot [1_{\s(\gamma_h)},\gamma_h]
  = [\gamma_g,\gamma_h]
  \]
  for all \((\gamma_g,\gamma_h)\in \Bisp_g\times_{\s,\s}\Bisp_h\).
  Thus \(t_*\circ \Psi\) is the identity functor on~\(\Gr[H]\).
\end{proof}

\subsection{Grading of the groupoid model}
\label{sec:grade_groupoid_model}

We still assume the diagram \((\Gr_x,\Bisp_g,\mu_{g,h})\) to be
tight.  The groupoid~\(\Gr[H]\) constructed above has some extra
structure that is analogous to the structure of the
Cuntz--Pimsner algebra of a product system over an Ore monoid found
in \cite{Albandik-Meyer:Product}*{Theorem~3.16}.  Namely, it is
graded by the groupoid completion~\(\GroupoidCat\) of~\(\Cat\).  We
are going to describe this groupoid completion and then build a
canonical functor from~\(\Gr[H]\) to~\(\GroupoidCat\).  This is how
a groupoid is graded by a group.

First we simplify~\(\GroupoidCat\) using the Ore conditions,
generalising a well known construction for Ore monoids to categories
with right Ore conditions.  The groupoid completion has the same
objects: \(\GroupoidCat^0 = \Cat^0\).  An arrow \(x\to y\)
in~\(\GroupoidCat\) is an equivalence class of zigzags
\[
g h^{-1}\colon x\xleftarrow{h} z \xrightarrow{g} y,
\]
where \(g,h\in\Cat\) satisfy \(\s(h)=\s(g)\), \(\rg(g)=y\), and
\(\rg(h)=x\); two zigzags \(g_1 h_1^{-1}\) and~\(g_2 h_2^{-1}\) as
above are equivalent if there are \(k_1,k_2\in\Cat\) with \(\rg(k_i)
= \s(g_i) = \s(h_i)\) for \(i=1,2\), and \(g_1 k_1 = g_2 k_2\),
\(h_1 k_1 = h_2 k_2\).  This relation is motivated by the formal
computation
\[
g_1 h_1^{-1}
= g_1 (k_1 k_1^{-1}) h_1^{-1}
= g_1 k_1 (h_1 k_1)^{-1}
= g_2 k_2 (h_2 k_2)^{-1}
= g_2 (k_2 k_2^{-1}) h_2^{-1}
= g_2 h_2^{-1}.
\]
We want to compose two zigzags \(g_1 h_1^{-1}\colon y\to x\) and
\(g_2 h_2^{-1}\colon z\to y\).  The Ore conditions allow us to
choose \(k_1,k_2\in \Cat\) with \(\rg(k_1) = \s(h_1)\), \(\rg(k_2) =
\s(g_2)\), and \(h_1 k_1 = g_2 k_2\).  Then we let \(g_1 h_1^{-1}
\cdot g_2 h_2^{-1} \defeq g_1 k_1 (h_2 k_2)^{-1}\), as suggested by
the formal computation
\[
g_1 h_1^{-1} \cdot g_2 h_2^{-1}
= (g_1 k_1) (h_1 k_1)^{-1} \cdot (g_2 k_2) (h_2 k_2)^{-1}
= (g_1 k_1) (h_2 k_2)^{-1}.
\]

\begin{lemma}
  \label{lem:Ore_groupoid_completion}
  The relation on zigzags above is an equivalence relation, and the
  multiplication is well defined.  It defines a
  groupoid~\(\GroupoidCat\).  This is the groupoid completion
  of~\(\Cat\), that is, any functor from~\(\Cat\) to a groupoid
  factors uniquely through~\(\GroupoidCat\).
\end{lemma}

\begin{proof}
  The relation on zigzags above is clearly reflexive and symmetric.
  To prove that it is transitive, let
  \(g_1 h_1^{-1} \sim g_2 h_2^{-1}\) and
  \(g_2 h_2^{-1} \sim g_3 h_3^{-1}\).  Then there are
  \(k_1,k_2,l_1,l_2\in\Cat\) with \(g_1 k_1 = g_2 k_2\),
  \(h_1 k_1 = h_2 k_2\), \(g_2 l_1 = g_3 l_2\) and
  \(h_2 l_1 = h_3 l_2\).  The Ore conditions imply that there are
  \(m_1,m_2\in\Cat\) with \(h_2 k_2 m_1 = h_2 l_1 m_2\).  Then there
  is \(m_3\in\Cat\) with \(k_2 m_1 m_3 = l_1 m_2 m_3\).  Then
  \(g_1 (k_1 m_1 m_3) = g_2 k_2 m_1 m_3 = g_2 l_1 m_2 m_3 = g_3 (l_2
  m_2 m_3)\) and
  \(h_1 (k_1 m_1 m_3) = h_2 k_2 m_1 m_3 = h_2 l_1 m_2 m_3 = h_3 (l_2
  m_2 m_3)\).  Then \(g_1 h_1^{-1} \sim g_3 h_3^{-1}\).

  Next we check that \(g_1 h_1^{-1} \cdot g_2 h_2^{-1}\) does not
  depend on the choice of~\(k_1,k_2\) above.  Let
  \(k_1,k_2,k_1',k_2'\in\Cat\) satisfy \(h_1 k_1 = g_2 k_2\) and
  \(h_1 k_1' = g_2 k_2'\).  The Ore conditions give
  \(l_1,l_2\in\Cat\) with \(h_1 k_1 l_1 = h_1 k_1' l_2\); they
  also give \(l_3\in\Cat\) with \(k_1 l_1 l_3 = k_1' l_2 l_3\).
  Then \(g_2 k_2 l_1 l_3 = h_1 k_1 l_1 l_3 = h_1 k_1' l_2 l_3 = g_2
  k_2' l_2 l_3\), so the Ore conditions give \(l_4\in\Cat\) with
  \(k_2 l_1 l_3 l_4 = k_2' l_2 l_3 l_4\).  Hence
  \begin{multline*}
    (g_1 k_1) (h_2 k_2)^{-1}
    \sim (g_1 k_1 l_1 l_3 l_4) (h_2 k_2 l_1 l_3 l_4)^{-1}
    \\= (g_1 k_1' l_2 l_3 l_4) (h_2 k_2' l_2 l_3 l_4)^{-1}
    = (g_1 k_1') (h_2 k_2')^{-1}.
  \end{multline*}
  Thus the multiplication in~\(\GroupoidCat\)
  is well defined and associative.  The arrow~\(1_x 1_x^{-1}\)
  is the unit arrow on~\(x\),
  and~\(h g^{-1}\)
  is inverse to~\(g h^{-1}\).  Hence~\(\GroupoidCat\) is a groupoid.

  The map \(g\mapsto g 1_x^{-1}\) is a functor \(\natural\colon
  \Cat\to\GroupoidCat\).  Let~\(\Gr\) be a groupoid and \(F\colon
  \Cat\to\Gr\) a functor.  Define \(\hat{F}\colon \GroupoidCat\to\Gr\)
  by \(\hat{F}(g h^{-1}) \defeq F(g) F(h)^{-1}\).  This is a well
  defined functor with \(\hat{F}(g 1_x^{-1}) = F(g)\) for all
  \(g\in\Cat\), and it is the unique functor on~\(\GroupoidCat\) with
  \(\hat{F}\circ{\natural} = F\).
\end{proof}

Implicitly, the proof above shows that certain categories
constructed from~\(\Cat\) are filtered.  For an arrow \(\hat{g}
\in\GroupoidCat\), let
\[
R_{\hat{g}} \defeq
\setgiven{(g,h)\in \Cat\times \Cat}{\s(g) = \s(h),\ g h^{-1} = \hat{g} \text{ in }\GroupoidCat}
\]
be the set of zigzags that represent it.  For
\((g_1,h_1),(g_2,h_2)\in R_g\), let
\[
\Cat[D]_{\hat{g}}\bigl((g_1,h_1),(g_2,h_2)\bigr)
\defeq \setgiven{k\in \Cat}{\rg(k)= \s(g_1) = \s(h_1),\ g_1 k = g_2,\ h_1 k = h_2}.
\]
We multiply such arrows through the multiplication of~\(k\)
in~\(\Cat\).  This defines a category~\(\Cat[D]_{\hat{g}}\).  This
category is filtered.  This is implicit in the proof of
\longref{Lemma}{lem:Ore_groupoid_completion} and shown if~\(\Cat\) is a
monoid in \cite{Albandik-Meyer:Product}*{Lemma~3.14}.

\begin{proposition}
  \label{pro:Ore_tight_partial_groupoid_fibration}
  The map \(\Gr^0\to\GroupoidCat^0=\Cat^0\) that maps
  \(\Gr_x^0\subseteq \Gr^0\) to~\(x\) and the map
  \(\Gr[H]\to\GroupoidCat\), \([\gamma_g,\gamma_h] \to g h^{-1}\), for
  \(g,h\in\Cat\) with \(\s(g)=\s(h)\) and \((\gamma_g,\gamma_h)\in
  \Bisp_g\times_{\s,\s}\Bisp_h\) combine to a well defined,
  continuous functor \(\Gr[H]\to\GroupoidCat\).
\end{proposition}

\begin{proof}
  The equivalence relation defining~\(\Gr[H]\) is generated by
  \((\gamma_g,\gamma_h)\sim (\gamma_g \eta,\gamma_h \eta)\) for
  \(k\in\Cat\), \(\eta\in\Bisp_k\) with \(\rg(k)=\s(g)=\s(h)\),
  \(\rg(\eta)=\s(\gamma_g) = \s(\gamma_h)\).  Since \(g k (h k)^{-1}
  \sim g h^{-1}\) in~\(\GroupoidCat\), the map
  \(\Gr[H]\to\GroupoidCat\) is well defined.  It is a functor by
  direct checking, and it is continuous because it is constant on
  \(\Bisp_g\times_{\s,\s}\Bisp_h\) for all \(g,h\in\Cat\) with
  \(\s(g)=\s(h)\).
\end{proof}

For \(\hat{g}\in\GroupoidCat\),
let \(\Gr[H]_{\hat{g}}\)
be the preimage of~\(\hat{g}\);
this is the set of all \([\gamma_g,\gamma_h]\)
with \((\gamma_g,\gamma_h) \in \Bisp_g\times_{\s,\s}\Bisp_h\)
and \(g h^{-1} = \hat{g}\)
in~\(\GroupoidCat\).
These subspaces are clopen and satisfy
\(\Gr[H]_{\hat{g}} \cdot \Gr[H]_{\hat{h}} \subseteq
\Gr[H]_{\hat{g}\cdot \hat{h}}\)
and \(\Gr[H]_{\hat{g}}^{-1} = \Gr[H]_{\hat{g}^{-1}}\)
for all \(\hat{g},\hat{h}\in\GroupoidCat\).
In particular, \(\Gr[H]_{1_x}\)
for the unit arrow \(1_x\in\GroupoidCat\)
is an open subgroupoid with unit space~\(\Gr_x^0\).

The subspace~\(\Gr[H]_{\hat{g}}\) of~\(\Gr[H]\) is the union of the
images of \(\Bisp_g \times_{\s,\s}\Bisp_h\) for all \((g,h)\in
R_{\hat{g}}\).  The map from the image of \(\Bisp_g
\times_{\s,\s}\Bisp_h\) to~\(\Gr[H]\) factors through the orbit
space \(\Bisp_g \Grcomp \Bisp_h^* \defeq (\Bisp_g
\times_{\s,\s}\Bisp_h)/\Gr_{\s(g)}\) for the diagonal action of
\(\Gr_{\s(g)} = \Gr_{\s(h)}\).

Let \((g,h)\in R_{\hat{g}}\) and \(k\in\Cat\) satisfy \(\s(g)=\s(h)
= \rg(k)\).  Then for any \([\gamma_g,\gamma_h] \in \Bisp_g \Grcomp
\Bisp_h^*\), there is \(\eta\in\Bisp_k\) with \(\s(\gamma_g) =
\s(\gamma_h) = \rg(\eta)\) and hence \([\gamma_g,\gamma_h] =
[\gamma_g \eta,\gamma_h \eta]\) in~\(\Gr[H]\).  Thus the image of
\(\Bisp_{g k} \times_{\s,\s}\Bisp_{h k}\) in~\(\Gr[H]\) contains
the image of \(\Bisp_g \times_{\s,\s}\Bisp_h\).  The proofs above
show implicitly that the map
\begin{equation}
  \label{eq:alpha_ghk}
  \alpha_{g,h}^k\colon
  \Bisp_g \Grcomp \Bisp_h^* \to
  \Bisp_{g k} \Grcomp \Bisp_{h k}^*,\qquad
  [\gamma_g,\gamma_h] \mapsto [\gamma_g \eta,\gamma_h \eta],
\end{equation}
is a well defined local homeomorphism.  These maps form a diagram of
topological spaces of shape~\(\Cat[D]_{\hat{g}}\).  The colimit of
this diagram is the quotient of \(\bigsqcup_{(g,h)\in R_{\hat{g}}}
\Bisp_g \Grcomp \Bisp_h^*\) by the equivalence relation generated by
\([\gamma_g,\gamma_h] \sim \alpha_{g,h}^k[\gamma_g,\gamma_h]\) for
all \(g,h,k\) as above and all \([\gamma_g,\gamma_h] \in \Bisp_g
\Grcomp \Bisp_h^*\).  The canonical maps \(\Bisp_g \Grcomp \Bisp_h^*
\to \Gr[H]_{\hat{g}}\) induce a homeomorphism
\[
\varinjlim_{\Cat[D]_{\hat{g}}}
{}(\Bisp_g \Grcomp \Bisp_h^*,\alpha_{g,h}^k)
\to  \Gr[H]_{\hat{g}}
\]
by definition of the equivalence relation~\(\sim\)
defining~\(\Gr[H]\).  The category~\(\Cat[D]_{\hat{g}}\) is
filtered.  Thus the colimit operation over \(\Cat[D]_{\hat{g}}\)
above behaves well.

If~\(x\in\Cat^0\) is a unit arrow in~\(\GroupoidCat\), then the
objects of the form \((g,g)\) with \(g\in\Cat^x\) are cofinal
in~\(\Cat[D]_{x}\).  Hence the subspace~\(\Gr[H]_x\) is also the
colimit of the restricted diagram \((\Bisp_g\Grcomp \Bisp_g^*,
\alpha_{g,g}^k)\) of shape~\(\Cat[D]_x\).

\section{Examples of groupoid models for diagrams of Ore shape}
\label{sec:examples_Ore}

\subsection{Diagrams of topological correspondences of Ore shape}
\label{sec:Ore_topological_corr}

Let~\(\Cat\) be a category with right Ore conditions and let
\((\Gr_x,\Bisp_g,\mu_{g,h})\) be a \(\Cat\)\nb-shaped diagram of
groupoid correspondences where all the groupoids~\(\Gr_x\) are just
spaces, viewed as groupoids.  We are going to explain the groupoid
model constructed above in this special case.  If~\(\Cat\) is a
monoid, the following analysis shows that we get exactly the same
groupoid, and the same structural properties, as
in~\cite{Albandik-Meyer:Product}.

We first make the diagram above more explicit; this does not yet use
the Ore conditions.  The groupoid correspondences~\(\Bisp_g\) are
topological correspondences between the spaces \(\Gr_{\s(g)}\)
and~\(\Gr_{\rg(g)}\) by
\cite{Antunes-Ko-Meyer:Groupoid_correspondences}*{Example~4.1}.
That is, they are spaces with continuous maps
\(\Gr_{\rg(g)} \xleftarrow{\rg} \Bisp_g \xrightarrow{\s}
\Gr_{\s(g)}\), where~\(\s\) is a local homeomorphism.  The
maps~\(\mu_{g,h}\) are homeomorphisms
\(\Bisp_g \times_{\s,\rg} \Bisp_h \to \Bisp_{g h}\),
\((\gamma_g,\gamma_h)\mapsto \gamma_g\cdot \gamma_h\), for
composable \(g,h\in\Cat\), which satisfy
\(\rg(\gamma_g\cdot \gamma_h) = \rg(\gamma_g)\),
\(\s(\gamma_g\cdot \gamma_h) = \s(\gamma_h)\) for all
\(\gamma_g\in\Bisp_g\), \(\gamma_h\in\Bisp_h\) with
\(\s(\gamma_g) = \rg(\gamma_h)\) and
\(\gamma_g\cdot (\gamma_h\cdot \gamma_k) = (\gamma_g\cdot
\gamma_h)\cdot \gamma_k\) for composable \(g,h,k\in\Cat\) and
\(\gamma_g\in\Bisp_g\), \(\gamma_h\in\Bisp_h\),
\(\gamma_k\in\Bisp_k\) with \(\s(\gamma_g) = \rg(\gamma_h)\) and
\(\s(\gamma_h) = \rg(\gamma_k)\).  In addition, \(\Bisp_x = \Gr_x\)
for unit arrows~\(x\), and the multiplication maps
\(\mu_{\rg(g),g}\) and~\(\mu_{g,\s(g)}\) are dictated by the above
conditions to be \(\rg(\gamma_g)\cdot \gamma_g = \gamma_g\) and
\(\gamma_g\cdot \s(\gamma_g) = \gamma_g\).

The first step in our construction replaces the given diagram by a
tight diagram.  Since the groupoids involved are of the form
\(\Gr_x\ltimes \Omega_x\), they are still spaces viewed as
groupoids.  A tight groupoid correspondence
\(\Omega_{\rg(g)}\leftarrow\Omega_{\s(g)}\) is equivalent to a local
homeomorphism \(\Omega_{\rg(g)}\rightarrow\Omega_{\s(g)}\) by
\cite{Antunes-Ko-Meyer:Groupoid_correspondences}*{Example~4.1}, and
there are no non-trivial
\(2\)\nb-arrows any more in this case.  Thus the first step in our
construction replaces the given diagram of topological
correspondences by an action of the opposite category~\(\Cat^\op\)
on a certain family of spaces \((\Omega_x)_{x\in\Cat^0}\) by local
homeomorphisms.

We now examine the construction of these spaces~\(\Omega_x\) and the
local homeomorphisms between them.  Since each~\(\Gr_x\) is a space,
\(\Bisp_g/\Gr_{\s(g)} = \Bisp_g\).  Hence~\(\Omega_x\) is the
projective limit of the spaces~\(\Bisp_g\) for \(g\in\Cat\) with
\(\rg(g)=x\) with respect to the canonical maps \(\pi_{g,k}\colon
\Bisp_{g k} \congto \Bisp_g \times_{\s,\rg} \Bisp_k \to \Bisp_g\),
where the first map is the inverse of the multiplication
homeomorphism and the second map is the coordinate projection.  So
\(\pi_{g,k}(\gamma_g\cdot \gamma_k) = \gamma_g\) for
\(\gamma_g\in\Bisp_g\), \(\gamma_k\in\Bisp_k\) with \(\s(\gamma_g) =
\rg(\gamma_k)\).  That is, an element of~\(\Omega_x\) is a family
\((\gamma_g)_{g\in\Cat^x}\) with \(\pi_{g,k}(\gamma_{g k}) =
\gamma_g\) for all \(g,k\in\Cat\) with \(\rg(g)=x\) and
\(\s(g)=\rg(k)\).

Now we examine the local homeomorphism \(\Omega_{\rg(g)} \to
\Omega_{\s(g)}\) corresponding to the tight correspondence
\(\Bisp\Omega_g\colon \Omega_{\rg(g)} \leftarrow \Omega_{\s(g)}\).
By definition, an element of the space~\(\Bisp\Omega_g\) is a pair
\((\gamma_g,\omega)\) where \(\gamma_g\in\Bisp_g\) and \(\omega\in
\Omega_{\s(g)}\).  The tightness of the
correspondence~\(\Bisp\Omega_g\) says that the range map
\((\gamma_g,\omega)\mapsto \gamma_g\cdot \omega\) is a homeomorphism
\(\Bisp\Omega_g \congto \Omega_{\rg(g)}\).
An element \((\gamma_k)_{k\in\Cat^{\rg(g)}}\) of~\(\Omega_{\rg(g)}\)
is determined uniquely and continuously by the entries~\(\gamma_{g
  k}\) for all \(k\in\Cat\) with \(\s(g)=\rg(k)\) (see
\longref{Lemma}{lem:Ore_gives_cofinal}); the Ore conditions are crucial to
prove this.  The local homeomorphism \(\sigma_g\colon
\Omega_{\rg(g)} \to \Omega_{\s(g)}\) corresponding
to~\(\Bisp\Omega_g\) maps \(\gamma_g\cdot \omega \mapsto \omega\).
Explicitly, if \((\gamma_h)_{h\in\Cat^{\rg(g)}}\) is a point
in~\(\Omega_{\rg(g)}\), then the \(k\)\nb-entry of
\(\sigma_g\bigl((\gamma_h)\bigr)\) for \(k\in\Cat^{\s(g)}\) is the
unique element \(\gamma_k\in \Bisp_k\) for which \(\gamma_{g k} =
\gamma_g\cdot \gamma_k\) for some \(\gamma_g\in \Bisp_g\).

The second step of the construction of~\(\Gr[H]\) is based on these
local homeomorphisms \(\sigma_g\colon \Omega_{\rg(g)}\to
\Omega_{\s(g)}\) for \(g\in\Cat\), which form an action
of~\(\Cat^\op\) on the spaces~\(\Omega_x\) for \(x\in\Cat^0\) by
local homeomorphisms.  We have identified \(\Bisp\Omega_g =
\Omega_{\rg(g)}\) above.  Since each~\(\Gr\Omega_x\) is still just a
space viewed as a groupoid,
\[
\Bisp\Omega_g \Grcomp \Bisp\Omega_h^*
\cong \Omega_{\rg(g)} \times_{\sigma_g,\sigma_h} \Omega_{\rg(h)}
\]
for all \(g,h\in\Cat\) with \(\s(g)=\s(h)\).  After this
identification, the local homeomorphism
\[
\alpha_{g,h}^k\colon
\Omega_{\rg(g)} \times_{\sigma_g,\sigma_h} \Omega_{\rg(h)}
\to
\Omega_{\rg(g k)} \times_{\sigma_{g k},\sigma_{h k}} \Omega_{\rg(h k)}
\]
in~\eqref{eq:alpha_ghk} becomes simply the inclusion map, which
exists because \(\sigma_g(\omega_1) = \sigma_h(\omega_2)\) implies
\(\sigma_{g k}(\omega_1) = \sigma_k \circ \sigma_g(\omega_1) =
\sigma_k \circ \sigma_h(\omega_2) = \sigma_{h k}(\omega_2)\).  Thus
these maps are all injective, and the colimit of these spaces is
their union inside \(\Omega_{\rg(g)} \times \Omega_{\rg(h)}\).  Thus
\[
\Gr[H]_{\hat{g}} \cong
\setgiven{(\omega_1,\omega_2) \in \Omega_{\rg(g)} \times \Omega_{\rg(h)}}{\sigma_g(\omega_1) = \sigma_h(\omega_2) \text{ for some }(g,h)\in R_{\hat{g}}},
\]
with the topology where a subset is open if and only if its
intersection with \(\Omega_{\rg(g)} \times_{\sigma_g,\sigma_h}
\Omega_{\rg(h)}\) is open for all \((g,h)\in R_{\hat{g}}\).  The
pair \((\omega_1,\omega_2)\) is an arrow \(\omega_1\leftarrow
\omega_2\).  Thus an arrow in~\(\Gr[H]\) is determined uniquely by
an arrow \(\hat{g}\in\GroupoidCat\) and by its range and source.  The
multiplication in~\(\Gr[H]\) is dictated by this:
\[
(\omega_1,\hat{g},\omega_2)\cdot
(\omega_2,\hat{h},\omega_3) =
(\omega_1,\hat{g}\cdot\hat{h},\omega_3).
\]

For a unit arrow \(\hat{g}=x\) in~\(\Cat^0\), the pairs~\((g,g)\)
are cofinal in~\(\Cat[D]_{\hat{g}}\).  Thus~\(\Gr[H]_x\) is the
union of \(\Omega_{\rg(g)} \times_{\sigma_g,\sigma_g}
\Omega_{\rg(g)}\) over all \(g\in\Cat^x\).  Each \(\Omega_{\rg(g)}
\times_{\sigma_g,\sigma_g} \Omega_{\rg(g)}\) is a proper equivalence
relation.  Thus~\(\Gr[H]_x\) is an approximately finite equivalence
relation.

The same structure was obtained in~\cite{Albandik-Meyer:Product} in
the case where~\(\Cat\) is a right Ore monoid, that is, has only one
object.  This also gives many important examples of groupoid models
for diagrams of Ore shape.

\subsection{Groupoid models of self-similar groups}
\label{sec:groupoid_self-similar}

Now we examine our construction in the case where~\(\Cat\) is the
Ore monoid \((\N,+)\) and~\(\Gr\) is a group.  Under the
faithfulness assumption in the theory of self-similar groups, this
reproduces the groupoid associated to a self-similar group by
Nekrashevych~\cite{Nekrashevych:Cstar_selfsimilar}.

We have seen in \longref{Section}{sec:free_monoid_action} that a diagram of
shape~\((\N,+)\) is equivalent to a groupoid~\(\Gr\) with a groupoid
correspondence \(\Bisp\colon \Gr\leftarrow\Gr\).   Groupoid
correspondences from a group to itself were described in
\cite{Antunes-Ko-Meyer:Groupoid_correspondences}*{Example~4.2}
through generalised
self-similarities.  Namely, \(\Bisp = A\times\Gr\) for a discrete
\(\Gr\)\nb-set~\(A\), with left action by \(g\cdot (x,h) = (g\cdot
x, g|_x\cdot h)\) for a \(1\)\nb-cocycle \(\Gr\times A\to \Gr\),
\((g,x)\mapsto g|_x\).  We do not need any further assumptions
on~\(A\) to construct a groupoid model for the diagram
\((\Gr,\Bisp)\).  As we shall see, however, the object space of our
groupoid is not locally compact unless~\(A\) is finite.

First we have to extend the single correspondence \(A\times\Gr\) to
a diagram, describing the composite correspondences
\((A\times\Gr)^{\Grcomp n}\) for all \(n\in\N\).  This is the
identity correspondence~\(\Gr\) for \(n=0\), which we write as
\(A^0\times\Gr\), where~\(A^0\) has only one element~\(\emptyset\).  The
composition of two correspondences \(A\times\Gr\) and \(Y\times\Gr\)
as above is defined so that its underlying space is
\[
(A\times\Gr)\Grcomp (Y\times \Gr) \cong A\times Y\times\Gr,
\]
again with the discrete topology and the obvious right
\(\Gr\)\nb-action.
The isomorphism maps
\(\bigl((x,g),(y,h)\bigr)\mapsto (x,g\cdot y, g|_y\cdot h)\).
Hence the left \(\Gr\)\nb-action is
\[
g\cdot (x,y,h) = (g\cdot x,g|_x\cdot y,(g|_x)|_y\cdot h).
\]
This involves the cocycle \(\Gr\times A\times Y \to \Gr\),
\((g,x,y)\mapsto g|_{(x,y)} \defeq (g|_x)|_y\).  By induction, we
get
\[
(A\times\Gr)^{\Grcomp n} = A^n\times \Gr
\]
with the \(\Gr\)\nb-action and cocycle
\begin{align*}
  g\cdot (x_1,x_2,\dotsc,x_n) &\defeq
  (g\cdot x_1,g|_{x_1}\cdot x_2,\dotsc,
  g|_{x_1 \dotsc x_{n-1}}\cdot x_n),\\
  g|_{x_1 x_2 \dotsc x_n} &\defeq g|_{x_1}|_{x_2} \dotsc|_{x_n}.
\end{align*}
This iteration is a standard construction in the theory of
self-similar groups, which is also used
in~\cite{Nekrashevych:Cstar_selfsimilar} to construct a groupoid from
the self-similar group.  The multiplication maps
\(\mu_{n,m}\colon (A^n\times\Gr) \Grcomp (A^m\times\Gr) \to
A^{n+m}\times\Gr\) are
\begin{multline}
  \label{eq:multiply_self-similarity_iterate}
  (x_1,x_2,\dotsc,x_n,g) \cdot
  (x_{n+1},x_{n+2},\dotsc,x_{n+m},h) \defeq
  \\(x_1,x_2,\dotsc,x_n, g\cdot x_{n+1},
  g|_{x_{n+1}}\cdot x_{n+2},\dotsc,
  g|_{x_{n+1}\dotsc x_{n+m-1}}\cdot
  x_{n+m}, g|_{x_{n+1}\dotsc x_{n+m}}\cdot h).
\end{multline}

The first step of our construction replaces the original diagram
over~\((\N,+)\) by a tight diagram of the form \(\Gr\ltimes \Omega\).
The space~\(\Omega\) is the projective limit of the sequence of spaces
\((A^n\times\Gr)/\Gr \cong A^n\) for \(n\in\N\).  Here the structure
maps \(A^n \to
A^m\) for \(n\ge m\) simply project to the first~\(m\) letters of a
word.  Hence~\(\Omega\) is the space of all infinite words in the
alphabet~\(A\), which is denoted~\(A^\omega\)
in~\cite{Nekrashevych:Cstar_selfsimilar}.  This is the object space of
the groupoid constructed in~\cite{Nekrashevych:Cstar_selfsimilar}.  It
is a compact space if~\(A\) is finite, which is assumed
in~\cite{Nekrashevych:Cstar_selfsimilar}.  If~\(A\) is infinite,
then~\(A^\omega\) is not even locally compact.

The action of the diagram on~\(A^\omega\) is given by an action of
the group~\(\Gr\) and a homeomorphism \(A\times A^\omega \to
A^\omega\) because \((A\times \Gr)\Grcomp_{\Gr} Y \cong A\times
Y\).  Here the action of~\(\Gr\) on~\(A^\omega\) is
\[
g\cdot (x_1,x_2,x_3,\dotsc) =
(g\cdot x_1, g|_{x_1}\cdot x_2, g|_{x_1 x_2}\cdot x_3,\dotsc)
\]
and the homeomorphism \(A\times A^\omega \to A^\omega\) is the
obvious one, \((x,w)\mapsto x\,w\).

The inverse semigroup~\(\IS(F)\) introduced in
\longref{Section}{sec:encoding} may be described very
concretely in this case:

\begin{lemma}
  \label{lem:S(F)_self-similar_group}
  The inverse semigroup~\(\IS(F)\) is the inverse semigroup with zero
  and unit generated by~\(\Gr\sqcup A\) with the relations that
  \(\Gr\to \IS(F)\) is a homomorphism, \(g\cdot x = (g x)\cdot g|_x\)
  for all \(g\in\Gr\), \(x\in A\) and \(x^*\cdot y = \delta_{x,y}\)
  for all \(x,y\in A\); here \(\cdot\) means the multiplication
  in~\(\IS(F)\) and \(g x\) the action of~\(\Gr\) on~\(A\).
  And~\(\delta_{x,y}\) is the zero element if \(x\neq y\) and the
  unit element of~\(\Gr\) if \(x=y\).
\end{lemma}

\begin{proof}
  Since a group has only one object, a slice of a groupoid
  correspondence \(\Bisp\colon \Gr[H] \leftarrow \Gr\) between two
  groups is either empty or a singleton subset of~\(\Bisp\).  Thus
  \(\Bis(F) = \emptyset \sqcup \bigsqcup_{n\ge0} A^n\times\Gr\),
  where~\(\emptyset\) acts as a zero element and the other elements
  are multiplied by~\eqref{eq:multiply_self-similarity_iterate}.
  Thus~\(\Bis(F)\) is the free semigroup with a zero element
  generated by \(\Gr\sqcup A\) with the relations of~\(\Gr\) and the
  commutation relation \(g\cdot x = (g\,x)\cdot g|_x\) for all
  \(g\in \Gr\), \(x\in A\); the element \((x_1,\dotsc,x_n,g)\) has
  the canonical presentation \(x_1\cdot x_2\dotsm x_n \cdot g\).
  This semigroup is a monoid with the same unit element as~\(\Gr\).

  The relations in~\(\Gr\) imply \(g g^{-1} g = g\) and
  \(g^{-1} g g^{-1} = g^{-1}\) for all \(g\in\Gr\), so that
  \(g^* = g^{-1}\) holds whenever we map~\(\Gr\) into an inverse
  semigroup.  Thus~\(\IS(F)\) is generated as a monoid by
  \(\Gr\sqcup A\sqcup A^*\sqcup \{\emptyset\}\).  Besides the
  relations in~\(\Gr\) and the relation
  \(g\cdot x = (g x)\cdot g|_x\) from~\(\Bis(F)\), in \(\IS(F)\), we
  also impose the relation
  \((w_1,g_1)^* \cdot (w_2,g_2) = \delta_{w_1,w_2} g_1^{-1}\cdot
  g_2\) for all \(g_1,g_2\in \Gr\), \(n\in\N\), \(w_1\in A^n\),
  \(w_2\in A^n\).  This relation holds for all \(g_1,g_2\) once it
  holds for \(g_1=g_2=1\), and an induction on~\(n\) shows that the
  case \(n=1\) suffices, that is, the relations above follow once
  \(x^* y = \delta_{x,y}\) for all \(x,y\in A\).
\end{proof}

Since we defined~\(\IS(F)\) as the free unital inverse semigroup
generated by some generators and relations, its unit element is
added formally and hence different from the unit element in~\(\Gr\),
which is a unit element also in~\(\Bis(F)\).  But we only consider
actions of~\(\IS(F)\) where \(1\in\Gr\) acts already by the identity
map on the whole space.  Thus it is possible to replace~\(\IS(F)\)
by the quotient by the relation that identifies the new formal unit
element with the unit element in~\(\Gr\).  In general, we add a
formal unit element when constructing~\(\IS(F)\) because~\(\Bis(F)\)
only has a unit when the shape category~\(\Cat\) has only one
object.

By \longref{Lemma}{lem:S(F)_self-similar_group}, an element of~\(\IS(F)\)
is either the zero element or the formal unit element, or it has the
form \(x_1 \dotsm x_m \cdot g \cdot y_n^*\dotsm y_1^*\)
with \(x_1,\dotsc,x_m,y_1,\dotsc,y_n\in A\),
\(g\in \Gr\),
\(n,m\in\N\).
When we identify the formal and the usual unit element, this gives the
inverse semigroup denoted~\(\langle \Gr,A\rangle\)
in \cite{Nekrashevych:Cstar_selfsimilar}*{Section 5.2}.  The action
of~\(\IS(F)\)
on~\(A^\omega\)
descends to an action of this quotient~\(\langle \Gr,A\rangle\),
and this action is the same one that is used
in~\cite{Nekrashevych:Cstar_selfsimilar}.  Hence
\(\langle \Gr,A\rangle\ltimes A^\omega\)
is a groupoid model for the diagram
\(A\times\Gr\colon \Gr\leftarrow\Gr\)
by \longref{Theorem}{the:Ore_universal_action} and
\longref{Proposition}{pro:groupoid_model_from_universal_F-action}.

Nekrashevych defines the groupoid associated to a self-similar
group as the ``groupoid of germs'' for the action of \(\langle
\Gr,A\rangle\) on~\(A^\omega\) in
\cite{Nekrashevych:Cstar_selfsimilar}*{Section 5.2}.  But he uses a
potentially coarser germ relation than in
\cite{Exel:Inverse_combinatorial}*{Proposition 5.4}: the groupoid
constructed in~\cite{Nekrashevych:Cstar_selfsimilar} is
\emph{effective} by construction, that is, if a slice of the
groupoid in~\cite{Nekrashevych:Cstar_selfsimilar} acts trivially
on~\(A^\omega\), then it consists of unit arrows only.  Thus the
construction in~\cite{Nekrashevych:Cstar_selfsimilar} gives
the effective quotient of our transformation groupoid \(\IS(F)\ltimes
A^\omega\), that is, the quotient by the interior of the isotropy
group bundle; this is equal to the transformation groupoid if and
only if the latter is effective, compare
\longref{Proposition}{pro:tight_effective_quotient}.

\begin{proposition}
  \label{pro:self-similar_groupoid_effective}
  Let~\(A\) be finite.  The groupoid \(\IS(F)\ltimes A^\omega =
  \langle \Gr,A\rangle\ltimes A^\omega\) is effective if and only if
  for each \(g\in \Gr\) that acts trivially on~\(A^\omega\), there
  is \(k\ge0\) so that \(g|_x=1\) for all \(x\in A^k\).
\end{proposition}

The hypothesis trivially holds if the action of~\(\Gr\) on the space
of finite words is effective, which is assumed in the theory of
self-similar groups in~\cite{Nekrashevych:Cstar_selfsimilar}.  So
the two germ relations agree in all cases considered
in~\cite{Nekrashevych:Cstar_selfsimilar}.

\begin{proof}
  Let \(w_1\,g\,w_2^*\in \IS(F)\) have trivial germ at some
  \(z\in A^\omega\).  That is, there is some open neighbourhood
  of~\(z\) on which~\(\vartheta(w_1\,g\,w_2^*)\) acts trivially.
  This open neighbourhood may be chosen of the form
  \(w_2\,x\,A^\omega\) for some \(k\ge 0\) and \(x\in A^k\) with
  \(z=w_2\,x\,z'\).  The restriction of \(\vartheta(w_1\,g\,w_2^*)\)
  to this neighbourhood is equal to
  \(\vartheta(w_1\,(g\cdot x)\,g|_x\,(w_2\,x)^*)\).  We may assume
  without loss of generality that~\(x\) is the empty word because
  otherwise we may replace \(w_1\,g\,w_2^*\) by
  \(w_1\,(g\cdot x)\,g|_x\,(w_2\,x)^*\).  Then
  \(\vartheta(w_1\,g\,w_2^*)\) acts by the identity map
  on~\(w_2\,A^\omega\).

  This means that \(w_2\,z = w_1\,(g\cdot z)\)
  for all \(z\in A^\omega\).
  This implies \(w_1=w_2\).
  Hence we may write \(w=w_1=w_2\).
  Then \(w\,z=w\,(g\cdot z)\)
  for all \(z\in A^\omega\)
  if and only if \(g\cdot z=z\)
  for all \(z\in A^\omega\).
  The arrow \([w\,g\,w^*,w z]\)
  in \(\IS(F)\ltimes A^\omega\)
  is a unit arrow if and only if there are \(k\ge0\)
  and \(x\in A^k\)
  with \(z=x\,z'\)
  and \(g|_x=1\).
  Then \([w\,g\,w^*,w z] = [w\,(g\cdot x)\,g|_x\,(w\,x)^*,w z]\),
  and \(g\cdot x = x\)
  follows from \(g\cdot z=z\).
  Conversely, the equivalence relation on triples that gives
  \(\IS(F)\ltimes A^\omega\)
  is generated by
  \([w\,g\,w^*,w z] = [w\,(g\cdot x)\,g|_x\,(w\,x)^*,w z]\)
  for \(z= x\,z'\)
  as above.  This follows from the structure of idempotents
  in~\(\IS(F)\)
  or from the alternative description of the groupoid model in the
  discussion below, which follows \longref{Section}{sec:tight_Ore_model}.
  Thus~\(\langle \Gr,A\rangle\ltimes A^\omega\)
  is effective if and only if for all \(g\in \Gr\)
  with \(g\cdot z = z\)
  for all \(z\in A^\omega\),
  the open subsets of the form~\(x\,A^\omega\)
  for all \(x\in A^k\)
  with \(g|_x=1\)
  cover~\(A^\omega\).
  Since we assume~\(A\)
  to be finite, the space~\(A^\omega\)
  is compact.  Hence the covering of~\(A^\omega\)
  by subsets of the form~\(x\,A^\omega\)
  with \(g|_x=1\)
  must admit a finite subcovering if it exists at all.  Let~\(k\)
  be the maximum of the lengths of the words~\(x\)
  used in this finite subcovering.  Taking all extensions of
  length~\(k\)
  for words that are shorter than~\(k\),
  we may arrange that~\(A^\omega\)
  is covered by subsets of the form \(x\,A^\omega\)
  with \(x\in A^k\)
  and \(g|_x=1\).  This means that \(g|_x=1\) for all \(x\in A^k\).
\end{proof}

Now we specialise the description of the groupoid model in
\longref{Theorem}{the:Ore_tight_groupoid_model} to the case at hand.
We prove that it yields the same groupoid model, as expected.

First we describe the tight diagram associated to the original
diagram.  This involves the groupoid \(\Gr\ltimes A^\omega\).  The
tight correspondence \(\Bisp^{\circ n} A^\omega\) over this groupoid
is the product space \(A^n\times\Gr\times A^\omega\).  By
definition, the left and right actions of \(\Gr\ltimes A^\omega\)
on~\(\Bisp^{\circ n} A^\omega\) have the anchor maps
\(\s(w,g,z) = z\) and \(\rg(w,g,z) = w\,(g\cdot z)\) for each
\(w\in A^n\), \(g\in\Gr\), \(z\in A^\omega\), that is, we
concatenate the finite word~\(w\) and the infinite
word~\(g\cdot z\).  The left and right multiplication maps are
determined by the following actions of~\(\Gr\):
\(g\cdot (w,h,z) \defeq (g\cdot w,g|_w\cdot h,z)\),
\((w,h,z)\cdot g \defeq (w,h\cdot g, g^{-1}\cdot z)\).  A direct
computation shows that this is tight, that is, the map~\(\rg\)
induces a homeomorphism
\((A^n\times\Gr\times A^\omega)/\Gr \congto A^\omega\).

Arrows in the groupoid~\(\Gr[H]\) are represented by
\[
\bigl((w_1,g_1,z_1), (w_2,g_2,z_2)\bigr) \in
\Bisp^{\circ n} A^\omega \times_{\s,\s} \Bisp^{\circ m} A^\omega
= A^n\times\Gr\times A^\omega\times_{\s,\s} A^m\times\Gr\times A^\omega,
\]
where the fibre product imposes the condition
\[
z_1 = \s(w_1,g_1,z_1) = \s(w_2,g_2,z_2) = z_2.
\]
We write \(z\) for \(z_1=z_2\).  The data above
gives an arrow from \(\rg(w_2,g_2,z) = w_2\,(g_2\cdot z)\) to
\(\rg(w_1,g_1,z) = w_1\,(g_1\cdot z)\).  The group~\(\Gr\)
acts on the above fibre product on the right by
\[
\bigl((w_1,g_1,z), (w_2,g_2,z)\bigr) \cdot h
\defeq
\bigl((w_1,g_1 \cdot h,h^{-1} z),
(w_2,g_2\cdot h,h^{-1} z)\bigr),
\]
and the arrow in~\(\Gr[H]\) depends only on the orbit under this
action.  Thus we may normalise the quintuples above by the condition
\(g_2=1\), say, and write them as
\[
(w_1,g,z,w_2) \defeq \bigl((w_1,g,z),(w_2,1,z)\bigr).
\]
The equivalence relation that describes the arrow space
of~\(\Gr[H]\) also involves the canonical maps~\(\alpha_{n,m}^k\)
in~\eqref{eq:alpha_ghk}, which map the orbit space
\((\Bisp^{\circ n} A^\omega \times_{\s,\s} \Bisp^{\circ m}
A^\omega)/\Gr\) to
\((\Bisp^{\circ n+k} A^\omega \times_{\s,\s} \Bisp^{\circ m+k}
A^\omega)/\Gr\).  To describe this map, write \(z\in A^\omega\)
as~\(x\,z'\) with \(x\in A^k\) and \(z'\in A^\omega\).  When we use
the simplified description of
\((\Bisp^{\circ n} A^\omega \times_{\s,\s} \Bisp^{\circ m}
A^\omega)/\Gr\), the map~\(\alpha_{n,m}^k\) becomes
\[
\alpha_{n,m}^k(w_1,g,z,w_2)
= (w_1\, (g\cdot x),g|_x,z', w_2\,x);
\]
this satisfies
\begin{align*}
  \s\bigl((w_1\, (g\cdot x),g,z, w_2\,x)\bigr)
  &= w_2\,x\,z = \s\bigl((w_1,g,xz,w_2)\bigr),\\
  \rg\bigl((w_1\, (g\cdot x),g|_x,z, w_2\,x)\bigr)
  &= w_1\,(g\cdot x)\,(g|_x\cdot z)
  = w_1\,(g\cdot (x\,z))
  = \rg\bigl((w_1,g,xz,w_2)\bigr),
\end{align*}
as it should be.  By construction,
\([w_1,g,z,w_2] = [w_1',g',z',w_2']\)
if and only if there are \(x\in A^k\),
\(x'\in A^{k'}\)
and \(\tilde{z}\in A^\omega\)
with
\[
z= x \tilde{z},\quad
z'= x' \tilde{z},\quad
w_1\, (g\cdot x) = w'_1\, (g'\cdot x'),\quad
g|_x = g'|_{x'},\quad
w_2\,x = w'_2\,x'.
\]
The groupoid~\(\Gr[H]\)
has some rather obvious slices that form a basis for the topology
on~\(\Gr[H]\).
Let \(\alpha_{w_1,g,w_2}\subseteq \Gr[H]\)
be the subset of all \([w_1,g,z,w_2]\)
with \(z\in A^\omega\)
for fixed \(w_1\in A^n\),
\(w_2\in A^m\),
\(g\in\Gr\)
with certain \(n,m\in\N\).
The subset~\(\alpha_{w_1,g,w_2}\)
is a slice because \(\rg[w_1,g,z,w_2] = w_1\,(g\cdot z)\)
and \(\s[w_1,g,z,w_2] = w_2\,z\)
determine~\(z\)
uniquely for fixed \(w_1,g,w_2\).
And~\(\alpha_{w_1,g,w_2}\)
has the codomain and domain \(w_1\, A^\omega\)
and~\(w_2\, A^\omega\),
respectively, that is, the sets of words starting with \(w_1\)
and~\(w_2\);
it acts on~\(A^\omega\)
by the partial homeomorphism \(w_2\,z\mapsto w_1\,(g\cdot z)\).
If \(x\in A^k\)
is another finite word, then the equivalence relation above implies
\[
\alpha_{w_1,g,w_2}|_{w_1\,x\, A^\omega}
= \alpha_{w_1\,(g\cdot x), g|_x, w_2\,x}.
\]
Thus
\[
\alpha_{w_1,g,w_2} = \bigsqcup_{x\in A^k}
\alpha_{w_1\,(g\cdot x), g|_x, w_2\,x}
\]
for each \(k\in\N\).  The description of the equivalence relation on
\([w_1,g,z,w_2]\) above shows that this implies all intersection
relations among the slices~\(\alpha_{w_1,g,w_2}\).

The slices above satisfy
\(\alpha_{w_1,g,w_2}^* = \alpha_{w_2,g^{-1},w_1}\) and
\(\alpha_{w_1,g,w_2}\cdot \alpha_{w_2,h,w_3} = \alpha_{w_1,g\cdot
  h,w_3}\) and
\(\alpha_{w_1,g,w_2} \alpha_{w_2',h,w_3} = \emptyset\) if neither
\(w_2\subseteq w_2'\) nor \(w_2'\subseteq w_2\).  Thus together
with~\(\emptyset\) they form an inverse semigroup.  This is exactly
the inverse semigroup denoted~\(\langle \Gr,A\rangle\) in
\cite{Nekrashevych:Cstar_selfsimilar}*{Section 5.2}.  To
get~\(\IS(F)\), we merely have to adjoin an extra unit element.

\begin{proposition}
  \label{pro:self-similar_groupoid_wide_subsemigroup}
  The inverse semigroup of slices \(\langle \Gr,A\rangle\)
  of~\(\Gr[H]\) is \emph{wide}, that is,
  \(\bigcup_{\tau\in\langle \Gr,A\rangle} \alpha_\tau = \Gr[H]\) and
  \(\alpha_{\tau_1} \cap \alpha_{\tau_2} = \bigcup_{\sigma \le
    \tau_1,\tau_2} \alpha_\sigma\) for all
  \(\tau_1,\tau_2\in \langle \Gr,A\rangle\).  Hence
  \[
  \Gr[H] = \langle \Gr,A\rangle \ltimes A^\omega.
  \]
\end{proposition}

\begin{proof}
  Any arrow in~\(\Gr[H]\)
  is represented by some \([w_1,g,z,w_2]\in \alpha_{w_1,g,w_2}\).
  That is,
  \(\bigcup_{\tau\in\langle \Gr,A\rangle} \alpha_\tau = \Gr[H]\).
  Let \(\tau_1 = (w_1,g_1,w_2)\)
  and \(\tau_2 = (w_3,g_2,w_4)\)
  with \(w_i \in A^{n_i}\)
  for \(1\le i\le 4\)
  and \(g_1,g_2\in \Gr\).
  We may exchange \(\tau_1\)
  and~\(\tau_2\)
  if necessary and hence assume that \(n_2\le n_4\).
  If there is no finite word~\(x\)
  with \(w_4 = w_2 x\),
  then the domains \(w_2\,A^\omega\)
  and \(w_4\,A^\omega\)
  of \(\alpha_{\tau_1}\)
  and~\(\alpha_{\tau_2}\)
  are disjoint and we are done.  So assume \(w_4 = w_2\,x\).
  Then the restriction of~\(\alpha_{\tau_1}\)
  to the domain of~\(\alpha_{\tau_2}\)
  is \(\alpha_{\tau_1'}\)
  with \(\tau_1' = (w_1\,g\cdot x,g|_x,w_2\,x)\).
  Thus
  \(\alpha_{\tau_1} \cap \alpha_{\tau_2} = \alpha_{\tau_1'} \cap
  \alpha_{\tau_2}\).
  So we may assume without loss of generality that \(w_2 = w_4\).
  Then \(w_1=w_3\)
  is necessary for
  \(\alpha_{\tau_1} \cap\alpha_{\tau_2}\neq\emptyset\)
  by our description of the equivalence relation defining~\(\Gr[H]\).
  Assume this as well.  If a point \((w_1,g_1,z,w_2)\)
  also lies in~\(\alpha_{w_1,g_2,w_2}\),
  then there is \(x\in A^k\)
  for some \(k\ge0\)
  with \(z = x\,z'\),
  \(g_1\cdot x= g_2\cdot x\),
  and \(g_1|_x = g_2|_x\).
  Thus \(\alpha_{w_1,g_1,w_2} \cap \alpha_{w_1,g_2,w_2}\)
  is the union of \(\alpha_{w_1\,g_1\cdot x,g_1|_x,w_2\,x}\)
  over all \(k\ge0\)
  and \(x\in A^k\)
  with \(g_1\cdot x = g_2\cdot x\)
  and \(g_1|_x = g_2|_x\).
  Thus~\(\langle \Gr,A\rangle\)
  is a wide inverse semigroup in~\(\Gr[H]\).
  Now \cite{Exel:Inverse_combinatorial}*{Proposition 5.4} implies
  \(\Gr[H] = \langle \Gr,A\rangle \ltimes A^\omega\).
\end{proof}

\longref{Proposition}{pro:self-similar_groupoid_wide_subsemigroup}
reproves the isomorphism between the groupoid model
\(\IS(F)\ltimes \Omega\) and the groupoid~\(\Gr[H]\) constructed in
\longref{Section}{sec:Ore_topological_corr} in the case of a
self-similar group.

The defining universal property of our groupoid model is also of
some interest.  It characterises the groupoid~\(\langle
\Gr,A\rangle\ltimes A^\omega\) by describing its actions on spaces.
We already described such actions rather explicitly in
\longref{Lemma}{lem:action_self-similar_group}.

\subsection{Groupoid models for self-similar graphs}
\label{sec:self-corr_groupoid}

Now we describe the groupoid model for a self-correspondence of a
transformation groupoid \(\Gr \defeq \Gamma\ltimes V\)
for a group~\(\Gamma\)
and a \(\Gamma\)\nb-set~\(V\),
equipped with the discrete topology.  The groupoid correspondences
\(\Gr\leftarrow \Gr\)
are described in
\cite{Antunes-Ko-Meyer:Groupoid_correspondences}*{Proposition~4.5}.
Namely, they come
from a \(\Gamma\)\nb-set~\(E\)
with a \(1\)\nb-cocycle
\(\Gamma\times E\to \Gamma\),
\((g,e)\mapsto g|_e\),
and two maps \(\rg,\s\colon E\rightrightarrows V\)
that satisfy \(\s(g\cdot e) = (g|_e)\cdot \s(e)\)
and \(\rg(g\cdot e) = g\cdot \rg(e)\)
for all \(g\in \Gamma\),
\(e\in E\).  The resulting groupoid correspondence
\begin{equation}
  \label{eq:self-similar_graph_correspondence}
  E\times \Gamma\colon \Gamma\ltimes V \leftarrow \Gamma\ltimes V
\end{equation}
has the underlying space \(E\times \Gamma\)
with the discrete topology, anchor maps
\[
\rg,\s\colon E\times \Gamma\rightrightarrows V,\qquad
\rg(e,g) = \rg(e),\quad
\s(e,g) = g^{-1}\cdot \s(e),
\]
and the left and right \(\Gamma\)\nb-actions
\(g_1\cdot (e,g_2)\cdot g_3 \defeq (g_1\cdot e,g_1|_e \cdot g_2\cdot
g_3)\) for all \(e\in E\), \(g,g_1,g_2,g_3\in \Gamma\).

Let \(\mathcal{P}^0(E)\defeq V\)
with the identical range and source maps, the given
\(\Gamma\)\nb-action,
and the trivial cocycle \(g|_v = g\)
for all \(v\in V\).
The self-correspondence of \(\Gamma\ltimes V\)
associated to this data is the identity correspondence.  For
\(n\ge1\), let
\begin{equation}
  \label{eq:self-similar_graph_pathn}
  \mathcal{P}^n(E) \defeq \setgiven{(e_1,\dotsc,e_n)\in E^n}
  {\s(e_i) = \rg(e_{i+1}) \text{ for }i=1,\dotsc,n-1}.
\end{equation}
Define the \(1\)\nb-cocycle
\(\Gamma\times \mathcal{P}^n(E) \to \Gamma\),
the \(\Gamma\)\nb-action
on~\(\mathcal{P}^n(E)\),
and \(\rg,\s\colon \mathcal{P}^n(E) \rightrightarrows V\)
recursively by
\(g|_{(e_1,e_2,\dotsc,e_n)} \defeq g|_{(e_1,e_2,\dotsc,
  e_{n-1})}|_{e_n}\) and
\begin{equation}
  \label{eq:group_action_on_path}
  g\cdot (e_1,\dotsc,e_n) \defeq
  (g\cdot e_1,g|_{e_1}\cdot e_2,\dotsc, g|_{e_1\dotsc e_{n-1}}\cdot e_n),
\end{equation}
\(\rg(e_1,\dotsc,e_n) \defeq \rg(e_1)\) and
\(\s(e_1,\dotsc,e_n) \defeq \s(e_n)\).  Induction gives
\(g\cdot (e_1,\dotsc,e_n) \in \mathcal{P}^n(E)\) if
\((e_1,\dotsc,e_n) \in \mathcal{P}^n(E)\).  This uses the properties
\(\s(g\cdot e) = (g|_e)\cdot \s(e)\) and
\(\rg(g\cdot e) = g\cdot \rg(e)\) of
\(\rg,\s\colon E\rightrightarrows V\); we do not need the equation
\(\s(g\cdot e) = g\cdot \s(e)\) that is assumed
in~\cite{Exel-Pardo:Self-similar}.  As in
\longref{Section}{sec:groupoid_self-similar}, the groupoid
correspondence
\[
\mathcal{P}^n(E)\times \Gamma\colon \Gamma\ltimes V \leftarrow \Gamma\ltimes V
\]
is isomorphic to the \(n\)\nb-fold
composite of \(E\times \Gamma\).  The multiplication maps
\[
(\mathcal{P}^n(E)\times \Gamma) \times (\mathcal{P}^m(E)\times \Gamma)
\to (\mathcal{P}^{n+m}(E)\times \Gamma)
\]
are given by a variant of~\eqref{eq:multiply_self-similarity_iterate},
see also \eqref{eq:group_action_on_path}.  The discussion above
describes the \((\N,+)\)-shaped
diagram associated to the groupoid
correspondence~\eqref{eq:self-similar_graph_correspondence}.

Now we describe the universal action of this diagram.  It takes place
on the projective limit~\(\mathcal{P}^\omega(E)\)
of the discrete spaces
\(\mathcal{P}^n(E) \cong (\mathcal{P}^n(E)\times \Gamma) \mathbin/
(\Gamma\ltimes V)\)
for \(n\in\N\), where the limit is taken for the truncation maps
\[
(e_1,\dotsc,e_{n+1})\mapsto (e_1,\dotsc,e_n).
\]
Thus \(\mathcal{P}^\omega(E)\) is the space of infinite paths,
\[
\mathcal{P}^\omega(E) =
\setgiven*{(e_n)_{n\in\N} \in \prod_{n\in\N} E}
{\s(e_n) = \rg(e_{n+1}) \text{ for all }n\in\N},
\]
equipped with the restriction of the product topology.  The left
action of~\(\Gamma\ltimes V\)
is described by the anchor map \((e_n)\mapsto \rg(e_0)\)
to~\(V\)
and by the action of~\(\Gamma\)
on infinite paths that generalises~\eqref{eq:group_action_on_path} for
finite paths in the obvious way.

If \(\rg\colon E\to V\)
is finite-to-one, then the set of all \(w\in \mathcal{P}^\omega(E)\)
with \(\rg(w) = v\)
for a fixed \(v\in V\)
is compact by Tychonov's Theorem.  And then the
space~\(\mathcal{P}^\omega(E)\)
is locally compact because these subspaces
of~\(\mathcal{P}^\omega(E)\)
for \(v\in V\)
form an open cover.  Conversely, if~\(\rg\)
is surjective, but not finite-to-one, then the
space~\(\mathcal{P}^\omega(E)\)
is not locally compact; we leave the proof to the reader.

The groupoid model for our diagram is
\(\IS(F)\ltimes \mathcal{P}^\omega(E)\)
by our description of the universal action~\(\Omega\)
and \longref{Proposition}{pro:groupoid_model_from_universal_F-action}.  The inverse
semigroup~\(\IS(F)\)
is not as simple now as in \longref{Section}{sec:groupoid_self-similar} because
\(\Gamma\ltimes V\)
has several objects.  There is, however, a more concrete inverse
semigroup \(\langle \Gamma,E\rangle\)
with
\(\langle \Gamma,E\rangle \ltimes \mathcal{P}^\omega(E) = \IS(F)\ltimes
\mathcal{P}^\omega(E)\);
this is the analogue of \(\langle \Gamma,A\rangle\)
for a self-similar group.  Let \(B\subseteq \Bis(F)\)
consist of the empty slice and all singleton slices, that is,
the elements of \(\bigsqcup_{n\ge0} \mathcal{P}^n(E)\times \Gamma\).
This is a subsemigroup because the product of two singleton slices
is either empty or again a singleton slice.  Explicitly, if
\(w_i\in\mathcal{P}^*(E) \defeq \bigsqcup_{n=0}^\infty
\mathcal{P}^n(E)\) and \(g_i\in \Gamma\) for \(i=1,2\), then
\[
(w_1,g_1)\cdot (w_2,g_2) =
\begin{cases}
  (w_1\,(g_1\cdot w_2),g_1|_{w_2}\cdot g_2)
  &\text{if }\s(w_1) = g_1\cdot \rg(w_2),\\
  \emptyset&\text{otherwise.}
\end{cases}
\]
Any slice is a disjoint union of slices in~\(B\).
Therefore, the proof that \(\IS(F)\ltimes \Omega\)
for a universal \(F\)\nb-action~\(\Omega\)
is a groupoid model
(\longref{Proposition}{pro:groupoid_model_from_universal_F-action}) still works if
we replace~\(\Bis(F)\) by the subsemigroup~\(B\) throughout.

Thus we replace~\(\IS(F)\)
be the universal inverse semigroup~\(\langle \Gamma,E\rangle\)
generated by the semigroup~\(B\) above with the relations of~\(B\) and
\begin{equation}
  \label{eq:self-similar_graph_GE_adjoint}
  (w_1,g_1)^* \cdot (w_2,g_2)
  = \langle (w_1,g_1),(w_2,g_2)\rangle
  \defeq \delta_{w_1,w_2} (g_1^{-1}\s(w_1), g_1^{-1} g_2)
\end{equation}
if \(g_1,g_2\in \Gamma\)
and \(w_1,w_2\in\mathcal{P}^n(E)\)
with the same \(n\in\N\),
where~\(\delta_{w_1,w_2}\)
means that the result is~\(\emptyset\)
if \(w_1\neq w_2\)
and \((g_1^{-1}\s(w_1), g_1^{-1} g_2)\)
if \(w_1= w_2\).
The element~\(\emptyset\)
remains a zero element in~\(\langle \Gamma,E\rangle\).

\begin{lemma}
  \label{lem:self-similar_graph_isg}
  Any element in~\(\langle \Gamma,E\rangle\)
  that is not the zero element~\(\emptyset\)
  may be written uniquely as
  \(w_1\,g\,w_2^* \defeq (w_1,g)\cdot (w_2,1)^*\)
  for some \(w_1,w_2\in\mathcal{P}^*(E)\),
  \(g\in \Gamma\) with \(g\cdot \s(w_2) = \s(w_1)\).
\end{lemma}

\begin{proof}
  It suffices to show that the subset of~\(\langle \Gamma,E\rangle\)
  consisting of \(\emptyset\) and the elements~\(w_1\,g\,w_2^*\) is
  closed under involution and products, so that it is an inverse
  semigroup as well.  Our proof is related to the proofs of
  \longref{Lemma}{lem:Ore_tight_simplify_groupoid} and
  \longref{Lemma}{lem:Ore_tight_simplify_groupoid2}.
  First we claim
  \((v,g)^* = (g^{-1}\cdot v,g^{-1})\)
  for all \(v\in V = \mathcal{P}^0(E)\).
  This follows because
  \((v,g)\cdot (g^{-1}\cdot v,g^{-1}) \cdot (v,g) = (v,g)\)
  and
  \((g^{-1}\cdot v,g^{-1}) \cdot (v,g) \cdot (g^{-1}\cdot v,g^{-1}) =
  (g^{-1}\cdot v,g^{-1})\).
  Since
  \(w_1 g w_2^* \defeq (w_1,g)\cdot (w_2,1)^* = (w_1,1)\cdot
  (\s(w_1),g)\cdot (w_2,1)^*\), this implies
  \begin{align*}
    (w_1\,g\,w_2^*)^* &\defeq (w_2,1)\cdot (\s(w_1),g)^* \cdot (w_1,1)^*
    = (w_2,1)\cdot (g^{-1}\cdot \s(w_1),g^{-1}) \cdot (w_1,1)^*
    \\ &= (w_2,1)\cdot (\s(w_2),g^{-1}) \cdot (w_1,1)^*
    \\ &= w_2\,g^{-1}\,w_1^*.
  \end{align*}
  This is of the desired form because \(\s(w_2) = g^{-1}\cdot
  \s(w_1)\).

  Now we prove that \(w_1 g_1 w_2^*\cdot w_3 g_2 w_4^*\)
  for \(w_1,w_2,w_3,w_4\in \mathcal{P}^*(E)\),
  \(g_1,g_2\in \Gamma\)
  with \(\s(w_1) = g_1\s(w_2)\)
  and \(\s(w_3) = g_2 \s(w_4)\)
  is empty or of the form \(w_5 g_3 w_6^*\)
  for some \(w_5,w_6\in\mathcal{P}^*(E)\),
  \(g_3\in \Gamma\)
  with \(\s(w_5) = g_3 \s(w_6)\).
  We distinguish the cases where~\(w_2\)
  is longer than~\(w_3\) or not.

  If~\(w_2\)
  is longer than~\(w_3\),
  then we write \(w_2 = w_2^1 w_2^2\)
  such that \(w_2^1,w_3\)
  have the same length.  Thus~\eqref{eq:self-similar_graph_GE_adjoint}
  gives
  \((w_2^1,1)^* (w_3,g_2) = \delta_{w_2^1,w_3} (\s(w_2^1),g_2)\).
  And \(\rg(w_2^1) = \rg(w_2)\),
  \(\s(w_2^1) = \rg(w_2^2)\),
  \(\s(w_2^2) = \s(w_2) = g_1^{-1}\s(w_1)\).  We compute
  \begin{align*}
    w_1 g_1 w_2^* w_3 g_2 w_4^*
    &= (w_1, g_1) (w_2^2,1)^* (w_2^1,1)^* (w_3, g_2) (w_4,1)^*
    \\ &= \delta_{w_2^1,w_3} (w_1, g_1) (w_2^2,1)^*(\s(w_2^1),g_2) (w_4,1)^*
    \\ &= \delta_{w_2^1,w_3} (w_1, g_1) (w_2^2,1)^*
    (g_2^{-1}\cdot \rg(w_2^2),g_2^{-1})^* (w_4,1)^*
    \\ &= \delta_{w_2^1,w_3} (w_1, g_1) (w_4\,(g_2^{-1}\cdot
    w_2^2),g_2^{-1}|_{w_2^2})^*
    \\ &= \delta_{w_2^1,w_3} w_1 (g_1 \cdot (g_2^{-1}|_{w_2^2})^{-1})
    (w_4\,(g_2^{-1}\cdot w_2^2))^*.
  \end{align*}
  This has the desired form because
  \(\rg(g_2^{-1}\cdot w_2^2) = g_2^{-1}\cdot \rg(w_2^2) =
  g_2^{-1}\cdot \s(w_2^1) = g_2^{-1}\cdot \s(w_3) = \s(w_4)\)
  and
  \(\s(w_1) = g_1 \s(w_2) = g_1 \s(w_2^1) = g_1 \s(w_3) = g_1 g_2
  \s(w_4)\) if \(w_2^1= w_3\).

  Now let~\(w_2\)
  be shorter than~\(w_3\).
  We split \(w_3 = w_3^1 w_3^2\),
  where~\(w_3^1\)
  has the same length as~\(w_2\).  As above, we compute
  \begin{align*}
    w_1 g_1 w_2^* w_3 g_2 w_4^*
    &= (w_1, g_1) (w_2,1)^* (w_3^1,1) (w_3^2, g_2) (w_4,1)^*
    \\&= \delta_{w_2,w_3^1} (w_1, g_1) (w_3^2,g_2) (w_4,1)^*
    \\&= \delta_{w_2,w_3^1} (w_1\,(g_1\cdot w_3^2), g_1|_{w_3^2}\cdot
    g_2) (w_4,1)^*
    \\&= \delta_{w_2,w_3^1} (w_1\,(g_1\cdot w_3^2))\, (g_1|_{w_3^2}\cdot
    g_2)\,w_4.\qedhere
  \end{align*}
\end{proof}

We let
\((w_1,g,w_2) \defeq w_1\, g\, w_2^{-1} \in\langle \Gamma,E\rangle\).
The proof of \longref{Lemma}{lem:self-similar_graph_isg} describes the
multiplication in~\(\langle \Gamma,E\rangle\):
\begin{equation}
  \label{eq:mult_self-similar_graph}
  (w_1, g_1, w_2)\cdot (w_3, g_2, w_4) =
  \begin{cases}
    \bigl(w_1\,(g_1\cdot x), g_1|_{x}\cdot g_2,w_4\bigr)&
    \text{if }w_3 = w_2 x,\\
    \bigl(w_1, g_1 (g_2^{-1}|_x)^{-1}, w_4\,(g_2^{-1}\cdot x)\bigr)&
    \text{if }w_2 = w_3 x,\\
    \emptyset&\text{otherwise,}
  \end{cases}
\end{equation}
with some \(x\in\mathcal{P}^*(E)\).
The involution is simply \((w_1,g,w_2)^* = (w_2,g^{-1},w_1)\).
Thus \(\langle \Gamma,E\rangle\)
is isomorphic to the inverse semigroup~\(\mathcal{S}_{\Gamma,E}\)
defined in \cite{Exel-Pardo:Self-similar}*{Definition 4.1}.

The generators of the inverse semigroup~\(\langle \Gamma,E\rangle\)
act on~\(\mathcal{P}^\omega(E)\) by the group action of~\(\Gamma\)
described above and by concatenating paths:
\(\vartheta_w(x) = w\,x\) is defined if \(\s(w) = \rg(x)\),
and~\(\vartheta_w\) is a homeomorphism from the clopen subset of
paths beginning in~\(\s(w)\) onto the clopen subset of words that
begin with the path~\(w\).  This extends to an action
of~\(\langle \Gamma,E\rangle\), where \((w_1,g,w_2)\) acts by the
homeomorphism
\(w_2\,\mathcal{P}^\omega(E) \congto w_1\,\mathcal{P}^\omega(E)\)
that first cuts away the leading path~\(w_2\), then acts by
\(g\in \Gamma\), and finally puts the path~\(w_1\) in front.  The
arguments in \cite{Exel-Pardo:Self-similar}*{Section~8} show that,
under the assumptions made there, our groupoid model is isomorphic
to the tight groupoid of the inverse semigroup
\(\langle \Gamma,E\rangle = \mathcal{S}_{\Gamma,E}\).  The
\(\Cst\)\nb-algebra~\(\mathcal{O}_{\Gamma,E}\) is isomorphic to the
groupoid \(\Cst\)\nb-algebra of this groupoid by
\cite{Exel-Pardo:Self-similar}*{Corollary~6.4}.

Now let~\(\Bisp\) be a groupoid correspondence on an arbitrary
discrete groupoid~\(\Gr\) instead of a discrete transformation group
\(\Gamma\ltimes V\).  The groupoid model may be described as above,
with only notational differences.  Let \(V\defeq \Gr^0\) and
\(E\defeq\Bisp/\Gr\) with the canonical maps
\(\rg,\s\colon E\rightrightarrows V\) as in
\cite{Antunes-Ko-Meyer:Groupoid_correspondences}*{Proposition~4.5}.
This defines a
directed graph.  We identify~\(\Bisp^{\Grcomp n}/\Gr\) with the
set~\(\mathcal{P}^n(E)\) of paths of length~\(n\) for all \(n\in\N\)
as above.  Hence the universal action of the groupoid correspondence
takes place on the space~\(\mathcal{P}^\omega(E)\) of infinite paths
in~\(E\) with the projective limit topology.  We may again replace
the semigroup of all slices of \(\Bisp^n\), \(n\in\N\), by the
semigroup of singleton or empty slices.  The resulting
\emph{inverse} semigroup with zero element that encodes actions of
the groupoid correspondence is the same one found in
\cite{Laca-Raeburn-Ramagge-Whittaker:Equilibrium_self-similar_groupoid}*{Proposition~4.5
  and Lemma~4.6}.  In particular, its non-zero elements are of the
form \((\mu,g,\nu)\) with finite paths \(\mu,\nu\in E^*\) and
\(g\in\Gr\) satisfying \(\s(\mu) = \rg(g)\), \(\s(g) = \rg(\nu)\),
and the multiplication is given by the same
formula~\eqref{eq:mult_self-similar_graph} as above.  This is also
the same formula as in
\cite{Laca-Raeburn-Ramagge-Whittaker:Equilibrium_self-similar_groupoid}*{Lemma~4.6}
because the cocycle condition implies
\(g^{-1}|_{g\cdot x} = (g|_x)^{-1}\) for all \(g\in\Gr\),
\(x\in E^*\) with \(\s(g) = \rg(x)\).

\section{Groupoid models as bicategorical limits}
\label{sec:models_vs_limits}

The concepts of limit and colimit of a diagram in a category have
bicategorical analogues.  In this section, we show that the groupoid
model of a diagram in~\(\Grcat\) is a limit of this diagram.  This
is an easy consequence of the universal property that defines the
groupoid model.  The interpretation of the groupoid model as a limit
has two nice consequences.  First, it explains the functoriality of
the groupoid model construction because the bicategorical limit
construction --~if it exists~-- is automatically a homomorphism of
bicategories.  Secondly, (absolute) Cuntz--Pimsner algebras of
proper product systems are described analogously
in~\cite{Albandik-Meyer:Colimits}.  Namely, the product system may
be viewed as a diagram in a certain bicategory of
\(\Cst\)\nb-correspondences, and the Cuntz--Pimsner algebra is shown
in~\cite{Albandik-Meyer:Colimits} to be its colimit.  When the
conventions regarding the \(\Cst\)\nb-correspondence bicategory are
changed as in~\cite{Antunes-Ko-Meyer:Groupoid_correspondences}, then
this colimit becomes a limit.  What we call a limit is often called a
bilimit.  This is necessary in \(2\)\nb-categories to avoid
confusion with the ordinary limits in the sense of category theory.
We may simply speak of limits because no confusion is possible in a
bicategory without an underlying category.

We first make limits of diagrams in bicategories more concrete
(see~\cite{Albandik-Meyer:Colimits}).  Let \(\Cat\) and~\(\Cat[D]\)
be bicategories and let \(F\colon \Cat\to\Cat[D]\) be a
homomorphism.

\begin{definition}
  \label{def:const_d}
  Let~\(d\) be an object of~\(\Cat[D]\).  The \emph{constant
    diagram} \(\const_d\colon \Cat\to\Cat[D]\) maps all objects
  of~\(\Cat\) to~\(d\), all arrows in~\(\Cat\) to the identity
  correspondence on~\(d\), and all pairs of composable arrows to the
  canonical invertible \(2\)\nb-arrow
  \(\id_d\circ \id_d\Rightarrow \id_d\).
\end{definition}

For an object~\(d\) of~\(\Cat[D]\), let \(\Cat[D]^{\Cat}(d,F)\) be
the category with strong transformations from the constant diagram
\(\const_d\colon \Cat\to\Cat[D]\) to~\(F\) as objects and
modifications between them as arrows.  The notions of strong
transformations and modifications are explained
in~\cite{Leinster:Basic_Bicategories} (see also
\longref{Proposition}{pro:trafo_category-diagram} and
\longref{Proposition}{pro:modification_category-diagram}).  Strong
transformations and transformations only differ if~\(\Cat[D]\) has
\(2\)\nb-arrows that are not invertible.  Then it is crucial to
restrict attention to strong transformations.  For two objects
\(d,x\in\Cat[D]\), let \(\Cat[D](d,x)\) be the category with arrows
\(d\to x\) as objects and \(2\)\nb-arrows among them as arrows.  An
arrow \(g\colon d_1 \to d_2\) induces functors
\(g^*\colon \Cat[D](d_2,x)\to\Cat[D](d_1,x)\) by composition
with~\(g\) and horizontal composition with the unit \(2\)\nb-arrow
\(1_g\colon g\Rightarrow g\).  It also induces a functor
\[
  g^*\colon \Cat[D]^{\Cat}(\const(d_2),F)\to\Cat[D]^{\Cat}(\const(d_1),F).
\]
A (strong) transformation \(\const(d_1) \to F\) consists of arrows
\(\lambda_x\colon d_1 \to F(x)\) for all \(x\in\Cat\) and invertible
\(2\)\nb-arrows
\(\lambda_\varphi\colon F(\varphi) \circ \lambda_x \Rightarrow
\lambda_y \circ 1_{d_1}\) for all \(x,y\in \Cat^0\) and
\(\varphi\in\Cat(x,y)\).  The functor~\(g^*\) composes
each~\(\lambda_x\) with \(g\colon d_1 \to d_2\) and produces
\(2\)\nb-arrows
\(F(\varphi) \circ (\lambda_x\circ g) \Rightarrow (\lambda_y\circ g)
\circ 1_{d_1}\) from the \(2\)\nb-arrows~\(\lambda_\varphi\) in the
canonical way.

A \emph{cone over~\(F\) with summit~\(d\)} for an object~\(d\)
of~\(\Cat[D]\) is a strong transformation
\(\vartheta\colon\const_d\to F\).  Given such a cone and another
object \(d'\in \Cat[D]\), there is a
functor
\[
  \vartheta_*\colon \Cat[D](d',d)\to \Cat[D]^{\Cat}(\const_{d'},F);
\]
it maps an arrow \(g\colon d'\to d\) to the strong
transformation \(g^*(\vartheta)\) with~\(g^*\) defined above, and a
\(2\)\nb-arrow \(\alpha\colon g_1\Rightarrow g_2\) to the canonical
modification \(g_1^*(\vartheta)\Rightarrow g_2^*(\vartheta)\)
defined by taking horizontal products with~\(\alpha\).
A \emph{limit cone} is a cone
\(\vartheta\colon \const(\lim F) \to F\) such that for each object
\(d\in\Cat[D]^0\), the functor
\[
  \vartheta_*\colon \Cat[D](d,\lim F)\to \Cat[D]^{\Cat}(\const_d,F)
\]
is an equivalence of categories.

We may embed~\(\Cat[D]\) into~\(\Cat[D]^{\Cat}\) by mapping objects,
arrows, and \(2\)\nb-arrows to constant diagrams, constant
transformations and constant modifications, respectively.  This
realises~\(\Cat[D]\) as a subbicategory of~\(\Cat[D]^{\Cat}\).  The
definition of the limit of~\(F\) says exactly that
\(\vartheta\colon \const(\lim F) \to F\) is a \emph{biuniversal
  arrow} (see \cite{Fiore:Pseudo_biadjoints}*{Definition~9.4})
from~\(F\) to an object of this subbicategory.  More precisely,
Fiore~\cite{Fiore:Pseudo_biadjoints} assumes~\(\Cat[D]\) to be
strict, that is, all associators and unit transformations are
identities.  The Coherence Theorem
(see~\cite{Leinster:Basic_Bicategories}) says that every small
bicategory is biequivalent to a strict one.  Since any set of
diagrams in~\(\Cat[D]\) only sees a small subcategory
of~\(\Cat[D]\), the results in~\cite{Fiore:Pseudo_biadjoints} carry
over to general bicategories.  We are particularly interested in the
following theorem, which says that a biadjunction exists once
biuniversal arrows exist:

\begin{theorem}[\cite{Fiore:Pseudo_biadjoints}*{Theorem~9.17}]
  \label{the:biadjoint}
  Let \(\Cat\) and~\(\Cat[D]\) be bicategories.  Let~\(\Cat\) be
  small.  Assume that each \(\Cat\)\nb-shaped diagram in~\(\Cat[D]\)
  has a limit.  Then there is a homomorphism of bicategories
  \(\lim\colon \Cat[D]^{\Cat}\to \Cat[D]\) which on objects maps a
  diagram to its limit.  This homomorphism is biadjoint to the
  inclusion \(\const\colon \Cat[D]\hookrightarrow \Cat[D]^{\Cat}\).
  The canonical transformations \(F\to \lim F\) form a strong
  transformation from the identity homomorphism on~\(\Cat[D]\) to
  the homomorphism \({\lim}\circ \const\).  The composition
  \(\const\circ {\lim}\) is isomorphic to the identity homomorphism
  on~\(\Cat[D]\).
\end{theorem}

Now we specialise to the target bicategory \(\Cat[D]=\Grcat\).  We
first describe the strong transformations to constant diagrams (that
is, cones over diagrams) and the modifications among them.  A strong
transformation from the constant diagram \(\const_{\Gr[H]}\) for a
groupoid~\(\Gr[H]\) to the diagram \(F=(\Gr_x,\Bisp_g,\mu_{g,h})\)
is equivalent to groupoid correspondences
\(\Bisp[Y]_x\colon \Gr_x\leftarrow\Gr[H]\) for all objects~\(x\)
of~\(\Cat\) and isomorphisms of groupoid correspondences
\[
V_g\colon \Bisp_g \Grcomp_{\Gr_y} \Bisp[Y]_y \congto \Bisp[Y]_x
\qquad\text{for all arrows }g\colon x\leftarrow y\text{ in }\Cat,
\]
such that~\(V_x\) is the canonical isomorphism for each object~\(x\)
and the diagrams
\begin{equation}
  \label{eq:cone_category-diagram}
  \begin{tikzpicture}[yscale=1.3,xscale=2.5,baseline=(current bounding
    box.west)]
    \node (cbb) at (144:1) {\(\Bisp[Y]_x\)};
    \node (cb) at (216:1) {\(\Bisp[Y]_x\)};
    \node (acb) at (72:1) {\(\Bisp_g\Grcomp_{\Gr_y} \Bisp[Y]_y\)};
    \node (ac) at (288:1) {\(\Bisp_{g h}\Grcomp_{\Gr_z} \Bisp[Y]_z\)};
    \node (aac) at (0:.8) {\(\Bisp_g\Grcomp_{\Gr_y}
      \Bisp_h\Grcomp_{\Gr_z}
      \Bisp[Y]_z\)};
    \draw[double,double equal sign distance] (cbb) -- node[swap]
    {\(\scriptstyle \id_{\Bisp[Y]_x}\)} (cb);
    \draw[dar] (acb) -- node[near start,swap] {\(\scriptstyle V_g\)} (cbb);
    \draw[dar] (ac) -- node[near start] {\(\scriptstyle V_{g h}\)} (cb);
    \draw[dar] (aac.north) -- node[swap]
    {\(\scriptstyle \id_{\Bisp_g}\Grcomp_{\Gr_y} V_h\)} (acb.south);
    \draw[dar] (aac.south) -- node {\(\scriptstyle \mu_{g,h} \Grcomp_{\Gr_z}
      \id_{\Bisp[Y]_z}\)} (ac.north);
  \end{tikzpicture}
\end{equation}
for composable arrows \(g\colon x\leftarrow y\),
\(h\colon y\leftarrow z\) in~\(\Cat\) commute.  Compared to
\longref{Proposition}{pro:trafo_category-diagram}, we have used the
unit isomorphisms
\(\Bisp[Y]_x \Grcomp_{\Gr_x} (\const_{\Gr[H]})_x \cong \Bisp[Y]_x\)
throughout to simplify the data.  This
turns~\eqref{eq:trafo_category-diagram}
into~\eqref{eq:cone_category-diagram}.  As in
\longref{Proposition}{pro:trafo_category-diagram}, the
diagram~\eqref{eq:cone_category-diagram} commutes automatically if
\(g\) or~\(h\) is an identity arrow.

Let \((\Bisp[Y]_x^1,V_g^1)\) and \((\Bisp[Y]_x^2,V_g^2)\) be two
such strong transformations.  A modification between them as
described in
\longref{Proposition}{pro:modification_category-diagram} is
equivalent to a family of continuous maps of correspondences
\[
  W_x\colon \Bisp[Y]_x^1\to \Bisp[Y]_x^2
\]
for all objects~\(x\) of~\(\Cat\), such that the diagrams
\begin{equation}
  \label{eq:cone_modification_category-diagram}
  \begin{tikzpicture}[yscale=1,xscale=4,baseline=(current bounding
    box.west)]
    \node (add) at (0,1) {\(\Bisp[Y]^1_x\)};
    \node (ad) at (1,1) {\(\Bisp[Y]^2_x\)};
    \node (ddc) at (0,0) {\(\Bisp_g\Grcomp_{\Gr_y} \Bisp[Y]^1_y\)};
    \node (dc) at (1,0) {\(\Bisp_g\Grcomp_{\Gr_y} \Bisp[Y]^2_y\)};
    
    \draw[dar] (add) -- node {\(\scriptstyle W_x\)} (ad);
    \draw[dar] (ddc) -- node {\(\scriptstyle V^1_g\)} (add);
    \draw[dar] (dc) -- node {\(\scriptstyle V^2_g\)} (ad);
    \draw[dar] (ddc) -- node {\(\scriptstyle \id_{\Bisp_g}\Grcomp_{\Gr_y} W_y\)}
    (dc);
  \end{tikzpicture}
\end{equation}
commute for all arrows \(g\colon x\leftarrow y\) in~\(\Cat\).  And
this diagram commutes automatically if~\(g\) is an identity arrow.
As above, the unit isomorphisms
\(\Bisp[Y]_x \Grcomp_{\Gr_x} (\const_{\Gr[H]})_x \cong \Bisp[Y]_x\)
allow to simplify the situation in
\longref{Proposition}{pro:modification_category-diagram}.

Now we prove that the groupoid model of a diagram is a limit as
well.  We first extend the defining universal property of a groupoid
model so that it describes correspondences from a groupoid model to
other groupoids:

\begin{lemma}
  \label{lem:model_is_limit}
  Let \(F=(\Gr_x,\Bisp_g,\mu_{g,h})\) be a diagram of groupoid
  correspondences.  Let~\(\Gr[U]\) be a groupoid model for it, and
  let~\(\Gr[H]\) be another groupoid.  Let~\(\Bisp[Y]\) be a space
  with a right \(\Gr[H]\)\nb-action and with a left
  \(\Gr[U]\)\nb-action.  The latter corresponds to an
  \(F\)\nb-action on~\(\Bisp[Y]\), which is given by a clopen
  decomposition \(\Bisp[Y] = \bigsqcup_{x\in\Cat^0} \Bisp[Y]_x\),
  groupoid actions of~\(\Gr_x\) on~\(\Bisp[Y]_x\) for
  \(x\in\Cat^0\), and homeomorphisms
  \(\alpha_g\colon \Bisp_g \Grcomp_{\Gr_{\s(g)}} \Bisp[Y]_{\s(g)}
  \congto \Bisp[Y]_{\rg(g)}\) for \(g\in \Cat\), subject to certain
  conditions.  The left \(\Gr[U]\)\nb-action and the right
  \(\Gr[H]\)\nb-action on~\(\Bisp[Y]\) commute if and only if
  \begin{enumerate}
  \item each subset \(\Bisp[Y]_x\subseteq \Bisp[Y]\) is
    \(\Gr[H]\)\nb-invariant,
  \item the actions of \(\Gr_x\) and~\(\Gr[H]\) on~\(\Bisp[Y]_x\)
    commute for all \(x\in\Cat^0\), and
  \item the maps~\(\alpha_g\) are \(\Gr[H]\)\nb-equivariant for all
    \(g\in\Cat\).
  \end{enumerate}
\end{lemma}

\begin{proof}
  The map \(\s_{\Bisp[Y]} = \s\colon \Bisp[Y] \to \Gr[H]^0\) is
  \(\Gr[U]\)\nb-invariant if and only if it is \(F\)\nb-invariant by
  \longref{Lemma}{lem:groupoid_model_invariant_map}.  This happens
  if and only if its restriction to~\(\Bisp[Y]_x\) is
  \(\Gr_x\)\nb-invariant and compatible with the maps~\(\alpha_g\)
  for all \(g\in \Cat\).  This is necessary both for the actions of
  \(\Gr[U]\) and~\(\Gr[H]\) to commute and for the conditions in the
  theorem.  So we may assume without loss of generality
  that~\(\s_{\Bisp[Y]}\) is \(F\)\nb-invariant.
	
  Let \(x\in\Gr[H]^0\).  The subset
  \(\s_{\Bisp[Y]}^{-1}(x)\subseteq \Bisp[Y]\) is \(F\)\nb-invariant
  because~\(\s_{\Bisp[Y]}\) is \(F\)\nb-invariant.  So there is a
  unique restricted \(F\)\nb-action on~\(\s_{\Bisp[Y]}^{-1}(x)\) for
  which the inclusion
  \(\s_{\Bisp[Y]}^{-1}(x)\hookrightarrow \Bisp[Y]\) is
  \(F\)\nb-equivariant.  It corresponds to a unique
  \(\Gr[U]\)\nb-action for which the inclusion
  \(\s_{\Bisp[Y]}^{-1}(x)\hookrightarrow \Bisp[Y]\) is
  \(\Gr[U]\)\nb-equivariant.  An arrow \(h\in\Gr[H]\) acts
  on~\(\Bisp[Y]\) by a homeomorphism
  \(\s_{\Bisp[Y]}^{-1}(\rg(h)) \to \s_{\Bisp[Y]}^{-1}(\s(h))\),
  \(y\mapsto y\cdot h\).  The left \(\Gr[U]\)\nb-action
  on~\(\Bisp[Y]\) commutes with the right \(\Gr[H]\)\nb-action if
  and only if all these homeomorphisms are \(\Gr[U]\)\nb-equivariant
  for the induced \(\Gr[U]\)\nb-actions on
  \(\s_{\Bisp[Y]}^{-1}(\rg(h))\) and \(\s_{\Bisp[Y]}^{-1}(\s(h))\).
  By the universal property of the groupoid model, this happens if
  and only if these homeomorphisms are \(F\)\nb-equivariant.  And
  this is equivalent to the three conditions in the statement of
  the lemma.
\end{proof}

A space~\(\Bisp[Y]\) with a left action of \(\Gr[U]\) and a right
action of~\(\Gr[H]\) is a correspondence \(\Gr[U]\leftarrow\Gr[H]\)
if and only if the actions of \(\Gr[U]\) and~\(\Gr[H]\) commute, the
action of~\(\Gr[H]\) is basic, and its anchor map
\(\Bisp[Y]\to\Gr[H]^0\) is a local homeomorphism.  The latter two
properties do not involve the left action at all.  Thus
\longref{Lemma}{lem:model_is_limit} describes correspondences
\(\Gr[U] \leftarrow \Gr[H]\) using \(F\)\nb-actions.

\begin{theorem}
  \label{the:groupoid_model_limit}
  A groupoid model~\(\Gr[U]\) for a diagram
  \(F\colon \Cat\to\Grcat\) is also a limit for~\(F\) in~\(\Grcat\).
\end{theorem}

\begin{proof}
  Let \(F \leftarrow \const_{\Gr[H]}\) for a groupoid~\(\Gr[H]\) be
  a cone over the groupoid diagram~\(F\).  This is given by groupoid
  correspondences \(\Bisp[Y]_x\colon \Gr_x\leftarrow\Gr[H]\) for all
  objects~\(x\) of~\(\Cat\) and isomorphisms of groupoid
  correspondences
  \(V_g\colon \Bisp_g \Grcomp_{\Gr_y} \Bisp[Y]_y \congto
  \Bisp[Y]_x\) for all arrows
  \(g\colon x\leftarrow y\text{ in }\Cat\), such that~\(V_x\) is the
  canonical isomorphism for each object~\(x\) and the
  diagrams~\eqref{eq:cone_category-diagram} commute for composable
  arrows \(g\colon x\leftarrow y\), \(h\colon y\leftarrow z\)
  in~\(\Cat\).  Let
  \(\Bisp[Y] \defeq \bigsqcup_{x\in\Cat^0} \Bisp[Y]_x\), equipped
  with its induced right \(\Gr[H]\)\nb-action.  The left actions of
  the groupoids~\(\Gr_x\) and the maps~\(V_g\) together are the same
  as an action of the diagram~\(F\) on~\(\Bisp[Y]\) that commutes
  with the right \(\Gr[H]\)\nb-action.  By
  \longref{Lemma}{lem:model_is_limit}, this is equivalent to an
  action of the groupoid model~\(\Gr[U]\) on~\(\Bisp[Y]\) that
  commutes with the right \(\Gr[H]\)\nb-action.  A modification from
  \((\Bisp[Y]_x,V_g)\) to \((\Bisp[Y]'_x,V'_g)\) is the same as a
  family of \(\Gr_x,\Gr[H]\)\nb-equivariant maps
  \(\Bisp[Y]_x \to \Bisp[Y]'_x\) that also intertwine the maps
  \(V_g\) and~\(V_g'\).  This is equivalent to a
  \(\Gr[U],\Gr[H]\)-equivariant map
  \(\bigsqcup_{x\in \Cat^0} \Bisp[Y]_x \to \bigsqcup_{x\in \Cat^0}
  \Bisp[Y]'_x\).  Thus the category of cones
  \(F\leftarrow \const_{\Gr[H]}\) is equivalent to the category of
  correspondences \(\Gr[U]\leftarrow \Gr[H]\) and \(2\)\nb-arrows
  between them.  This equivalence is natural when we
  modify~\(\Gr[H]\).  Hence the groupoid model~\(\Gr[U]\) is a limit
  of the diagram~\(F\)
\end{proof}

We will show in~\cite{Ko-Meyer:Groupoid_models} that a groupoid
model always exists.

The converse of \longref{Theorem}{the:groupoid_model_limit} fails
because limits in a bicategory are unique only up to equivalence,
whereas groupoid models are unique up to isomorphism by
\longref{Proposition}{pro:groupoid_model}.  The equivalences
in~\(\Grcat\) are the Morita equivalences by
\longref{Theorem}{the:groupoid_equivalence}.

\begin{corollary}
  \label{cor:groupoid_model_functor}
  Assume that any diagram \(\Cat\to\Grcat\) has a groupoid model.
  Then the construction of groupoid models is part of a homomorphism
  of bicategories \(\Grcat^{\Cat} \to \Grcat\).
\end{corollary}

\begin{proof}
  This follows from \longref{Theorem}{the:biadjoint} because the
  groupoid model of a diagram is also a limit for it by
  \longref{Theorem}{the:groupoid_model_limit}.
\end{proof}

Let \(F_1=(\Gr_x,\Bisp_g,\mu_{g,h})\) and
\(F_2=(\Gr[H]_x,\Bisp[Y]_g,\nu_{g,h})\) be \(\Cat\)\nb-shaped
diagrams.  Let \(\Gr[U]\) and~\(\Gr[V]\) be groupoid models for
\(F_1\) and~\(F_2\), respectively.  A strong transformation
\((\Bisp[Z]_x,V_g) \colon F_1 \leftarrow F_2\) from~\(F_2\)
to~\(F_1\) induces a groupoid correspondence
\(\Gr[V] \leftarrow \Gr[U]\) by
\longref{Corollary}{cor:groupoid_model_functor}.  It is easy to see
that the following explicit construction gives this correspondence.

First, the left action of~\(\Gr[V]\) on its own arrow space is
equivalent to an action of the diagram~\(F_2\) on~\(\Gr[V]\) by the
universal property of the groupoid model~\(\Gr[V]\).  This action commutes with
the action of~\(\Gr[V]\) by right multiplication.  The proof of
\longref{Theorem}{the:groupoid_model_limit} shows that the
\(F_2\)\nb-action on~\(\Gr[V]\) is equivalent to a strong
transformation \(F_2 \leftarrow \const_{\Gr[V]}\).  We compose this
with the given strong transformation \(F_1 \leftarrow F_2\) to get a
strong transformation \(F_1 \leftarrow \const_{\Gr[V]}\).  Denote
its underlying groupoid correspondences \(\Gr_x \leftarrow \Gr[V]\)
by~\(\Bisp[T]_x\).  Then the strong transformation is equivalent to
an \(F_1\)\nb-action on
\(\Bisp[T] \defeq \bigsqcup_{x\in \Cat^0} \Bisp[T]_x\) that commutes
with the right \(\Gr[V]\)\nb-action.  And this is equivalent to a
left action of~\(\Gr[U]\) on~\(\Bisp[T]\) that commutes with the
right \(\Gr[V]\)\nb-action.  So~\(\Bisp[T]\) becomes a groupoid
correspondence \(\Gr[U]\leftarrow \Gr[V]\).

\end{document}